\documentclass[11pt,reqno]{amsart}
\usepackage{amsmath, amssymb, amsthm}
\usepackage{url}
\usepackage[breaklinks]{hyperref}
\usepackage[czech, english]{babel}
 \renewcommand{\a}{\alpha}
\renewcommand{\b}{\beta}

\newcommand{\g}{\gamma}
\newcommand{\G}{\Gamma}
\renewcommand{\l}{\lambda}

\renewcommand{\(}{\left\(}
\renewcommand{\)}{\right\)}
\renewcommand{\[}{\left\[}
\renewcommand{\]}{\right\]}

\numberwithin{equation}{section}
 \theoremstyle{plain}
\newtheorem{theorem}{Theorem}[section]
\newtheorem{lemma}[theorem]{Lemma}
\newtheorem{remark}[]{Remark}

\newtheorem{corollary}[theorem]{Corollary}

\newcommand{\Ppsum}{\mathop{{\;\,{\sum}^{\prime \prime}}}}

   \makeatletter
\def\proof{\@ifnextchar[{\@oproof}{\@nproof}}
\def\@oproof[#1][#2]{\trivlist\item[\hskip\labelsep\textit{#2 Proof of\
#1.}~]\ignorespaces}
\def\@nproof{\trivlist\item[\hskip\labelsep\textit{Proof.}~]\ignorespaces}
\makeatother

\begin{document}
\title[Generalized Lambert series and Raabe's integral]{Generalized Lambert series, Raabe's integral and a two-parameter generalization of Ramanujan's formula for $\zeta(2m+1)$} 
%{\Large }

\author{Atul Dixit, Rajat Gupta, Rahul Kumar and Bibekananda Maji}
\address{Discipline of Mathematics, Indian Institute of Technology Gandhinagar, Palaj, Gandhinagar 382355, Gujarat, India} 
\email{adixit@iitgn.ac.in, rajat\_gupta@iitgn.ac.in, rahul.kumr@iitgn.ac.in,\newline bibekananda.maji@iitgn.ac.in}
\begin{abstract}
A comprehensive study of the generalized Lambert series $\displaystyle\sum_{n=1}^{\infty}\frac{n^{N-2h}\textup{exp}(-an^{N}x)}{1-\textup{exp}(-n^{N}x)},\newline 0<a\leq 1,\ x>0$, $N\in\mathbb{N}$ and $h\in\mathbb{Z}$, is undertaken. Two of the general transformations of this series that we obtain here lead to two-parameter generalizations of Ramanujan's famous formula for $\zeta(2m+1)$, $m>0$ and the transformation formula for $\log\eta(z)$. Numerous important special cases of our transformations are derived. An identity relating $\zeta(2N+1), \zeta(4N+1),\cdots, \zeta(2Nm+1)$ is obtained for $N$ odd and $m\in\mathbb{N}$. Certain transcendence results of Zudilin- and Rivoal-type are obtained for odd zeta values and generalized Lambert series. A criterion for transcendence of $\zeta(2m+1)$ and a Zudilin-type result on irrationality of Euler's constant $\g$ are also given. New results analogous to those of Ramanujan and Klusch for $N$ even, and a transcendence result involving $\zeta\left(2m+1-\frac{1}{N}\right)$, are obtained. 
\end{abstract}
\maketitle
\section{Introduction}\label{intro}
 In his address to the American Mathematical Society on September 5, 1941 \cite{rademacheraddress}, Hans Rademacher writes ``\dots the impression may have prevailed that analytic number theory deals foremost with asymptotic expressions for arithmetical functions. This view, however, overlooks another side of analytic number theory, which I may indicate by the words ``identities'', ``group-theoretical arguments'', ``structural considerations''. This line of research is not yet so widely known ; it may very well be that methods of its type will lead to the ``deeper'' results, will reveal the sources of some of the results of the first direction of approach.''

Indeed, the developments that have taken place, since Rademacher's time, in the theory of partitions, theory of modular forms, mock modular forms and harmonic Maass forms \cite{hmfmmf}, to name a few, prove that his assessment of the impact of this other side of analytic number theory was correct. In the present paper, we offer the reader new examples further corroborating Rademacher's claim, namely, we derive some identities which lead to important results on transcendence of certain values and, at the same time, hint connections with the modular world. 

In \cite[Theorem 1.1]{dixitmaji1}, a transformation of the series $\displaystyle\sum_{n=1}^{\infty}\frac{n^{N-2h}}{e^{n^{N}x}-1}$ was obtained for any positive integer $N$ and any integer $h$. Ramanujan, by the way, explicitly wrote down this exact same series on page $332$ of his Lost Notebook \cite{lnb} but he did not give any transformation for it. Kanemitsu, Tanigawa and Yoshimoto \cite{ktyhr} were the first to obtain a transformation of this series, however, they considered the case $0<h\leq N/2$ only. In fact, in \cite[Theorem 1.1]{dixitmaji1}, it was observed that working out the transformation in the remaining two cases, that is $h>N/2$ and $h\leq 0$, in the case when $N$ is an odd positive integer, enables us to decode valuable information in that when $N=1$, together they give, as a special case, Ramanujan's following famous formula for $\zeta(2m+1), m\neq 0$ \cite[p.~173, Ch. 14, Entry 21(i)]{ramnote}, \cite[pp.~319-320, formula (28)]{lnb}, \cite[pp.~275-276]{bcbramsecnote}:

 For $\a, \b>0$ with $\a\b=\pi^2$ and $m\in\mathbb{Z}, m\neq 0$,
\begin{align}\label{zetaodd}
\a^{-m}\left\{\frac{1}{2}\zeta(2m+1)+\sum_{n=1}^{\infty}\frac{n^{-2m-1}}{e^{2\a n}-1}\right\}&=(-\b)^{-m}\left\{\frac{1}{2}\zeta(2m+1)+\sum_{n=1}^{\infty}\frac{n^{-2m-1}}{e^{2\b n}-1}\right\}\nonumber\\
&\quad-2^{2m}\sum_{j=0}^{m+1}\frac{(-1)^jB_{2j}B_{2m+2-2j}}{(2j)!(2m+2-2j)!}\a^{m+1-j}\b^j,
\end{align}
where for $j\geq 0$, $B_{j}$ is the Bernoulli number, which is the special case $a=1$ of the Bernoulli polynomial $B_{j}(a)$ defined by $\sum_{j=0}^{\infty}\frac{B_j(a) z^j}{j!}=\frac{ze^{az}}{e^{z}-1}$, $0<a\leq 1, |z|<2\pi$. (For references in the literature on Ramanujan's formula, we refer the reader to a recent paper \cite{berndtstraubzeta}.)

Not only this, when $N\geq 1$ is an odd positive integer, the aforementioned two cases $h>N/2$ and $h\leq 0$ also give, as a special case, an elegant generalization of Ramanujan's formula \cite[Theorem 1.2]{dixitmaji1} given below.

Let $N$ be an odd positive integer and $\a,\b>0$ such that $\a\b^{N}=\pi^{N+1}$. Then for any non-zero integer $m$,
{\allowdisplaybreaks\begin{align}\label{zetageneqn}
&\a^{-\frac{2Nm}{N+1}}\left(\frac{1}{2}\zeta(2Nm+1)+\sum_{n=1}^{\infty}\frac{n^{-2Nm-1}}{\textup{exp}\left((2n)^{N}\a\right)-1}\right)\nonumber\\
&=\left(-\b^{\frac{2N}{N+1}}\right)^{-m}\frac{2^{2m(N-1)}}{N}\Bigg(\frac{1}{2}\zeta(2m+1)+(-1)^{\frac{N+3}{2}}\sum_{j=\frac{-(N-1)}{2}}^{\frac{N-1}{2}}(-1)^{j}\sum_{n=1}^{\infty}\frac{n^{-2m-1}}{\textup{exp}\left((2n)^{\frac{1}{N}}\b e^{\frac{i\pi j}{N}}\right)-1}\Bigg)\nonumber\\
&\quad+(-1)^{m+\frac{N+3}{2}}2^{2Nm}\sum_{j=0}^{\left\lfloor\frac{N+1}{2N}+m\right\rfloor}\frac{(-1)^jB_{2j}B_{N+1+2N(m-j)}}{(2j)!(N+1+2N(m-j))!}\a^{\frac{2j}{N+1}}\b^{N+\frac{2N^2(m-j)}{N+1}}.
\end{align}}
In \cite[Theorem 2.1]{ktyacta}, Kanemitsu, Tanigawa and Yoshimoto also studied the more general series 
\begin{equation}\label{gls}
\sum_{n=1}^{\infty}n^{N-2h}\frac{\textup{exp}(-an^{N}x)}{1-\textup{exp}(-n^{N}x)}
\end{equation}
and obtained a transformation for it when $0<a\leq 1$, $h\geq N/2$ and $N$ even \footnote{In the statement of this theorem in \cite{ktyacta}, the only condition given on $a$ is that it be positive, but it should really be $0<a\leq 1$, for, when $a>1$, one has to slightly modify the expression involving the Hurwitz zeta function. See Remark \ref{ag1} of the current paper. Also, the version of this transformation given there includes an additional parameter $\ell$, however, it is easily seen to be equivalent to the condition $N$ even and $h\geq N/2$ in conjunction with the series in \eqref{gls}.}. In the same paper, the trio also obtained a similar result for multiple Hurwitz zeta function \cite[Theorem 4.1]{ktyacta}.

In the current paper, we derive a transformation for the series in \eqref{gls} for \emph{any} positive integer $N$. This transformation can be conceived of as a formula for the Hurwitz zeta function $\zeta\left(\frac{N-2h+1}{N},a\right)$. In the case when $N$ is even and $h\geq N/2$, our result, though different in appearance, is equivalent to that of Kanemitsu, Tanigawa and Yoshimoto \cite[Theorem 2.1]{ktyacta}. However, we extend it to include the case $h<N/2$ too. Also, in the special case $a=1$ of the above series that was considered in \cite{dixitmaji1}, it was demonstrated that one obtains more interesting results when $N$ is odd. Here too, the same phenomenon is observed for $0<a\leq 1$ in general. 
%Also, as was the case $a=1$ of the above series \cite{dixitmaji1}, the most interesting case here as well happens when $N$ is odd.
A transformation of the above series for $N$ odd and $h\geq 0$ is derived for the first time in this paper. It not only involves the generalized Lambert series with coefficients as trigonometric functions but also contains a new construct, which is an infinite series consisting of $\psi(z)$, the logarithmic derivative of the gamma function $\G(z)$, and a logarithm. Two of the main theorems of our paper, namely Theorems \ref{dgkmgen} and \ref{dgkmord2}, which give the transformation for the series in \eqref{gls} for any positive integer $N$ and $h\geq N/2$ are presented below. The nice thing about them is that they are totally explicit, and the expression other than the residual terms, that is $S(x, a)$ (see Equations \eqref{sodda} and \eqref{sevena} below), is written in the form where one of the inner expressions involve only $\cos(2\pi n a)$ and the other, only $\sin(2\pi n a)$. This allows us to easily recover, under certain conditions, the results in \cite{dixitmaji1} as corollaries since when $a=1$, the expression involving $\sin(2\pi n a)$ simply vanishes. Such an expression is also reminiscent of Hurwitz's formula \cite[p.~72]{dav}, namely, for Re$(s)<0$ and $0<a\leq 1$,
\begin{equation}\label{hformula}
\zeta(s, a)=\frac{2\G(1-s)}{(2\pi)^{1-s}}\left(\sin\left(\frac{\pi s}{2}\right)\sum_{n=1}^{\infty}\frac{\cos(2\pi n a)}{n^{1-s}}+\cos\left(\frac{\pi s}{2}\right)\sum_{n=1}^{\infty}\frac{\sin(2\pi n a)}{n^{1-s}}\right).
\end{equation}
It is also valid for Re$(s)<1$ provided $a\neq 1$. Indeed, Hurwitz's formula will play an important role in the proofs of our theorems. 

We now state the first main result of our paper.
\begin{theorem}\label{dgkmgen}
Let $N$ be a positive integer and $h$ be an integer such that $h\geq N/2$. Let $x>0$ and $0<a\leq 1$. Let $A_{N,j}(y):=\pi\left(2\pi y\right)^{\frac{1}{N}}e^{\frac{i\pi j}{N}}$. If $\displaystyle\frac{N-2h+1}{N}\neq-2\left\lfloor\frac{h}{N}-\frac{1}{2}\right\rfloor$, then
\begin{align}\label{dgkmgeneqn}
\sum_{n=1}^{\infty}n^{N-2h}\frac{\textup{exp}(-an^{N}x)}{1-\textup{exp}(-n^{N}x)}=P(x,a)+S(x,a),
\end{align}
where
{\allowdisplaybreaks\begin{align}\label{pxa}
P(x,a)&:=-\left(a-\frac{1}{2}\right)\zeta(-N+2h)+\frac{\zeta(2h)}{x}+\frac{1}{N}\G\left(\frac{N-2h+1}{N}\right)\zeta\left(\frac{N-2h+1}{N},a\right)x^{-\frac{(N-2h+1)}{N}}\nonumber\\
&\qquad-\sum_{j=1}^{\left\lfloor\frac{h}{N}-\frac{1}{2}\right\rfloor}\frac{B_{2j+1}(a)}{(2j+1)!}\zeta\left(2h-(2j+1)N\right)x^{2j}\nonumber\\
&\qquad+\frac{(-1)^{h+1}}{2}(2\pi)^{2h}\sum_{j=1}^{\left\lfloor\frac{h}{N}\right\rfloor}\left(\frac{-1}{4\pi^2}\right)^{jN}\frac{B_{2j}(a)B_{2h-2jN}}{(2j)!(2h-2jN)!}x^{2j-1},
\end{align}}
and
\begin{align}\label{sodda}
S(x,a)&:=\frac{(-1)^{h+1}}{N}\left(\frac{2\pi}{x}\right)^{\frac{N-2h+1}{N}}\sum_{j=-\frac{(N-1)}{2}}^{\frac{(N-1)}{2}}e^{\frac{i\pi(1-2h)j}{N}}\Bigg\{\sum_{n=1}^{\infty}\frac{\cos(2\pi na)}{n^{\frac{2h-1}{N}}\left(\textup{exp}\left(2A_{N,j}\left(\frac{n}{x}\right)\right)-1\right)}\nonumber\\
&\quad+\frac{(-1)^{j+\frac{N+1}{2}}}{\pi}\sum_{n=1}^{\infty}\frac{\sin(2\pi na)}{n^{\frac{2h-1}{N}}}\left\{\log\left(\tfrac{1}{\pi}A_{N,j}\left(\tfrac{n}{x}\right)\right)-\tfrac{1}{2}\left(\psi\left(\tfrac{i}{\pi}A_{N,j}\left(\tfrac{n}{x}\right)\right)+\psi\left(-\tfrac{i}{\pi}A_{N,j}\left(\tfrac{n}{x}\right)\right)\right)\right\}\Bigg\}
\end{align}
for $N$ odd, and 
\begin{align}\label{sevena}
S(x,a)&:=\frac{(-1)^{h+1}}{N}\left(\frac{2\pi}{x}\right)^{\frac{N-2h+1}{N}}\sum_{j=-\frac{N}{2}}^{\frac{N}{2}-1}e^{\frac{i\pi(1-2h)\left(j+\frac{1}{2}\right)}{N}}\sum_{n=1}^{\infty}\frac{\cos(2\pi na)+i(-1)^{j+\frac{N}{2}+1}\sin(2\pi na)}{n^{\frac{2h-1}{N}}\left(\textup{exp}\left(2A_{N,j+\frac{1}{2}}\left(\frac{n}{x}\right)\right)-1\right)}
%\nonumber\\
%&\quad\sum_{n=1}^{\infty}\frac{\sin(2\pi na)}{n^{\frac{2h-1}{N}}\left(\textup{exp}\left(2A_{N,j+\frac{1}{2}}\left(\frac{n}{x}\right)\right)-1\right)}\Bigg\}
\end{align}
for $N$ even.
\end{theorem}
\begin{remark}
Note that the above theorem does not hold for $N=1$.
\end{remark}
\begin{remark}\label{bc}
The appearance of $\zeta\left(\frac{N-2h+1}{N},a\right)$ in Theorem \textup{\ref{dgkmgen}} implies that this result can be conceived of as a formula for the Hurwitz zeta function at rational arguments, namely $\zeta\left(\frac{b}{c},a\right)$, when $b$ is odd and $c$ is a positive even integer, or when $b$ is even and $c$ is a positive odd integer. The former case when $b$ is a negative odd integer and $c$ is a positive even integer was established in \cite{ktyacta} as discussed earlier.
\end{remark}
 \begin{remark}\label{ag1}
When $a>1$, one can still obtain a representation for $\zeta\left(\frac{b}{c},a\right)$. We consider two cases depending upon whether $a$ is an integer or not. If $a>1$ is not an integer, we apply Theorem \textup{\ref{dgkmgen}} with $a$ replaced by its fractional part $\{a\}$ and then using the fact that $\zeta(s, \{a\})=\zeta(s, a)+\sum_{\ell=1}^{\lfloor a\rfloor}\left(\ell+\{a\}-1\right)^{-s}$. The above identity can be easily seen to be true for \textup{Re}$(s)>1$ first, and then for all complex $s$ by analytic continuation. Now if $a>1$ is an integer, we can use Theorem \textup{\ref{dgkmgen}} with $a$ \emph{there} to be $1$, and then the identity $\zeta(s)=\zeta(s, a)+\sum_{\ell=1}^{a-1}\ell^{-s}$. This identity can be also first proved for \textup{Re}$(s)>1$ and then extended to all complex $s$ by analytic continuation.
\end{remark}

The above theorem is proved by representing the series on the left side of \eqref{dgkmgeneqn} as a line integral and then doing a careful analysis of it using contour integration. An important ingredient in the proof is a new identity which gives a closed-form expression for an infinite sum whose summand is Raabe's integral $\mathfrak{R}(y, w)$. For Re$(w)>0$ and $y>0$, the latter is given by \cite[p.~144]{htf2}
\begin{equation}\label{raabe}
\mathfrak{R}(y, w):=\displaystyle\int_{0}^{\infty}\frac{t\cos(yt)}{t^2+w^2}\, \mathrm{d}t. 
\end{equation}
The aforementioned identity on infinite series of Raabe's integrals which is interesting in itself, and to the best of our knowledge is new, is now given.
\begin{theorem}\label{raabesum}
Let $u\in\mathbb{C}$ be fixed such that \textup{Re}$(u)>0$. Then, 
\begin{equation}\label{raabesumeqn}
\sum_{m=1}^{\infty}\int_{0}^{\infty}\frac{t\cos(t)}{t^2+m^2u^2}\, \mathrm{d}t=\frac{1}{2}\left\{\log\left(\frac{u}{2\pi}\right)-\frac{1}{2}\left(\psi\left(\frac{iu}{2\pi}\right)+\psi\left(\frac{-iu}{2\pi}\right)\right)\right\}.
\end{equation}
\end{theorem}
The series on the left-hand side of this result is not amenable to a straightforward evaluation and hence to obtain the result we had to use Guinand's generalization of the Poisson summation formula \cite[Theorem 1]{apg1}. Note that interchanging the order of summation and integration leads to a divergent integral. 
%Our initial unsuccessful attempts to evaluate this sum now give us interesting alternate representations for this sum, for example, through a result of Dixon and Ferrar \cite{dixfer3}, and hence their evaluation too.
It is interesting to note that while Raabe's integral itself is evaluable in terms of, either the exponential integral function \cite[p.~144, Equation (13)]{htf2}, \cite[p.~428, Formula \textbf{3.723.5}]{grn} or, equivalently, $\textup{Shi}(x)$ and $\textup{Chi}(x)$ functions \cite[p.~895, Formulas \textbf{8.221.1, 8.221.2}]{grn}, which are not-so-common special functions, the infinite sum of Raabe integrals can be expressed in terms of well-known functions, namely, the digamma function $\psi(z)$, and $\log(z)$, which is an elementary function.

%\textbf{Remark.}  Note that when $h<N/2$, one may still have $\displaystyle\frac{N-2h+1}{N}=-2\left\lfloor\frac{h}{N}-\frac{1}{2}\right\rfloor$. That doesn't mean that we should be considering Theorem \ref{dgkmord2} in this case. This is because, when $h<N/2$, $\left\lfloor\frac{h}{N}-\frac{1}{2}\right\rfloor\leq -1$ and hence $-2\left\lfloor\frac{h}{N}-\frac{1}{2}\right\rfloor\geq 2$. However, this double pole at $-2\left\lfloor\frac{h}{N}-\frac{1}{2}\right\rfloor$ was arising only because $\Gamma(s)$ has poles at negative even integers. But the fact that $-2\left\lfloor\frac{h}{N}-\frac{1}{2}\right\rfloor\geq 2$ implies that one cannot use Theorem \ref{dgkmord2} when $h<N/2$. So we have to use Theorem \ref{dgkmgen}. This can also be concluded from that fact that the equation $\frac{N-2h+1}{N}=-2j$ has no solution for $j\in\mathbb{N}$ when $h<N/2$.
A complement of Theorem \ref{dgkmgen} is stated next.
\begin{theorem}\label{dgkmord2}
Let $\g$ denote Euler's constant. Let $0<a\leq 1$. Let $N$ be an odd positive integer. Let $h$ be an integer such that $h>N/2$. Let $A_{N,j}(y)$ be defined as in Theorem \textup{\ref{dgkmgen}}. If $\displaystyle\tfrac{N-2h+1}{N}=-2\left\lfloor\tfrac{h}{N}-\tfrac{1}{2}\right\rfloor\neq 0$, then
%\begin{align}
%\sum_{n=1}^{\infty}n^{N-2h}\frac{\textup{exp}(-an^{N}x)}{1-\textup{exp}(-n^{N}x)}=P^{*}(x,a)+S(x,a),
%\end{align}
{\allowdisplaybreaks\begin{align}\label{psxa}
&\sum_{n=1}^{\infty}n^{N-2h}\frac{\textup{exp}(-an^{N}x)}{1-\textup{exp}(-n^{N}x)}\nonumber\\
&=-\left(a-\tfrac{1}{2}\right)\zeta(-N+2h)-\g \frac{B_{2\left\lfloor\frac{h}{N}-\frac{1}{2}\right\rfloor+1}(a)}{\left(2\left\lfloor\tfrac{h}{N}-\tfrac{1}{2}\right\rfloor\right)!}\hspace{0.5mm}x^{2\left\lfloor\frac{h}{N}-\frac{1}{2}\right\rfloor}+\frac{(-1)^{\left\lfloor\frac{h}{N}-\frac{1}{2}\right\rfloor}}{2N}\left(\frac{2\pi}{x}\right)^{\frac{N-2h+1}{N}}\sum_{n=1}^{\infty}\frac{\cos(2\pi na)}{n^{\frac{2h-1}{N}}}\nonumber\\
&\quad-\sum_{j=1}^{\left\lfloor\frac{h}{N}-\frac{1}{2}\right\rfloor-1}\frac{B_{2j+1}(a)}{(2j+1)!}\zeta\left(2h-(2j+1)N\right)x^{2j}+\frac{(-1)^{h+1}(2\pi)^{2h}}{2}\sum_{j=0}^{\left\lfloor\frac{h}{N}\right\rfloor}\left(\tfrac{-1}{4\pi^2}\right)^{jN}\frac{B_{2j}(a)B_{2h-2jN}}{(2j)!(2h-2jN)!}x^{2j-1}\nonumber\\
&\quad+\frac{(-1)^{h+1}}{N}\left(\frac{2\pi}{x}\right)^{\frac{N-2h+1}{N}}\sum_{j=-\frac{(N-1)}{2}}^{\frac{(N-1)}{2}}(-1)^j\Bigg\{\sum_{n=1}^{\infty}\frac{\cos(2\pi na)}{n^{\frac{2h-1}{N}}\left(\textup{exp}\left(2A_{N,j}\left(\frac{n}{x}\right)\right)-1\right)}\nonumber\\
&\qquad+\frac{(-1)^{j+\frac{N+3}{2}}}{2\pi}\sum_{n=1}^{\infty}\frac{\sin(2\pi na)}{n^{\frac{2h-1}{N}}}\left(\psi\left(\tfrac{i}{\pi}A_{N,j}\left(\tfrac{n}{x}\right)\right)+\psi\left(-\tfrac{i}{\pi}A_{N,j}\left(\tfrac{n}{x}\right)\right)\right)\Bigg\}.
\end{align}}
%\begin{align}
%\sum_{n=1}^{\infty}n^{N-2h}\frac{\textup{exp}(-an^{N}x)}{1-\textup{exp}(-n^{N}x)}=P^{*}(x,a)+S(x,a),
%\end{align}
%\begin{align}\label{psxa}
%P^{*}(x,a)&:=-\left(a-\frac{1}{2}\right)\zeta(-N+2h)+\frac{\zeta(2h)}{x}\nonumber\\
%&\quad+\frac{x^{2\left\lfloor\frac{h}{N}-\frac{1}{2}\right\rfloor}}{N\left(2\left\lfloor\tfrac{h}{N}-\tfrac{1}{2}\right\rfloor\right)!}\left\{-\frac{B_{2\left\lfloor\frac{h}{N}-\frac{1}{2}\right\rfloor+1}(a)}{2\left\lfloor\frac{h}{N}-\frac{1}{2}\right\rfloor+1}\left(\psi\left(2\left\lfloor\tfrac{h}{N}-\tfrac{1}{2}\right\rfloor+1\right)+N\g-\log x\right)+\zeta'\left(-2\left\lfloor\tfrac{h}{N}-\tfrac{1}{2}\right\rfloor,a\right)\right\}\nonumber\\
%&\quad-\sum_{j=1}^{\left\lfloor\frac{h}{N}-\frac{1}{2}\right\rfloor-1}\frac{B_{2j+1}(a)}{(2j+1)!}\zeta\left(2h-(2j+1)N\right)x^{2j}+(-1)^{h+1}2^{2h-1}\pi^{2h}\sum_{j=1}^{\left\lfloor\frac{h}{N}\right\rfloor}\left(\frac{-1}{4\pi^2}\right)^{jN}\frac{B_{2j}(a)B_{2h-2jN}}{(2j)!(2h-2jN)!}x^{2j-1}.
%\end{align}
\end{theorem}
\begin{remark}\label{equi}
An equivalent version of the above theorem, comparable in appearance to Theorem \textup{\ref{dgkmgen}}, is given in \eqref{dgkmord2p}.
\end{remark}
%%%%%%%%%%%%%%%%%%%%%%%%%%%%%%%%%%%%%%%%%%
%\begin{remark}\label{ag1ord2}
%The above result can be conceived of as a representation for the derivative of Hurwitz zeta function at negative even integers. We note that while the derivatives of Hurwitz zeta function at negative odd integers are known to have closed-form expressions \cite{miladam}, the corresponding problem for negative even integers is very difficult in that there are no closed-form expressions known for $\zeta'(-2j,a), j>0$, as of yet. Similar to what was done in Remark \ref{ag1}, one can obtain a representation for $\zeta'(-2j,a)$ for $a>1$, but with the restriction that $a$ not be an integer. To do this, we replace $a$ by its fractional part $\{a\}$, apply Theorem \ref{dgkmord2}, and then apply the identity
%\begin{equation*}
%\zeta'(s, \{a\})=\zeta'(s, a)-\sum_{\ell=1}^{\lfloor a\rfloor}\frac{\log(\ell+\{a\}-1)}{(\ell+\{a\}-1)^{s}},
%\end{equation*}
%which can be obtained by differentiating \eqref{zetax} with respect to $s$.
%\end{remark}
%%%%%%%%%%%%%%%%%%%%%%%%%%%%%%%%%%%%%%%%%%%5
One difference in the hypotheses of Theorems \ref{dgkmgen} and \ref{dgkmord2}, when $N$ is an odd positive integer, is that in the first, we have $\frac{N-2h+1}{N}\neq-2\left\lfloor\frac{h}{N}-\frac{1}{2}\right\rfloor$, whereas in the second, $\frac{N-2h+1}{N}=-2\left\lfloor\frac{h}{N}-\frac{1}{2}\right\rfloor\neq0$. (The remaining case $\frac{N-2h+1}{N}=-2\left\lfloor\frac{h}{N}-\frac{1}{2}\right\rfloor=0$ is covered in Theorem \ref{dgkmord2m0} below.) Note that the equality $\frac{N-2h+1}{N}=-2\left\lfloor\frac{h}{N}-\frac{1}{2}\right\rfloor$ does not hold for any even $N$, but it may very well for some specific values of $N$ odd and $h$. Even though at a first glance, these conditions may look artificial, as will be seen in the proofs, they arise naturally while examining the poles of the integrand of the line integral representation of the series $\sum_{n=1}^{\infty}n^{N-2h}\frac{\textup{exp}(-an^{N}x)}{1-\textup{exp}(-n^{N}x)}$, for this representation has, in its contour integral representation, its integrand as $\G(s)\zeta(s, a)\zeta\left(Ns-(N-2h)\right)x^{-s}$ (see \eqref{mainequality} below). So if we now consider the poles of $\G(s)$ at $-2, -4, -6,\cdots$, they get canceled by the zeros of $\zeta(s, a)$ \emph{only} when $a=1$ or $a=\frac{1}{2}$, for then $\zeta(s, 1)=\zeta(s)$ and 
\begin{equation}\label{hzetahalf}
\zeta\left(s,\tfrac{1}{2}\right)=(2^s-1)\zeta(s),
\end{equation}
and it is well-known that $\zeta(-2m)=0$ for $m\geq 1$. However, for $0<a<1, a\neq\frac{1}{2}$, $\zeta(-2m,a), m\geq 1$, may not always be zero. 

In fact, a theorem due to Spira \cite[Theorem 3]{spira} states that if Re$(s)\leq-(4a+1+2\left\lfloor1-2a\rfloor\right)$ and $|\textup{Im}(s)|\leq 1$, then $\zeta(s, a)\neq 0$ except for trivial zeros on the negative real axis, one in each interval $(-2n-4a-1, -2n-4a+1)$, where $n\geq 1-2a$. Thus, some (or all) of the poles of $\G(s)$ at $s=-2m, m\geq 1$, may very well contribute non-zero residues towards the evaluation of the line integral. Now $h\geq N/2$ implies that $\left\lfloor\frac{h}{N}-\frac{1}{2}\right\rfloor\geq 0$. First consider $\left\lfloor\frac{h}{N}-\frac{1}{2}\right\rfloor>0$ so that $-2\left\lfloor\frac{h}{N}-\frac{1}{2}\right\rfloor$ is a legitimate pole of $\G(s)$. If, in addition, we have $\frac{N-2h+1}{N}=-2j$ for some $j\in\mathbb{N}$, then Lemma \ref{equality} below implies that $j=\left\lfloor\frac{h}{N}-\frac{1}{2}\right\rfloor$. Now since $\frac{N-2h+1}{N}$ is the pole of $\zeta(Ns-(N-2h)$, we find that this is a double order pole of the integrand. This is why $P(x, a)$ in Theorem \ref{dgkmgen} gets modified to $P^{*}(x, a)$ as can be seen in \eqref{dgkmord2p}, which is an equivalent version of Theorem \ref{dgkmord2}. 

%The other difference in the hypotheses of Theorems \ref{dgkmgen} and \ref{dgkmord2}, when $N$ is an odd positive integer, is that while the first one is valid for $0<a\leq 1$, the second is valid only for $0<a<1$. This is because, letting $a=1$ leads us back to $\zeta(s)$, and then we cannot have a double pole at $-2\left\lfloor\frac{h}{N}-\frac{1}{2}\right\rfloor$ since this is also a zero of $\zeta(s)$.

%It makes sense to pause now for a bit and note how the argument drastically differs when we transition from $\zeta(s)$ (as considered in \cite[Theorem 1.1]{dixitmaji1}) to $\zeta(s, a)$, $a\neq\frac{1}{2}, 1$. 
The aforementioned fact about $\zeta(s, a)$ not always having zeros at $s=-2m, m\in\mathbb{N}$ for $0<a<1$ suggests us to write down the important differences that are present between $\zeta(s, a)$ and $\zeta(s)$. 
%We note that the resemblance of $\zeta(s, a)$ with $\zeta(s)$ is, at best, superficial.
Unlike $\zeta(s)$, $\zeta(s, a)$, $a\neq\frac{1}{2}, 1$, has no Euler product. It is known, due to Davenport and Heilbronn \cite{davhel} in the case when $a(\neq\frac{1}{2}, 1)$ is rational or transcendental, and due to Cassels \cite{cassels} in the case when $a$ is algebraic irrational, that $\zeta(s, a)$ has infinitely many zeros in the half-plane Re$(s)>1$. Moreover, when $a(\neq\frac{1}{2}, 1)$ is rational, Voronin \cite{voronin} proved that $\zeta(s, a)$ has infinitely many zeros in the critical strip, and to the right of the critical line Re$(s)=\frac{1}{2}$. The corresponding result when $a$ is transcendental was obtained by Gonek \cite{gonek}.

%Letting $h$ in Theorem \ref{dgkmord2} to run through a specific sequence, namely $h=\frac{N+1}{2}+Nm, m>0$, we obtain the following result. 
We now give an equivalent version of Theorem \ref{dgkmord2}, which, for $m>0$, gives a two-parameter generalization of Ramanujan's formula for $\zeta(2m+1)$.
\begin{theorem}\label{ggram}
Let $0<a\leq 1$, let $N$ be an odd positive integer and $\a,\b>0$ such that $\a\b^{N}=\pi^{N+1}$. Then for any positive integer $m$,
{\allowdisplaybreaks\begin{align}\label{zetageneqna}
&\a^{-\frac{2Nm}{N+1}}\bigg(\left(a-\frac{1}{2}\right)\zeta(2Nm+1)+\sum_{j=1}^{m-1}\frac{B_{2j+1}(a)}{(2j+1)!}\zeta(2Nm+1-2jN)(2^N\a)^{2j}\nonumber\\
&\qquad\qquad+\sum_{n=1}^{\infty}\frac{n^{-2Nm-1}\textup{exp}\left(-a(2n)^{N}\a\right)}{1-\textup{exp}\left(-(2n)^{N}\a\right)}\bigg)\nonumber\\
&=\left(-\b^{\frac{2N}{N+1}}\right)^{-m}\frac{2^{2m(N-1)}}{N}\bigg[\frac{(-1)^{m+1}(2\pi)^{2m}B_{2m+1}(a)N \g}{(2m+1)!}+\frac{1}{2}\sum_{n=1}^{\infty}\frac{\cos(2\pi na)}{n^{2m+1}}\nonumber\\
&\quad+(-1)^{\frac{N+3}{2}}\sum_{j=\frac{-(N-1)}{2}}^{\frac{N-1}{2}}(-1)^{j}\bigg\{\sum_{n=1}^{\infty}\frac{n^{-2m-1}\cos(2\pi na)}{\textup{exp}\left((2n)^{\frac{1}{N}}\b e^{\frac{i\pi j}{N}}\right)-1}\nonumber\\
&\quad+\frac{(-1)^{j+\frac{N+3}{2}}}{2\pi}\sum_{n=1}^{\infty}\frac{\sin(2\pi na)}{n^{2m+1}}\left(\psi\left(\tfrac{i\beta}{2\pi}(2n)^{\frac{1}{N}} e^{\frac{i\pi j}{N}}\right)+\psi\left(\tfrac{-i\beta}{2\pi}(2n)^{\frac{1}{N}} e^{\frac{i\pi j}{N}}\right)\right)\bigg\}\bigg]\nonumber\\
&\quad+(-1)^{m+\frac{N+3}{2}}2^{2Nm}\sum_{j=0}^{\left\lfloor\frac{N+1}{2N}+m\right\rfloor}\frac{(-1)^jB_{2j}(a)B_{N+1+2N(m-j)}}{(2j)!(N+1+2N(m-j))!}\a^{\frac{2j}{N+1}}\b^{N+\frac{2N^2(m-j)}{N+1}}.
\end{align}}
%\begin{align}\label{zetageneqna}
%&\a^{-\frac{2Nm}{N+1}}\left(\left(a-\frac{1}{2}\right)\zeta(2Nm+1)+\sum_{j=1}^{m-1}\frac{B_{2j+1}(a)}{(2j+1)!}\zeta(2Nm+1-2jN)(2^N\a)^{2j}+\sum_{n=1}^{\infty}\frac{n^{-2Nm-1}\textup{exp}\left(-a(2n)^{N}\a\right)}{1-\textup{exp}\left(-(2n)^{N}\a\right)}\right)\nonumber\\
%&=\left(-\b^{\frac{2N}{N+1}}\right)^{-m}\frac{2^{2m(N-1)}}{N}\bigg[\frac{(-1)^m(2\pi)^{2m}}{(2m)!}\left\{-\frac{B_{2m+1}(a)}{2m+1}\left(\psi(2m+1)+N\g-\log(2^{N}\a)\right)+\zeta'(-2m,a)\right\}\nonumber\\
%&\quad+(-1)^{\frac{N+3}{2}}\sum_{j=\frac{-(N-1)}{2}}^{\frac{N-1}{2}}(-1)^{j}\bigg(\sum_{n=1}^{\infty}\frac{n^{-2m-1}\cos(2\pi na)}{\textup{exp}\left((2n)^{\frac{1}{N}}\b e^{\frac{i\pi j}{N}}\right)-1}\nonumber\\
%&\quad+\frac{1}{\pi}(-1)^{j+\frac{N+1}{2}}\sum_{n=1}^{\infty}\frac{\sin(2\pi na)}{n^{2m+1}}\left\{\log\left(\tfrac{\beta}{2\pi}(2n)^{\frac{1}{N}} e^{\frac{i\pi j}{N}}\right)-\frac{1}{2}\left(\psi\left(\tfrac{i\beta}{2\pi}(2n)^{\frac{1}{N}} e^{\frac{i\pi j}{N}}\right)+\psi\left(\tfrac{-i\beta}{2\pi}(2n)^{\frac{1}{N}} e^{\frac{i\pi j}{N}}\right)\right)\right\}\bigg)\bigg]\nonumber\\
%&\quad+(-1)^{m+\frac{N+3}{2}}2^{2Nm}\sum_{j=0}^{\left\lfloor\frac{N+1}{2N}+m\right\rfloor}\frac{(-1)^jB_{2j}(a)B_{N+1+2N(m-j)}}{(2j)!(N+1+2N(m-j))!}\a^{\frac{2j}{N+1}}\b^{N+\frac{2N^2(m-j)}{N+1}}.
%\end{align}
\end{theorem}
When we let $a=1$ in the above theorem, we obtain \eqref{zetageneqn} for positive integers $m$, which, in turn, as remarked before, gives Ramanujan's formula \eqref{zetaodd} for positive integers $m$ as its special case. 
%This is done in Section \ref{ord2}.
%\textbf{Remark:} Note that when $N$ is odd and $0<h<N/2$, then $\displaystyle\frac{N-2h+1}{N}\neq-2\left\lfloor\frac{h}{N}-\frac{1}{2}\right\rfloor$.

The above theorem gives, as a special case, the following beautiful formula relating $\zeta(3), \zeta(5)$,\newline
$\zeta(7), \zeta(9)$ and $\zeta(11)$.
\begin{corollary}\label{zeta311}
The following identity holds:
{\allowdisplaybreaks\begin{align*}
&\frac{277}{8257536}\frac{\zeta(3)}{\pi^3}-\frac{61}{184320}\frac{\zeta(5)}{\pi^5}+\frac{5}{1536}\frac{\zeta(7)}{\pi^7}-\frac{1}{32}\frac{\zeta(9)}{\pi^9}+\frac{1049599}{4194304}\frac{\zeta(11)}{\pi^{11}}\nonumber\\
&\quad+\frac{1315686689}{3570822807552000}-\frac{50521}{14863564800}\frac{\gamma}{\pi}\nonumber\\
%&\frac{277}{8257536}\frac{\zeta(3)}{\pi^3}-\frac{61}{184320}\frac{\zeta(5)}{\pi^5}+\frac{5}{1536}\frac{\zeta(7)}{\pi^7}-\frac{1}{32}\frac{\zeta(9)}{\pi^9}+\frac{1049599}{4194304}\frac{\zeta(11)}{\pi^{11}}+\frac{1315686689}{3570822807552000}-\frac{50521\g}{14863564800\pi}\nonumber\\
&=\frac{1}{\pi^{11}}\sum_{n=1}^{\infty}\frac{e^{3\pi n/2}}{n^{11}\left(e^{2\pi n}-1\right)}+\frac{1}{2048\pi^{11}}\sum_{n=1}^{\infty}\frac{(-1)^n}{n^{11}\left(e^{4\pi n}-1\right)}+\frac{1}{2\pi^{12}}\sum_{n=1}^{\infty}\frac{\sin\left(\frac{\pi n}{2}\right)}{n^{11}}\left(\psi(in)+\psi(-in)\right).
\end{align*}}
\end{corollary}
It should be noted that there are formulas of other type linking $\zeta(3), \cdots, \zeta(2m+1)$ discovered, for example, by Wilton \cite{wiltonmess}, by Srivastava \cite{srivastava} (see also the references therein), and by Kanemitsu, Tanigawa and Yoshimoto \cite{ktyham}. For details, refer to \cite{ktyham}. However, the advantage of Theorem \ref{ggram} lies in the fact that one can vary $N$ over the set of odd positive integers, and hence it allows us to obtain a relation between odd zeta values $\zeta(2N+1), \zeta(4N+1), \zeta(6N+1), \cdots, \zeta(2Nm+1)$. We refer the reader to Table \ref{t4}.

Note that it is widely believed \cite[Conjecture 27]{waldschmidt} that for any $n  \in \mathbb{N}$, and any non-zero polynomial $P \in \mathbb{Q}[x_0, x_1, \cdots, x_n]$, $P(\pi, \zeta(3), \zeta(5), \cdots, \zeta(2n+1)) \neq 0 $, that is, $\pi$ and all odd zeta values are algebraically independent over $\mathbb{Q}$. This conjecture, if true, would imply, in particular, that all odd zeta values are transcendental. While this is not known as of yet for even a single odd zeta value $\zeta(2m+1), m>0$, Ap\'{e}ry \cite{apery1}, \cite{apery2} surprisingly proved that $\zeta(3)$ is irrational. Also, Rivoal \cite{rivoal}, and Ball and Rivoal \cite{ballrivoal} have proved that there exist infinitely many odd zeta values which are irrational. However, one does not know which out of these odd zeta values (except $\zeta(3)$) are irrational. Currently the best result in this direction is due to Zudilin \cite{zudilin} which says that at least one of $\zeta(5), \zeta(7), \zeta(9)$ or $\zeta(11)$ is irrational.

\noindent
We now deduce a new formula for $\zeta(2m+1)$ by letting $a=1/2$ in Theorem \ref{ggram}.
\begin{theorem}\label{ggramhalf}
Let $N$ be an odd positive integer and $\a,\b>0$ such that $\a\b^{N}=\pi^{N+1}$. Then for any positive integer $m$,
{\allowdisplaybreaks\begin{align}\label{zetagenhalfa}
&\a^{-\frac{2Nm}{N+1}}\sum_{n=1}^{\infty}\frac{n^{-2Nm-1}\textup{exp}\left(-\tfrac{1}{2}(2n)^{N}\a\right)}{1-\textup{exp}\left(-(2n)^{N}\a\right)}\nonumber\\
&=\left(-\b^{\frac{2N}{N+1}}\right)^{-m}\frac{2^{2m(N-1)}}{N}\bigg[\frac{(2^{-2m}-1)}{2}\zeta(2m+1)\nonumber\\
&\qquad\qquad\qquad\qquad\quad\qquad+(-1)^{\frac{N+3}{2}}\sum_{j=\frac{-(N-1)}{2}}^{\frac{N-1}{2}}(-1)^{j}\sum_{n=1}^{\infty}\frac{(-1)^nn^{-2m-1}}{\textup{exp}\left((2n)^{\frac{1}{N}}\b e^{\frac{i\pi j}{N}}\right)-1}\bigg]\nonumber\\
&\quad+(-1)^{m+\frac{N+3}{2}}2^{2Nm}\sum_{j=0}^{\left\lfloor\frac{N+1}{2N}+m\right\rfloor}\frac{(-1)^j(2^{1-2j}-1)B_{2j}B_{N+1+2N(m-j)}}{(2j)!(N+1+2N(m-j))!}\a^{\frac{2j}{N+1}}\b^{N+\frac{2N^2(m-j)}{N+1}}.
\end{align}}
\end{theorem}
%One of the reasons this result is interesting is because it contains only one odd zeta value, namely $\zeta(2m+1)$, $m>0$, and hence can be regarded as a formula for $\zeta(2m+1)$. Note that while Lerch's formula \cite{lerch} 
%\begin{align}\label{lerch}
%\zeta(2m+1)+2\sum_{n=1}^{\infty}\frac{1}{n^{2m+1}(e^{2\pi n}-1)}=\pi^{2m+1}2^{2m}\sum_{j=0}^{m+1}\frac{(-1)^{j+1}B_{2j}B_{2m+2-2j}}{(2j)!(2m+2-2j)!}
%\end{align}
%can also be considered to be a representation for $\zeta(2m+1)$, it is valid only for $m$ odd and positive. On the other hand, in our representation, $m$ can be any natural number. Also, our result contains two free parameters such as 
Theorem \ref{ggramhalf} gives the following Zudilin-type result on transcendence of certain constants.
\begin{corollary}\label{transzetaodd}
Let $m$ be a positive integer and $N$ be a positive odd integer. Then at least one of
\begin{equation*}
\zeta(2m+1),\hspace{3mm} \sum_{n=1}^{\infty}\frac{\textup{exp}\left(-\tfrac{1}{2}(2n)^{N}\pi\right)}{n^{2Nm+1}\left(1-\textup{exp}\left(-(2n)^{N}\pi\right)\right)},\hspace{3mm} \sum_{n=1}^{\infty}\frac{(-1)^n}{n^{2m+1}}\textup{Re}\bigg\{\frac{1}{\textup{exp}\left((2n)^{\frac{1}{N}}\pi e^{\frac{i\pi j}{N}}\right)-1}\bigg\},
\end{equation*}
where $j$ takes every value from $0$ to $\frac{N-1}{2}$, is transcendental.
\end{corollary}
For an odd positive integer $m$, Lerch's formula \cite{lerch} is given by
\begin{align*}
\zeta(2m+1)+2\sum_{n=1}^{\infty}\frac{1}{n^{2m+1}(e^{2\pi n}-1)}=\pi^{2m+1}2^{2m}\sum_{j=0}^{m+1}\frac{(-1)^{j+1}B_{2j}B_{2m+2-2j}}{(2j)!(2m+2-2j)!}.
\end{align*}
It is a special case of Ramanujan's formula \eqref{zetaodd}. Lerch's formula implies that at least one of $\zeta(2m+1)$ and $\sum_{n=1}^{\infty}\frac{1}{n^{2m+1}(e^{2\pi n}-1)}$ is transcendental \cite{gmr}. However, such information cannot be inferred from \eqref{zetaodd} when $m$ is even. The result in Corollary \ref{transzetaodd}, on the other hand, is valid irrespective of the parity of $m$.
If we now fix an odd positive integer $N$ and vary $m$ over the set of natural numbers, we obtain the following Rivoal-type result.
\begin{corollary}\label{rtype}
Let $N$ be any fixed odd positive integer. Then the set
\begin{align*}
&\bigcup_{m=1}^{\infty}\bigg\{\zeta(2m+1),\hspace{1mm} \sum_{n=1}^{\infty}\frac{\textup{exp}\left(-\tfrac{1}{2}(2n)^{N}\pi\right)}{n^{2Nm+1}\left(1-\textup{exp}\left(-(2n)^{N}\pi\right)\right)},\hspace{1mm} \sum_{n=1}^{\infty}\frac{(-1)^n}{n^{2m+1}}\textup{Re}\left(\frac{1}{\textup{exp}((2n)^{\frac{1}{N}}\pi e^{\frac{i\pi j}{N}})-1}\right):\nonumber\\
 &\qquad\quad j=0\,\,\textup{to}\,\,\frac{N-1}{2}\bigg\},
\end{align*}
%where $0\leq j\leq\frac{N-1}{2}$,
contains infinitely many transcendental numbers.
\end{corollary}
If we now fix $m$ and vary $N=2\ell+1$, $\ell\in\mathbb{N}\cup\{0\}$, in Corollary \ref{transzetaodd}, we obtain the following criterion for the transcendence of $\zeta(2m+1)$:
\begin{corollary}\label{criterionzeta}
If the set 
\begin{align*}
&\bigcup_{\ell=0}^{\infty}\bigg\{\sum_{n=1}^{\infty}\frac{\textup{exp}\left(-\tfrac{\pi}{2}(2n)^{2\ell+1}\right)}{n^{2m(2\ell+1)+1}\left(1-\textup{exp}\left(-\pi(2n)^{2\ell+1}\right)\right)},\hspace{3mm} \sum_{n=1}^{\infty}\frac{(-1)^n}{n^{2m+1}}\textup{Re}\bigg(\frac{1}{\textup{exp}\big((2n)^{\frac{1}{2\ell+1}}\pi e^{\frac{i\pi j}{2\ell+1}}\big)-1}\bigg)\nonumber\\
&\qquad: j=0\,\,\textup{to}\,\, \ell\bigg\}
\end{align*}
has only finitely many transcendental numbers, then $\zeta(2m+1)$ must be transcendental.
\end{corollary}

The transformation for the series $\sum_{n=1}^{\infty}n^{N-2h}\frac{\textup{exp}(-an^{N}x)}{1-\textup{exp}(-n^{N}x)}$ for $N$ odd and $\frac{N-2h+1}{N}=-2\left\lfloor\frac{h}{N}-\frac{1}{2}\right\rfloor=0$, that is, $h=\frac{N+1}{2}$, is given in the theorem below.
\begin{theorem}\label{dgkmord2m0}
Let $0<a\leq 1$ and $N$ be an odd positive integer. Let $A_{N,j}(y)$ be defined as in Theorem \textup{\ref{dgkmgen}}. Then
%If $\displaystyle\frac{N-2h+1}{N}=-2\left\lfloor\frac{h}{N}-\frac{1}{2}\right\rfloor=0$, that is, when $h=\frac{N+1}{2}$,
%Let $0<a\leq 1$. Let $N$ be an odd positive integer. If $\displaystyle\frac{N-2h+1}{N}=-2\left\lfloor\frac{h}{N}-\frac{1}{2}\right\rfloor=0$, that is, when $h=\frac{N+1}{2}$,
{\allowdisplaybreaks\begin{align}\label{limalim}
&\sum_{n=1}^{\infty}\frac{\textup{exp}(-an^{N}x)}{n(1-\textup{exp}(-n^{N}x))}\nonumber\\
&=\frac{\zeta(N+1)}{x}+\frac{1}{N}\left(\left(\tfrac{1}{2}-a\right)\left((N-1)\gamma-\log x\right)+\log\G(a)-\tfrac{1}{2}\log 2\pi\right)\nonumber\\
&\quad+(-1)^{\frac{N+3}{2}}2^{N}\pi^{N+1}\sum_{j=1}^{\left\lfloor\frac{N+1}{2N}\right\rfloor}\left(\frac{-1}{4\pi^2}\right)^{jN}\frac{B_{2j}(a)B_{N+1-2jN}}{(2j)!(N+1-2jN)!}x^{2j-1}\nonumber\\
&\quad+\frac{(-1)^{\frac{N+3}{2}}}{N}\sum_{j=-\frac{(N-1)}{2}}^{\frac{(N-1)}{2}}(-1)^j\bigg\{\sum_{n=1}^{\infty}\frac{\cos(2\pi na)}{n\left(\textup{exp}\left(2A_{N,j}\left(\frac{n}{x}\right)\right)-1\right)}\nonumber\\
&\quad+\frac{(-1)^{j+\frac{N+1}{2}}}{\pi}\sum_{n=1}^{\infty}\frac{\sin(2\pi na)}{n}\left\{\log\left(\tfrac{1}{\pi}A_{N,j}\left(\tfrac{n}{x}\right)\right)-\tfrac{1}{2}\left(\psi\left(\tfrac{i}{\pi}A_{N,j}\left(\tfrac{n}{x}\right)\right)+\psi\left(-\tfrac{i}{\pi}A_{N,j}\left(\tfrac{n}{x}\right)\right)\right)\right\}\bigg\}.
\end{align}}
Equivalently, if  $\a,\b>0$ such that $\a\b^{N}=\pi^{N+1}$, then
{\allowdisplaybreaks\begin{align}\label{limab}
&\sum_{n=1}^{\infty}\frac{\textup{exp}(-a(2n)^{N}\a)}{n(1-\textup{exp}(-(2n)^{N}\a))}-\frac{1}{N}(-1)^{\frac{N+3}{2}}\sum_{j=-\frac{(N-1)}{2}}^{\frac{N-1}{2}}(-1)^{j}\bigg(\sum_{n=1}^{\infty}\frac{\cos(2\pi na)}{n\left(\textup{exp}\left((2n)^{\frac{1}{N}}\b e^{\frac{i\pi j}{N}}\right)-1\right)}\nonumber\\
&+\frac{(-1)^{j+\frac{N+1}{2}}}{\pi}\sum_{n=1}^{\infty}\frac{\sin(2\pi na)}{n}\left\{\log\left(\tfrac{\beta}{2\pi}(2n)^{\frac{1}{N}} e^{\frac{i\pi j}{N}}\right)-\frac{1}{2}\left(\psi\left(\tfrac{i\beta}{2\pi}(2n)^{\frac{1}{N}} e^{\frac{i\pi j}{N}}\right)+\psi\left(\tfrac{-i\beta}{2\pi}(2n)^{\frac{1}{N}} e^{\frac{i\pi j}{N}}\right)\right)\right\}\bigg)\nonumber\\
&=\frac{1}{N}\left((a-1)\log(2\pi)+\log\G(a)\right)+\left(a-\frac{1}{2}\right)\left\{\frac{(N-1)(\log 2-\g)}{N}+\frac{\log(\a/\b)}{N+1}\right\}\nonumber\\
%&=\frac{1}{N}\left(\left(\tfrac{1}{2}-a\right)\left((N-1)\gamma-\log\left(2^{N}\a\right)\right)+\log\G(a)-\frac{1}{2}\log 2\pi\right)\nonumber\\
&\quad+(-1)^{\frac{N+3}{2}}\sum_{j=0}^{\left\lfloor\frac{N+1}{2N}\right\rfloor}\frac{(-1)^jB_{2j}(a)B_{N+1-2Nj}}{(2j)!(N+1-2Nj)!}\a^{\frac{2j}{N+1}}\b^{N-\frac{2N^2j}{N+1}}.
\end{align}}
\end{theorem}
\begin{remark}
The transformation in \eqref{limalim} can be conceived of as a formula for $\log\G(a),$ $0<a\leq 1$. A representation for $\log\G(a)$ for $a>1$ can then be obtained by replacing $a$ by its fractional part $\{a\}$ in \eqref{limalim} and then making use of the fact that $\log\G(\{a\})=\log\G(a)-\sum_{\ell=1}^{\lfloor a\rfloor}\log(a-\ell)$, which, in turn, can be proved using the functional equation \eqref{feg}.
\end{remark} 
When $a=1$ in \eqref{limab}, one recovers Corollary 1.6 from \cite{dixitmaji1}. Further, if we let $N=1$, one obtains the well-known transformation formula for the logarithm of the Dedekind eta-function \cite[Ch. 14, Sec. 8, Cor. (ii) and Ch. 16, Entry 27(iii)]{ramnote}, \cite[p.~256]{bcbramsecnote}, \cite[p.~43]{bcbramthinote}, \cite[p.~320, Formula (29)]{lnb}:
\begin{align}\label{logdede}
\sum_{n=1}^{\infty}\frac{1}{n(e^{2n\a}-1)}-\sum_{n=1}^{\infty}\frac{1}{n(e^{2n\b}-1)}=\frac{\b-\a}{12}+\frac{1}{4}\log\left(\frac{\a}{\b}\right),
\end{align}
Note that the Dedekind eta-function $\eta(z)$ is defined for $z\in\mathbb{H}$ (upper half plane) by $\eta(z):=e^{2\pi iz/24}\prod_{n=1}^{\infty}(1-e^{2\pi inz})$, and satisfies the transformation formula \cite[p.~48]{apostol2} $\eta\left(-\frac{1}{z}\right)=\sqrt{-iz}\eta(z)$, which is equivalent to \eqref{logdede}. Thus, \eqref{limab} is a two-parameter generalization of the transformation formula for $\log\eta(z)$. 

For $0<a<1$, a vastly simplified version of Theorem \ref{dgkmord2m0} given below can be obtained.
\begin{corollary}\label{dgkmord2m0kum}
Let $0<a<1$ and $N$ be an odd positive integer. Then
{\allowdisplaybreaks\begin{align}\label{limalimsim}
&\sum_{n=1}^{\infty}\frac{\textup{exp}(-an^{N}x)}{n(1-\textup{exp}(-n^{N}x))}\nonumber\\
&=\g\left(\frac{1}{2}-a\right)-\frac{\log\left(2\sin(\pi a)\right)}{2N}+(-1)^{\frac{N+3}{2}}2^{N}\pi^{N+1}\sum_{j=0}^{\left\lfloor\frac{N+1}{2N}\right\rfloor}\left(\frac{-1}{4\pi^2}\right)^{jN}\frac{B_{2j}(a)B_{N+1-2jN}}{(2j)!(N+1-2jN)!}x^{2j-1}\nonumber\\
&\quad+\frac{(-1)^{\frac{N+3}{2}}}{N}\sum_{j=-\frac{(N-1)}{2}}^{\frac{(N-1)}{2}}(-1)^j\bigg\{\sum_{n=1}^{\infty}\frac{\cos(2\pi na)}{n\left(\textup{exp}\left(2\pi\left(\frac{2\pi n}{x}\right)^{\frac{1}{N}}e^{\frac{i\pi j}{N}}\right)-1\right)}\nonumber\\
&\quad+\frac{(-1)^{j+\frac{N+3}{2}}}{2\pi}\sum_{n=1}^{\infty}\frac{\sin(2\pi na)}{n}\left(\psi\left(i\left(\tfrac{2\pi n}{x}\right)^{\frac{1}{N}}e^{\frac{i\pi j}{N}}\right)+\psi\left(-i\left(\tfrac{2\pi n}{x}\right)^{\frac{1}{N}}e^{\frac{i\pi j}{N}}\right)\right)\bigg\}.
\end{align}}
%{\allowdisplaybreaks\begin{align}\label{limalimsim}
%&\sum_{n=1}^{\infty}\frac{\textup{exp}(-an^{N}x)}{n(1-\textup{exp}(-n^{N}x))}\nonumber\\
%&=\g\left(\frac{1}{2}-a\right)-\frac{\log\left(2\sin(\pi a)\right)}{2N}+(-1)^{\frac{N+3}{2}}2^{N}\pi^{N+1}\sum_{j=0}^{\left\lfloor\frac{N+1}{2N}\right\rfloor}\left(\frac{-1}{4\pi^2}\right)^{jN}\frac{B_{2j}(a)B_{N+1-2jN}}{(2j)!(N+1-2jN)!}x^{2j-1}\nonumber\\
%&\quad+\frac{(-1)^{\frac{N+3}{2}}}{N}\sum_{j=-\frac{(N-1)}{2}}^{\frac{(N-1)}{2}}(-1)^j\bigg\{\sum_{n=1}^{\infty}\frac{\cos(2\pi na)}{n\left(\textup{exp}\left(2A_{N,j}\left(\frac{n}{x}\right)\right)-1\right)}\nonumber\\
%&\quad\qquad\qquad\qquad+\frac{(-1)^{j+\frac{N+3}{2}}}{2\pi}\sum_{n=1}^{\infty}\frac{\sin(2\pi na)}{n}\left(\psi\left(\tfrac{i}{\pi}A_{N,j}\left(\tfrac{n}{x}\right)\right)+\psi\left(-\tfrac{i}{\pi}A_{N,j}\left(\tfrac{n}{x}\right)\right)\right)\bigg\}.
%\end{align}}
\end{corollary}
The additional parameter $a$ allows us to obtain new analogues of \eqref{logdede}, for example, the following two results. 
\begin{corollary}\label{limitingahalfN1}
For $\a, \b>0$ such that $\a\b=\pi^2$,
\begin{align*}
\sum_{n=1}^{\infty}\frac{e^{n\a}}{n\left(e^{2n\a}-1\right)}-\sum_{n=1}^{\infty}\frac{(-1)^n}{n\left(e^{2n\b}-1\right)}=-\frac{1}{2}\log 2+\frac{\a+2\b}{24}.
\end{align*}
%If we write the above theorem in the form of $\a$ and $\b$, we get the following:
%\begin{corollary}\label{limab}
%Let $0<a\leq 1$. Let $N$ be an odd positive integer. If
\end{corollary}
An equivalent form of this identity is
\begin{equation*}
\sqrt{2}e^{\frac{\a}{24}}\prod_{n=0}^{\infty}\left(1-e^{-(2n+1)\a}\right)=e^{-\frac{\b}{12}}\prod_{n=0}^{\infty}\left(1+e^{-2n\b}\right),
\end{equation*}
which draws similarity with the aforementioned transformation formula for $\eta(z)$. 
%%%%%%%%%%%%%%%%%%%%%%%%%%
%In \cite{nesterenko1}, \cite{nesterenko2}, Nesterenko proved a beautiful theorem that the numbers $\pi,e^{\pi}$ and $\G(1/4)$ are algebraically independent over $\mathbb{Q}$, and hence, in particular, that each of them is transcendental. The case $a=1/4, N=1$ of \eqref{limab} gives the following result whose special case, namely \eqref{nesttrans}, involves logarithms of these three numbers, that is, $\log(\pi), \pi$ and $\log\G\left(\frac{1}{4}\right)$.
%%%%%%%%%%%%%%%%%%%%%%%%%%%%%%%%%%%%
\begin{corollary}\label{nesterenkotype}
Let $\a,\b>0$ be such that $\a\b=\pi^2$. Then
%\begin{align}
%&\sum_{n=1}^{\infty}\frac{e^{3 n\a/2}}{n(e^{2n\a}-1)}-\frac{1}{2}\sum_{n=1}^{\infty}\frac{(-1)^n}{n\left(e^{4n\b}-1\right)}\nonumber\\
%&\quad-\frac{1}{\pi}\sum_{n=1}^{\infty}\frac{(-1)^n}{2n-1}\left\{\log\left(\tfrac{\b}{\pi}(2n-1)\right)-\frac{1}{2}\left(\psi\left(\tfrac{i\b}{\pi}(2n-1)\right)+\psi\left(-\tfrac{i\b}{\pi}(2n-1)\right)\right)\right\}\nonumber\\
%&=-\frac{3}{4}\log(2\pi)+\log\G\left(\frac{1}{4}\right)-\frac{1}{8}\log\left(\frac{\a}{\b}\right)+\frac{\a+8\b}{96}.
%\end{align}
%In particular,
%\begin{align}\label{nesttrans}
%&\frac{3}{4}\log(2\pi)-\log\G\left(\frac{1}{4}\right)+\sum_{n=1}^{\infty}\frac{e^{3 n\pi/2}}{n(e^{2n\pi}-1)}-\frac{1}{2}\sum_{n=1}^{\infty}\frac{(-1)^n}{n\left(e^{4n\pi}-1\right)}\nonumber\\
%&-\frac{1}{\pi}\sum_{n=1}^{\infty}\frac{(-1)^n}{2n-1}\left\{\log\left(2n-1\right)-\frac{1}{2}\left(\psi\left(i(2n-1)\right)+\psi\left(-i(2n-1)\right)\right)\right\}\nonumber\\
%&=\frac{3\pi}{32},
%\end{align}
%so that at least one of the five numbers in the set
%\begin{align}
%&\bigg\{\log(2\pi),\hspace{1mm}\log\G\left(\frac{1}{4}\right),\hspace{1mm}\sum_{n=1}^{\infty}\frac{e^{3 n\pi/2}}{n(e^{2n\pi}-1)},\hspace{1mm}\sum_{n=1}^{\infty}\frac{(-1)^n}{n\left(e^{4n\pi}-1\right)},\nonumber\\
%&\sum_{n=1}^{\infty}\frac{(-1)^n}{2n-1}\left(\log\left(2n-1\right)-\frac{1}{2}\left(\psi\left(i(2n-1)\right)+\psi\left(-i(2n-1)\right)\right)\right)\bigg\}
%\end{align}
%is transcendental.
\begin{align}\label{gammalogirr}
&-\frac{\g}{4}+\sum_{n=1}^{\infty}\frac{e^{3 n\a/2}}{n(e^{2n\a}-1)}-\frac{1}{2}\sum_{n=1}^{\infty}\frac{(-1)^n}{n\left(e^{4n\b}-1\right)}+\frac{1}{2\pi}\sum_{n=1}^{\infty}\frac{(-1)^n}{2n-1}\left(\psi\left(\tfrac{i\b}{\pi}(2n-1)\right)+\psi\left(-\tfrac{i\b}{\pi}(2n-1)\right)\right)\nonumber\\
&=-\frac{1}{4}\log 2+\frac{\a+8\b}{96}.
\end{align}
In particular,
\begin{align}\label{gammalogirrpi}
&-\frac{\g}{4}+\sum_{n=1}^{\infty}\frac{e^{3 n\pi/2}}{n(e^{2n\pi}-1)}-\frac{1}{2}\sum_{n=1}^{\infty}\frac{(-1)^n}{n\left(e^{4n\pi}-1\right)}+\frac{1}{2\pi}\sum_{n=1}^{\infty}\frac{(-1)^n}{2n-1}\left(\psi\left(i(2n-1)\right)+\psi\left(-i(2n-1)\right)\right)\nonumber\\
&=-\frac{1}{4}\log 2+\frac{3\pi}{32}.
\end{align}
\end{corollary}
Equation \eqref{gammalogirr} readily gives the following results on Euler's constant.
\begin{corollary}\label{irreuler1}
Let $\a,\b>0$ such that $\a\b=\pi^2$. If $\a, \b$ and $\log 2$ are linearly independent over $\mathbb{Q}$, at least one of
\begin{align*}
\g, \hspace{1mm}\sum_{n=1}^{\infty}\frac{e^{3 n\a/2}}{n(e^{2n\a}-1)},\hspace{1mm}\sum_{n=1}^{\infty}\frac{(-1)^n}{n\left(e^{4n\b}-1\right)},\text{and}\hspace{1mm}\frac{1}{2\pi}\sum_{n=1}^{\infty}\frac{(-1)^n}{2n-1}\left(\psi\left(\tfrac{i\b}{\pi}(2n-1)\right)+\psi\left(-\tfrac{i\b}{\pi}(2n-1)\right)\right)
\end{align*}
is irrational. 
\end{corollary}
\begin{corollary}\label{irreuler2}
At least one of the numbers
\begin{align*}
\g, \hspace{1mm}\sum_{n=1}^{\infty}\frac{e^{3 n\pi/2}}{n(e^{2n\pi}-1)},\hspace{1mm}\sum_{n=1}^{\infty}\frac{(-1)^n}{n\left(e^{4n\pi}-1\right)},\text{and}\hspace{1mm}\frac{1}{2\pi}\sum_{n=1}^{\infty}\frac{(-1)^n}{2n-1}\left(\psi\left(i(2n-1)\right)+\psi\left(-i(2n-1)\right)\right)
\end{align*}
is irrational.
\end{corollary}
So far we have discussed transformations of the series $\sum_{n=1}^{\infty}n^{N-2h}\frac{\textup{exp}(-an^{N}x)}{1-\textup{exp}(-n^{N}x)}$ for $h\geq N/2$. Our aim is to now consider the case when $h<N/2$. When $N$ is even, we are able to transform the series for any integer value of $h<N/2$. However, when $N$ is odd, we succeed in obtaining a transformation only for $0\leq h<N/2$ as the series consisting of $\sin(2\pi n a)$, logarithm and digamma functions in the summand does not converge for $h<0$.
\begin{theorem}\label{dgkmgenh0Nby2}
Let $N$ be a positive integer and $h$ be a positive integer such that $0\leq h<N/2$. Let $x>0$ and $0<a\leq 1$.

\textup{(i)} Let $N$ be odd and $S(x, a)$ be defined as in \eqref{sodda}. If $g(N, h, a)$ is defined by
\begin{equation}\label{gnha}
g(N, h, a):=\begin{cases}
-\frac{1}{2}\zeta(-N+2h),\hspace{2mm}\text{if}\hspace{2mm} a=1,\\
\hspace{9mm}0,\hspace{15mm}\text{if}\hspace{2mm}0<a<1. 
\end{cases}
\end{equation}
%g(N, h, a)
then
\begin{align}\label{oddh0Nby2}
\sum_{n=1}^{\infty}n^{N-2h}\frac{\textup{exp}(-an^{N}x)}{1-\textup{exp}(-n^{N}x)}&=\frac{\zeta(2h)}{x}+\frac{1}{N}\G\left(\frac{N-2h+1}{N}\right)\zeta\left(\frac{N-2h+1}{N},a\right)x^{-\frac{(N-2h+1)}{N}}\nonumber\\
&\quad+S(x,a)+g(N, h, a).
%%&\quad+\frac{(-1)^{h+\frac{N+3}{2}}}{\pi N}\left(\frac{2\pi}{x}\right)^{\frac{N-2h+1}{N}}\sum_{j=-\frac{(N-1)}{2}}^{\frac{(N-1)}{2}}(-1)^{j}\textup{exp}\left(\frac{i\pi(1-2h)j}{N}\right)\nonumber\\
%%&\quad\quad\times\sum_{n=1}^{\infty}\frac{\sin(2\pi na)}{n^{\frac{2h-1}{N}}}T(N, h, x, j),
\end{align}
%\begin{align}\label{oddh0nby2}
%\sum_{n=1}^{\infty}n^{n-2h}\frac{\textup{exp}(-an^{n}x)}{1-\textup{exp}(-n^{n}x)}&=-\left(a-\frac{1}{2}\right)\zeta(-n+2h)+\frac{\zeta(2h)}{x}\nonumber\\
%&\quad+\frac{1}{n}\g\left(\frac{n-2h+1}{n}\right)\zeta\left(\frac{n-2h+1}{n},a\right)x^{-\frac{(n-2h+1)}{n}}+s(x,a)\nonumber\\
%&\quad+\frac{(-1)^{h+\frac{n+3}{2}}}{\pi n}\left(\frac{2\pi}{x}\right)^{\frac{n-2h+1}{n}}\sum_{j=-\frac{(n-1)}{2}}^{\frac{(n-1)}{2}}(-1)^{j}\textup{exp}\left(\frac{i\pi(1-2h)j}{n}\right)\nonumber\\
%&\quad\quad\times\sum_{n=1}^{\infty}\frac{\sin(2\pi na)}{n^{\frac{2h-1}{n}}}t(n, h, x, j),
%\end{align}
%where
%\begin{equation}\label{t}
%\end{equation}
%t(n, h, x, j):=2\displaystyle\sum_{k=1}^{\frac{n-2h+1}{2}}\frac{\g(2k)\zeta(2k)}{\left(2\pi\left(\frac{2\pi n}{x}\right)^{1/n}e^{\frac{i\pi j}{n}}\right)^{2k}}.

\textup{(ii)} If $N$ is even and $S(x, a)$ is defined as in \eqref{sevena}, then
\begin{align}\label{evenh0Nby2}
\sum_{n=1}^{\infty}n^{N-2h}\frac{\textup{exp}(-an^{N}x)}{1-\textup{exp}(-n^{N}x)}&=\frac{\zeta(2h)}{x}+\frac{1}{N}\G\left(\frac{N-2h+1}{N}\right)\zeta\left(\frac{N-2h+1}{N},a\right)x^{-\frac{(N-2h+1)}{N}}\nonumber\\
&\quad+S(x,a).
\end{align}
In addition, \eqref{evenh0Nby2} holds also when $h<0$.
\end{theorem}
\begin{remark}
The method described in Remark \textup{\ref{ag1}} for extending the formula in Theorem \textup{\ref{dgkmgen}} to $a>1$ applies to the above theorem as well.
\end{remark}
\begin{remark}
Note that the right-hand side of \eqref{evenh0Nby2} is exactly the same as that of \eqref{dgkmgeneqn} since for $h<N/2$, $N$ even, the term $-\left(a-\tfrac{1}{2}\right)\zeta(-N+2h)$ as well as the two finite sums in $P(x, a)$ vanish.
\end{remark}

Kanemitsu, Tanigawa and Yoshimoto \cite{ktyhr} have obtained the above result for $a=1$.

We now give a special case of part (i) of the above theorem.
\begin{corollary}\label{hzerospl}
Let $0<a\leq 1$. Let $g(N, h, a)$ be defined in \eqref{gnha}. 
%\begin{align}\label{ramgennnn}
%\sum_{n=1}^{\infty}\frac{ne^{-anx}}{1-e^{-nx}}&=\frac{\psi'(a)}{x^2}-\frac{1}{2x}+\frac{1}{12}\left(a-\frac{1}{2}\right)-\frac{4\pi^2}{x^2}\sum_{n=1}^{\infty}\frac{n\cos(2\pi n a)}{\exp{(\frac{4\pi^2n}{x})}-1}\nonumber\\
%&\quad+\frac{4\pi}{x^2}\sum_{n=1}^{\infty}n\sin(2\pi na)\left\{\log\left(\frac{2\pi n}{x}\right)-\frac{1}{2}\left(\psi\left(\frac{2\pi in}{x}\right)+\psi\left(\frac{-2\pi in}{x}\right)\right)+\frac{x^2}{48\pi^2n^2}\right\}.
%\end{align}
If $\a$ and $\b$ are two positive numbers such that $\a\b=\pi^2$, then
\begin{align}\label{ramgennn}
&\a\sum_{n=1}^{\infty}\frac{ne^{2n\a(1-a)}}{e^{2n\a}-1}+\b\sum_{n=1}^{\infty}\frac{n\cos(2\pi n a)}{e^{2n\b}-1}\nonumber\\
&=\a\hspace{0.5mm} g(1, 0, a)+\frac{\psi'(a)}{4\a}-\frac{1}{4}+\frac{\b}{\pi}\sum_{n=1}^{\infty}n\sin(2\pi na)\left\{\log\left(\frac{n\b}{\pi}\right)-\frac{1}{2}\left(\psi\left(\frac{in\b}{\pi}\right)+\psi\left(\frac{-in\b}{\pi}\right)\right)\right\}.
\end{align}
\end{corollary}
When $a=1$, \eqref{ramgennn} gives a result of Schl\"{o}milch \cite{schlomilch}, rediscovered by Ramanujan \cite[Ch. 14, Sec. 8, Cor. (i)]{ramnote}, \cite[p.~318, formula (23)]{lnb}:
\begin{equation*}
\a\sum_{n=1}^{\infty}\frac{n}{e^{2n\a}-1}+\b\sum_{n=1}^{\infty}\frac{n}{e^{2n\b}-1}=\frac{\a+\b}{24}-\frac{1}{4}.
\end{equation*}
Let $q=e^{2\pi iz}, z\in\mathbb{H}$. Then the analytic continuation of the above formula for Re$(\a)>0$, Re$(\b)>0$ is equivalent to the transformation formula for the Eisenstein series $E_2(z):=1-24\sum_{n=1}^{\infty}\frac{nq^n}{1-q^{n}}$, namely, $E_2\left(\frac{-1}{z}\right)=z^2E_2(z)+\frac{6z}{\pi i}$.

Two new corollaries of \eqref{ramgennn} are now given.
\begin{corollary}\label{befkon}
\begin{equation}\label{befkong}
\sum_{n=1}^{\infty}\frac{ne^{n\pi}}{e^{2\pi n}-1}+\sum_{n=1}^{\infty}\frac{n(-1)^n}{e^{2\pi n}-1}=\frac{1}{8}-\frac{1}{4\pi}.
\end{equation}
\end{corollary}
Note that \eqref{befkong} is an analogue of the following famous result, first proved by Schl\"{o}milch \cite{schlomilch} (see \cite[p.~159]{berndtrocky} for more references): 
\begin{equation*}
\sum_{n=1}^{\infty}\frac{n}{e^{2n\pi}-1}=\frac{1}{24}-\frac{1}{8\pi}.
\end{equation*}
\begin{corollary}\label{catalan}
If $G$ denotes Catalan's constant, then
\begin{align*}
&\sum_{n=1}^{\infty}\frac{ne^{\frac{3n\pi}{2}}}{e^{2n\pi}-1}+2\sum_{n=1}^{\infty}\frac{n(-1)^{n}}{e^{4n\pi}-1}\nonumber\\
&=\frac{2G}{\pi^2}+\frac{1}{4}\left(1-\frac{1}{\pi}\right)+\frac{1}{\pi}\sum_{n=1}^{\infty}(-1)^{n-1}(2n-1)\left\{\log(2n-1)-\frac{1}{2}\left(\psi(i(2n-1))+\psi(-i(2n-1))\right)\right\}.
\end{align*}
\end{corollary}
%A spectacular cancellation happens between the terms in part (i) of Theorem \ref{dgkmgenh0Nby2} when we assume $a\neq 1$, that is, $0<a<1$. This simplification of Theorem \ref{dgkmgenh0Nby2} along with the simplified forms of Corollaries \ref{hzero} and \ref{hzerospl} are respectively derived in Theorem \ref{dgkmgenh0Nby2sim} in Section \ref{0hNby2} below and Corollaries \ref{hzerosim} and \ref{hzerosplsim} that follow it.
A counterpart of Theorem \ref{ggram}, which is just a reformulation of Theorem \ref{dgkmgen} for $N$ even, is now given in terms of $\a$ and $\b$.
%\section{The case $N$ even}\label{niseven}
\begin{theorem}\label{nevena}
Let $N$ be an even positive integer and $m$ be any integer. Let $0<a\leq 1$. For any $\a,\b>0$ satisfying $\a\b^N=\pi^{N+1}$,
\begin{align}\label{neveneqna}
&\a^{-\left(\frac{2Nm-1}{N+1}\right)}\bigg(\left(a-\tfrac{1}{2}\right)\zeta(2Nm)+\sum_{j=1}^{m}\frac{B_{2j+1}(a)}{(2j+1)!}\zeta(2N(m-j))(2^N\a)^{2j}+\sum_{n=1}^{\infty}\frac{n^{-2Nm}\textup{exp}\left(-a(2n)^{N}\a\right)}{1-\exp{\left(-(2n)^{N}\a\right)}}\bigg)\nonumber\\
&=\b^{-\left(\frac{2Nm-1}{N+1}\right)}\frac{2^{2Nm-1}}{N}\Bigg\{\pi^{-\left(\frac{1-2Nm}{N}\right)}\G\left(\frac{1-2Nm}{N}\right)\zeta\left(\frac{1-2Nm}{N},a\right)\nonumber\\
&\quad-2(-1)^{\frac{N}{2}+m}2^{\frac{1-2Nm}{N}}\sum_{j=0}^{\frac{N}{2}-1}(-1)^j\Bigg[\sum_{n=1}^{\infty}\frac{\cos(2\pi na)}{n^{2m+1-\frac{1}{N}}}\textup{Im}\Bigg(\frac{e^{\frac{i\pi(2j+1)}{2N}}}{\exp{\left((2n)^{\frac{1}{N}}\b e^{\frac{i\pi(2j+1)}{2N}}\right)}-1}\Bigg)\nonumber\\
&\qquad\qquad\qquad\qquad\qquad\qquad+(-1)^{j+\frac{N}{2}+1}\sum_{n=1}^{\infty}\frac{\sin(2\pi na)}{n^{2m+1-\frac{1}{N}}}\textup{Re}\Bigg(\frac{e^{\frac{i\pi(2j+1)}{2N}}}{\exp{\left((2n)^{\frac{1}{N}}\b e^{\frac{i\pi(2j+1)}{2N}}\right)}-1}\Bigg)\Bigg]\Bigg\}\nonumber\\
&\quad+(-1)^{\frac{N}{2}+1}2^{2Nm-1}\sum_{j=0}^{m}\frac{B_{2j}(a)B_{(2m+1-2j)N}}{(2j)!((2m+1-2j)N)!}\a^{\frac{2j}{N+1}}\b^{N+\frac{2N^2(m-j)-N}{N+1}}.
\end{align}
\end{theorem}
When $a=1$, we recover Theorem 1.10 from \cite{dixitmaji1}, which itself is a generalization of Wigert's formula \cite[pp.~8-9, Equation (5)]{wig}, \cite[Equation (1.2)]{dixitmaji1}.
% which can also be obtained by letting $h=N/2$ in Theorem \ref{dgkmgen}. 

When $a=1/2$, Theorem \ref{nevena} gives the following result on transcendence.
\begin{corollary}\label{transzetaeven}
Let $N$ be a positive even integer and $m$ be any integer. Then at least one of the numbers
\begin{equation*}
\zeta\left(2m+1-\frac{1}{N}\right), \sum_{n=1}^{\infty}\frac{n^{-2Nm}\textup{exp}\left(-\frac{1}{2}(2n)^{N}\pi\right)}{1-\exp{\left(-(2n)^{N}\pi\right)}},\hspace{2mm}\text{and}\hspace{2mm} \sum_{n=1}^{\infty}\frac{(-1)^n}{n^{2m+1-\frac{1}{N}}}\textup{Im}\Bigg(\frac{e^{\frac{i\pi(2j+1)}{2N}}}{\exp{\left((2n)^{\frac{1}{N}}\pi e^{\frac{i\pi(2j+1)}{2N}}\right)}-1}\Bigg),
\end{equation*}
where $j$ takes every value between $0$ and $\frac{N}{2}-1$, is transcendental.
\end{corollary}
The interesting result we now give is reminiscent of the corrected version of Klusch's formula given by Kanemitsu, Tanigawa and Yoshimoto in \cite[Proposition 1.1]{ktyacta} but is actually very different in nature from the latter.
\begin{theorem}\label{klusch}
Let $\a, \b$ be positive numbers such that $\a\b=4\pi^3$. For $0<a<1$, we have
\begin{align}\label{kluscheqn}
\sum_{n=1}^{\infty}\frac{\exp{(-an^2\a)}}{1-\exp{(-n^2\a)}}=\frac{1}{2}\left(a-\frac{1}{2}\right)+\frac{\pi^2}{6\a}&+\frac{\sqrt{\pi}}{2\sqrt{\a}}\bigg\{\sum_{n=1}^{\infty}\frac{\cos(2\pi na)}{\sqrt{n}}\left(\frac{\sinh(\sqrt{n\b})-\sin(\sqrt{n\b})}{\cosh(\sqrt{n\b})-\cos(\sqrt{n\b})}\right)\nonumber\\
&\quad+\sum_{n=1}^{\infty}\frac{\sin(2\pi na)}{\sqrt{n}}\left(\frac{\sinh(\sqrt{n\b})+\sin(\sqrt{n\b})}{\cosh(\sqrt{n\b})-\cos(\sqrt{n\b})}\right)\bigg\}.
\end{align}
In particular, when $a=1/2$,
\begin{align}\label{kluscheqnhalf}
\sum_{n=1}^{\infty}\textup{cosech}\left(\frac{n^2\a}{2}\right)=\frac{\pi^2}{3\a}+\sqrt{\frac{\pi}{\a}}\sum_{n=1}^{\infty}\frac{(-1)^n}{\sqrt{n}}\left(\frac{\sinh(\sqrt{n\b})-\sin(\sqrt{n\b})}{\cosh(\sqrt{n\b})-\cos(\sqrt{n\b})}\right).
\end{align}
\end{theorem}
This paper is organized as follows. In Section \ref{prelim}, we collect preliminary results to be used in the sequel. Section \ref{secraabe} is devoted to finding new properties of the Raabe integral $\mathfrak{R}(y,w)$ and to proving Theorem \ref{raabesum}. 
%As claimed in the introduction, we also give various alternate representations for the series on the left side of \eqref{raabesumeqn} for an interested reader.
In Section \ref{hgNby2}, we prove Theorem \ref{dgkmgen} and also obtain, as its special case, Theorem 2.1 of Kanemitsu, Tanigawa and Yoshimoto from \cite{ktyacta}. We derive Theorem \ref{dgkmord2} and its special case Theorem \ref{ggram} in Section \ref{ord2}. The special cases of these theorems when $a$ takes values such as $\frac{1}{2}, \frac{1}{4}$ etc. are given in two separate sub-sections of this section. These include Corollary \ref{zeta311}, Theorem \ref{ggramhalf} and Corollaries \ref{transzetaodd}-\ref{criterionzeta}. Section \ref{limiting} is devoted to the proof of Theorem \ref{dgkmord2m0} and to the proofs of Corollaries \ref{dgkmord2m0kum}-\ref{irreuler2} that result from it. Theorem \ref{dgkmgenh0Nby2} and its Corollaries \ref{hzerospl}-\ref{catalan} are proved in Section \ref{0hNby2}. We prove Theorem \ref{nevena}, Corollary \ref{transzetaeven} and Theorem \ref{klusch} in Section \ref{nevenasec}. We end the paper with some concluding remarks in Section \ref{con}. The numerical verification of each of Theorems \ref{dgkmgen}, \ref{dgkmord2} and \ref{dgkmord2m0} is done in Tables \ref{t1}, \ref{t2} and \ref{t3} respectively.  

\section{Preliminaries}\label{prelim}
The functional equation, the reflection formula (along with a
variant), and Legendre's duplication formula for the Gamma function
$\G(s)$ are given by
{\allowdisplaybreaks\begin{align}
 \G(s+1)&=s\G(s),\label{feg}\\
 \G(s)\G(1-s)&=\frac{\pi}{\sin(\pi s)}\hspace{4mm}(s\notin\mathbb{Z}),\\
 \G\left(\frac{1}{2}+s\right)\G\left(\frac{1}{2}-s\right)&=\frac{\pi}{\cos(\pi s)}\hspace{4mm}(s-\tfrac{1}{2}\notin\mathbb{Z}),\\
 \G(s)\G\left(s+\frac{1}{2}\right)&=\frac{\sqrt{\pi}}{2^{2s-1}}\G(2s).\label{dup}
\end{align}}%
The inverse Mellin transform of the gamma function for $c=$ Re$(s)>0$ and Re$(y)>0$ is well-known:
\begin{align}\label{gammamelli}
\frac{1}{2\pi i}\int_{(c)}\G(s)y^{-s}\, \mathrm{d}s=e^{-y}.
\end{align}
Here, and throughout the sequel, we use $\int_{(c)}$ to denote $\int_{c-i\infty}^{c+i\infty}$. Stirling's formula on a vertical strip states that if $s=\sigma+it$, then for $a\leq\sigma\leq b$ and $|t|\geq 1$,
\begin{equation}\label{strivert}
|\Gamma(s)|=(2\pi)^{\tfrac{1}{2}}|t|^{\sigma-\tfrac{1}{2}}e^{-\tfrac{1}{2}\pi |t|}\left(1+O\left(\frac{1}{|t|}\right)\right)
\end{equation}
as $t\to\infty$. The digamma function $\psi(z)$ satisfies the functional equation \cite[p.~54]{temme}
\begin{align}\label{psifunc}
\psi(z+1)=\psi(z)+\frac{1}{z}.
\end{align}
From \cite[p.~259, formula 6.3.18]{as}, for $|\arg$ $z|<\pi$, as
$z\to\infty$,
\begin{equation}\label{psiasymp}
\psi(z) \sim \log z-\frac{1}{2z}-\frac{1}{12z^{2}}+\frac{1}{120z^{4}}-\frac{1}{252z^{6}}+\cdots.
\end{equation}
Throughout the paper, we use, without mention, Euler's formula \cite[p.~5, Equation (1.14)]{temme}
\begin{equation}\label{ef}
\zeta(2 m ) = (-1)^{m +1} \frac{(2\pi)^{2 m}B_{2 m }}{2 (2 m)!}.
\end{equation}
The functional equation of $\zeta(s)$ in the asymmetric form is given by \cite[p.~259]{apostol-1998a}
\begin{equation}\label{zetafe1}
\zeta(s)=2^s\pi^{s-1}\Gamma(1-s)\zeta(1-s)\sin\left(\tfrac{1}{2}\pi s\right).
\end{equation}
We now state a generalization of Poisson's summation formula due to Guinand \cite[Theorem 1]{apg1} which is crucial in the proof of Theorem \ref{raabesum}.
\begin{theorem}\label{Guind-pois}
\textit{If $f(x)$ is an integral, $f(x)$ tends to zero as $x\rightarrow\infty$, and $xf'(x)$ belongs to $L^p(0,\infty)$, for some p, $1<p\leq 2$, then}
\begin{align*}
\lim_{M\rightarrow\infty}\left(\sum_{m=1}^M f(m)-\int_0^M f(v)\, \mathrm{d}v\right)=\lim_{M\rightarrow\infty}\left(\sum_{m=1}^M g(m)-\int_0^M g(v)\, \mathrm{d}v\right),
\end{align*}
where
\begin{align*}
g( x )=2\int_0^{\rightarrow \infty}f(t)\cos(2\pi x t)\, \mathrm{d}t.
\end{align*}
\end{theorem}
\section{Some results on Raabe's integral}\label{secraabe}
The left side of \eqref{raabesumeqn} is an infinite series whose summands are the Raabe integrals defined in \eqref{raabe}. In order to prove Theorem \ref{raabesum} one cannot interchange the order of summation and integration in this series since that leads to a divergent integral. A version of the classical Poisson summation formula \cite[pp.~60-61]{titchfou} states that if $f(t)$ is continuous and of bounded variation on $[0,\infty)$, and if $\int_{0}^{\infty}f(t)\, \mathrm{d}t$ exists, then
\begin{equation*}
\frac{1}{2}f(0)+\sum_{m=1}^{\infty}f(m)=\int_{0}^{\infty}f(t)\, \mathrm{d}t+2\sum_{m=1}^{\infty}\int_{0}^{\infty}f(t)\cos(2\pi mt)\, \mathrm{d}t.
\end{equation*}
The desired series of which we would like to obtain a closed form is the one on the right side of the above equation with $f(t)=\frac{4\pi^2t}{4\pi^2t^2+u^2}$, as can be easily seen by a simple change of variable. Unfortunately this formula is also inapplicable towards proving Theorem \ref{raabesum} because the hypothesis that $\int_{0}^{\infty}f(t)\, \mathrm{d}t$ be convergent is not satisfied. The idea is to use Guinand's generalization of Poisson's summation formula, that is, Theorem \ref{Guind-pois}.

However, before using Theorem \ref{Guind-pois}, it is imperative to obtain some results on Raabe's integral. 
%These are then used along with an application of Theorem \ref{Guind-pois} to prove Theorem \ref{raabesum}.
We begin with the following identity which readily depicts the asymptotic behavior of the Raabe integral for positive small values of $y$.
\begin{lemma}\label{lemsix}
For $y>0$ and \textup{Re}$(w)>0$, the following identity holds:
\begin{equation}\label{six}
\mathfrak{R}(y, w)=\sum_{k=0}^{\infty}\frac{(wy)^{2k}}{(2k)!}\left(\psi(2k+1)-\log(wy)\right).
\end{equation}
In particular, as $y\to 0^{+}$,
\begin{equation}\label{rabasysmall}
\mathfrak{R}(y, w)\sim-\g-\log(wy).
\end{equation}
\end{lemma}
\begin{proof}
First let $w>0$. From \cite[p.~428, Formula \textbf{3.723.5}]{grn},
\begin{equation}\label{one}
\mathfrak{R}(y, w)=-\frac{1}{2}\left(e^{-wy}\overline{\textup{Ei}}(wy)+e^{wy}\textup{Ei}(-wy)\right),
\end{equation}
where $\textup{Ei}(x)$ is the exponential integral defined for $x>0$ by \cite[p.~1]{je} $\textup{Ei}(-x)=-\int_{x}^{\infty}e^{-t}/t\, \mathrm{d}t$. Thus the exponential integral function is related to logarithmic integral $\textup{li}(x)=\int_{0}^{x}\frac{\mathrm{d}t}{\log(t)}$ by
\begin{equation}\label{two}
\textup{Ei}(-x)=\textup{li}(e^{-x}).
\end{equation}
Also \cite[p.~3]{je},
\begin{equation}\label{three}
\overline{\textup{Ei}}(x)=\textup{li}(e^{x}).
\end{equation}
Thus from \eqref{one}-\eqref{three}, we see that
\begin{equation}\label{four}
\mathfrak{R}(y, w)=-\frac{1}{2}\left(e^{-wy}\textup{li}\left(e^{wy}\right)+e^{wy}\textup{li}\left(e^{-wy}\right)\right).
\end{equation}
Now Dixon and Ferrar \cite[p.~165, Equation (5.4)]{dixfer3} have proved that
\begin{equation}\label{five}
e^{x}\textup{li}(e^{-x})+e^{-x}\textup{li}(e^x)=\pi^{3/2}K_{/\frac{1}{2}}(x),
\end{equation}
where \cite[Equation (3.12)]{dixfer1}
\begin{equation}\label{klnu}
K_{/\nu}(z)=\frac{1}{\pi}\sum_{k=0}^{\infty}\frac{(z/2)^{2k}}{\G(k+1)\G(k+\nu+1)}\left\{2\log(z/2)-\psi(k+1)-\psi(k+\nu+1)\right\}.
\end{equation}
Now let $\nu=1/2$ in \eqref{klnu}, use \eqref{dup} and then combine the resulting identity with \eqref{four} and \eqref{five} to obtain \eqref{six} for $w>0$. Since both sides of \eqref{six} are analytic for Re$(w)>0$, we obtain \eqref{six} in this region by the principle of analytic continuation. To prove \eqref{rabasysmall}, divide both sides of \eqref{six} by the first term of the right-hand side and note that $\psi(1)=-\g$ as well as $\lim_{y\to 0^{+}}(\psi(2k+1)-\log(wy))/(-\g-\log(wy))=0$ for $k\geq 1$.
\end{proof}

\textbf{Second proof.} For $ 0 < $ Re$(s) < 1+ 2 $Re$(\nu) $,  we have \cite[p.~43, Formula (5.10)]{ober}
\begin{align*}
\int_{0}^\infty  \frac{ \cos( y t  )}{(w^2 + t^2)^{ \nu}} t^{s-1} \mathrm{d}t & = \frac{w^{s- 2 \nu}}{2} B\Big( \frac{s}{2}, \nu - \frac{s}{2} \Big){}_1F_2\left(\frac{s}{2}; 1 - \nu + \frac{s}{2}, \frac{1}{2}; \frac{w^2 y^2}{4}   \right) \\
& \quad+ \frac{\sqrt{\pi}}{2} \Big(\frac{y}{2} \Big)^{ 2 \nu -s} \frac{\Gamma\left(\frac{s}{2}- \nu  \right)}{\Gamma\left(\frac{1}{2} + \nu - \frac{s}{2}  \right) } {}_1F_2\left(\nu; \frac{1}{2} + \nu - \frac{s}{2} , 1 + \nu - \frac{s}{2} ; \frac{w^2 y^2}{4}   \right),
\end{align*}
where $B(z_1,z_2):=\frac{\G(z_1)\G(z_2)}{\G(z_1+z_2)}$ is Euler's beta function. Let $\nu=1$ so that for $0<$Re$(s)<3$, we have \footnote{There is a typo in this formula stated in \cite[p.~43, Formula 5.8]{ober} in that $b^{-z}$ should be replaced by $b^{2-z}$.}
\begin{align}\label{Oberhettinger 5.8p}
\int_{0}^\infty  \frac{ \cos( y t  )}{w^2 + t^2} t^{s-1} \mathrm{d}t & = \frac{\sqrt{\pi}}{2 } \Big( \frac{y}{2}\Big)^{2-s} \frac{\Gamma(\frac{s}{2} - 1 )}{\Gamma( \frac{3}{2} - \frac{s}{2} )} {}_1 F_2\left(1; \frac{3}{2} - \frac{s}{2}, 2 - \frac{s}{2}; \frac{ w^2 y^2}{4} \right)\nonumber\\
&\quad+ \frac{1}{2}\pi w^{s-2}\mathrm{cosec}\Big( \frac{\pi s}{2} \Big) \cosh( w y).
\end{align}
Note that the following expansions are valid as $z\to 0$:
\begin{align}
\Gamma(z) & = \frac{1}{z} - \gamma + \frac{1}{2}\Big( \gamma^2 + \frac{\pi^2}{6} \Big)z + O(z^2),\label{asymptotic of gamma} \\
\mathrm{cosec}(z) & = \frac{1}{z} + \frac{z}{6} + \frac{7}{360} z^3 + O(z^5),  \label{cosec z expansion}\\
(a + z)_n & = (a)_n \Big( 1+ \{ \psi(a + n) - \psi(a) \}z + O(z^2)  \Big),  \label{assymptotic expansion of Pochhammer symbol}
\end{align}
where $(a)_n:=a(a+1)\cdots(a+n-1)$ is the Pochhammer symbol. Note that \eqref{assymptotic expansion of Pochhammer symbol} implies
\begin{align}\label{asymptotic}
\frac{1}{ (1-2z)_{2k} } = \frac{1}{(2k)!} \left( 1 + 2z \Big( \psi(2k+1) -\psi(1) \Big) + O(z^2) \right).
\end{align}
as $z\to 0$. Since ${}_1 F_2\left(1; \frac{1}{2} - z, 1 -z ; \frac{ w^2 y^2}{4} \right)=\sum_{ k=0}^{\infty} \frac{1}{(1-2z)_{2k}} (wy)^{2k}$ is uniformly convergent on $|z|\leq r_1<1$, employing \eqref{asymptotic} leads to 
\begin{align}\label{Final expansion of 1F2}
{}_1 F_2\left(1; \frac{1}{2} - z, 1 -z ; \frac{ w^2 y^2}{4} \right) = \sum_{ k=0}^{\infty} \frac{(wy)^{2k}}{(2k)!} \left( 1 + 2z \Big( \psi(2k+1) -\psi(1) \Big) + O(z^2) \right). 
\end{align}
Letting $s=2z+2$ in the second step below, and then invoking \eqref{asymptotic of gamma}, \eqref{cosec z expansion} and \eqref{Final expansion of 1F2}, we see that
{\allowdisplaybreaks\begin{align*}
&\int_{0}^\infty  \frac{ t\cos( y t  )}{w^2 + t^2}\mathrm{d}t\nonumber\\
&= \lim_{s\to 2}\Bigg\{\frac{\sqrt{\pi}}{2 } \Big( \frac{y}{2}\Big)^{2-s} \frac{\Gamma(\frac{s}{2} - 1 )}{\Gamma( \frac{3}{2} - \frac{s}{2} )} {}_1 F_2\left(1; \frac{3}{2} - \frac{s}{2}, 2 - \frac{s}{2}; \frac{ w^2 y^2}{4} \right)+\frac{1}{2}\pi w^{s-2}\mathrm{cosec}\Big( \frac{\pi s}{2} \Big) \cosh( w y)\Bigg\}\nonumber\\
&=\lim_{z\to 0}\Bigg\{e^{-2 z \log( y)}\Gamma(2z) \cos
(\pi z)  {}_1 F_2\left(1; \frac{1}{2} - z, 1 -z ; \frac{ w^2 y^2}{4} \right)-\frac{\pi}{2} e^{2 z \log(w)} \mathrm{cosec}( \pi z )   \cosh( w y) \Bigg\}\nonumber\\
&= \lim_{z\to 0} \Bigg[ \Big( 1 - 2 z \log(y) + 2 z^2 \log^2(y ) + \cdots \Big)\Big( \frac{1}{2z} - \gamma + \Big( \gamma^2 + \frac{\pi^2}{6} \Big)z + \cdots  \Big) \nonumber \\
&\qquad\qquad \times \Big( 1- \frac{(\pi z)^2}{2!} + \cdots \Big) \left( \sum_{ k=0}^{\infty} \frac{(wy)^{2k}}{(2k)!} \left( 1 + 2z \Big( \psi(2k+1) -\psi(1) \Big) + \cdots \right) \right)  \nonumber \\
&\qquad\qquad- \Big( \frac{1}{2z} + \log(w) + \Big( \frac{\pi^2  }{12} +  \log^2(w) \Big) z + \cdots   \Big) \cosh(wy)\Bigg]\nonumber\\
&=\sum_{k=0}^{\infty} \frac{(w y)^{2k}}{(2k)!} \Big(\psi(2k+1) - \log(w y) \Big).
\end{align*}}
\qed

Next, the asymptotic expansion of the Raabe integral for large values of $y$ is obtained.
\begin{lemma}
Let \textup{Re}$(w)>0$. As $y\to\infty$,
\begin{equation}\label{rabasylarge}
\mathfrak{R}(y, w)\sim-\sum_{n=1}^{\infty}\frac{(2n-1)!}{w^{2n}y^{2n}}.
\end{equation}
\end{lemma}
\begin{proof}
Here we use the analogue of Watson's lemma for Laplace transform in the setting of Fourier transforms \cite{olver1974}, \cite[Equations (1.3), (1.4)]{dainaylor}. It says that if the form of $h(t)$ near $t=0$ is given as a series of algebraic powers, that is,
\begin{equation}\label{algh}
h(t)\sim\sum_{n=0}^{\infty}b_nt^{n+\l-1}
\end{equation} 
as $t\to 0^{+}$, then under certain restrictions on $h$ (see \cite{olver1974}, \cite[Section 2]{dainaylor} for the same),
\begin{align}\label{olvthm}
\int_{0}^{\infty}e^{ist}h(t)\, \mathrm{d}t\sim \sum_{n=0}^{\infty}b_ne^{i(n+\l)\pi/2}\G(n+\l)s^{-n-\l}
\end{align}
as $s\to\infty$.
Let $h(t)=t/(t^2+w^2)$. Then it is easy to see that $h(t)$ satisfies \eqref{algh} with $\l=1$ and
\begin{equation*}
b_n=\begin{cases}
\hspace{4mm} 0,\quad\qquad\qquad \text{for}\hspace{1mm}n\hspace{1mm}\text{even},\\
(-1)^{\frac{n-1}{2}}w^{-n-1},  \text{for}\hspace{1mm}n\hspace{1mm}\text{odd}.
\end{cases}
\end{equation*}
Now invoking \eqref{olvthm} twice, once with $s=y$ and then with $s=-y$, and then adding the resulting two identities, we arrive at
{\allowdisplaybreaks\begin{align*}
\mathfrak{R}(y, w)&\sim\frac{1}{2}\left(\sum_{n=0}^{\infty}b_ne^{i(n+1)\pi/2}\G(n+1)y^{-n-1}+\sum_{n=0}^{\infty}b_ne^{i(n+1)\pi/2}\G(n+1)(-y)^{-n-1}\right)\nonumber\\
&=\sum_{n=1}^{\infty}\frac{b_{2n-1}(2n-1)!\cos(n\pi)}{y^{2n}}=-\sum_{n=1}^{\infty}\frac{(2n-1)!}{w^{2n}y^{2n}}.
\end{align*}}
This completes the proof.
\end{proof}
We now give two proofs of a crucial lemma which is interesting in itself, and is employed in the proof of Theorem \ref{raabesum}. Each has its advantage over the other in that one is instructive and the other employs known identities on special functions. We begin with the instructive one first.
\begin{lemma}\label{zerodoubleintegral}
For $y>0$ and \textup{Re}$(u)>0$, 
\begin{align}\label{zerodoubleintegraleqn}
\int_0^\infty\int_0^\infty\frac{t\cos(2\pi yt)}{u^2+t^2} \mathrm{d}t  \mathrm{d}y=0.
\end{align}
\end{lemma}
\begin{proof}
It is important to note that the above double integral does not converge absolutely and hence Fubini's theorem is inapplicable, that is, we cannot change the order of integration. First assume $u>0$. We prove \eqref{zerodoubleintegraleqn} in the form $\int_0^\infty\mathfrak{R}(y,w)\,\mathrm{d}y=0$, where $w = 2 \pi u$ and $\mathfrak{R}(y,w)$ is defined in \eqref{raabe}. First of all, \eqref{rabasysmall} and \eqref{rabasylarge} imply that this integral exists. Now let $N$ be a positive integer and consider the integral
\begin{align}\label{befleb}
I(w, N) & := \int_0^\infty e^{-\frac{y}{N}}\mathfrak{R}(y,w)    \,\mathrm{d}y.
\end{align}
With the help of Fubini's theorem, one can write
 {\allowdisplaybreaks\begin{align}\label{iwnano}
 I(w,N)  & = \int_0^\infty \frac{t}{w^2+t^2} \mathrm{d}t  \int_0^\infty  e^{-\frac{y}{N}} \cos( y t)     \,\mathrm{d}y  \nonumber\\
        & = \frac{1}{N}\int_0^\infty \frac{t}{(w^2+t^2)\left(\tfrac{1}{N^2}+t^2\right)}\,\mathrm{d}t \nonumber\\
        & = \frac{N}{1 - N^2w^2 } \int_0^\infty \left( \frac{t}{w^2+t^2} -\frac{t}{\tfrac{1}{N^2}+t^2}  \right) \mathrm{d}t \nonumber\\
        & =   \frac{N}{1 - N^2w^2 }\lim_{A \rightarrow \infty}\left\{ \int_{0}^A \frac{t}{w^2+t^2} \mathrm{d}t - \int_{0}^A \frac{t}{\tfrac{1}{N^2}+t^2} \mathrm{d}t \right\} \nonumber \\
        & =  \frac{N}{2(1 - N^2w^2) }\lim_{A \rightarrow \infty} \left\{ \Big[\log(w^2 + t^2) \Big]_{0}^A -\Big[\log(\tfrac{1}{N^2} + t^2) \Big]_{0}^A   \right\} \nonumber\\
        & = \frac{N}{2(1 - N^2w^2) }\lim_{A \rightarrow \infty} \left\{ \log\Big(\frac{w^2 + A^2}{w^2}\Big)   - \log\left(1+N^2A^2\right) \right\} \nonumber\\
        & = \frac{N}{2(1 - N^2w^2) } \lim_{A \rightarrow \infty} \left\{ \log\left(\frac{w^2 + A^2}{\tfrac{1}{N^2} + A^2}\right)   + 2 \log\left(\frac{1}{ Nw}\right) \right\} \nonumber\\
        & =  \frac{N \log\left(Nw\right) }{N^2w^2 -1 },
\end{align}}%
since $\lim_{A \rightarrow \infty}  \log\left(\frac{w^2 + A^2}{\frac{1}{N^2} + A^2}\right)=0$. Now note that $e^{-\frac{y}{N}}\mathfrak{R}(y,w)\to\mathfrak{R}(y,w)$ pointwise as $N\to\infty$ and $\left|e^{-\frac{y}{N}}\mathfrak{R}(y,w)\right|\leq\mathfrak{R}(y,w)$. Also, as mentioned before, the fact that $\mathfrak{R}(y,w)$ is integrable, as a function of $y$ from $0$ to $\infty$, is clear from \eqref{rabasysmall} and \eqref{rabasylarge}. Hence letting $N\to\infty$ on both sides of \eqref{befleb} and employing Lebesgue's dominated convergence theorem, we see that
\begin{align*}
\lim_{N\to\infty}I(w,N)&=\int_0^\infty\lim_{N\to\infty} e^{-\frac{y}{N}}\mathfrak{R}(y,w)    \,\mathrm{d}y\\
&=\int_0^\infty\mathfrak{R}(y,w)    \,\mathrm{d}y,
\end{align*}
whereas \eqref{iwnano} implies that $\lim_{N\to\infty}I(w,N)=0$. Together, these complete the proof of \eqref{zerodoubleintegraleqn} for $u>0$. Note that the left-hand side of \eqref{zerodoubleintegraleqn} is analytic for Re$(u)>0$ as can be seen using Theorem 2.3 from \cite[p.~30]{temme}. Hence by analytic continuation, the result holds for Re$(u)>0$.
\end{proof}
%We now wish to let $N \rightarrow \infty$ on both sides and interchange the order of the limit and the outer integral.  so as to arrive at 
%\begin{align*}
%\lim_{d \rightarrow 0^{+}} I(w, d) & = \int_0^\infty \int_0^\infty\frac{t\  \cos( y t)}{w^2+t^2} \lim_{d \rightarrow 0^{+}} e^{-d y}  \mathrm{d}t    \,\mathrm{d}y \\
%& = \lim_{d \rightarrow 0^{+}} \frac{d \log\Big(\frac{d}{ w}\Big) }{(d^2 - w^2) } =0. 
%\end{align*}
%\end{proof}
The second proof of \eqref{zerodoubleintegraleqn} is now given.\\

\noindent
\textbf{Second proof.} Let $y>0$ and $u\in\mathbb{C}$ with $\mathrm{Re}(u) >0$. From \eqref{Oberhettinger 5.8p}
\begin{align}\label{Oberhettinger 5.8}
\int_{0}^\infty  \frac{ \cos( y t  )}{w^2 + t^2} t^{s-1} \mathrm{d}t&=\frac{\sqrt{\pi}}{2} w^{s-2} G_{1, 3}^{2, 1} \left( \frac{ w^2 y^2}{4} \Big| \begin{matrix}
 &1 - \frac{s}{2} \\
& 0, 1 -\frac{s}{2}, \frac{1}{2}
\end{matrix}  \right),
\end{align}
where $G_{p,q}^{\,m,n} \!\left(  z  \Big| \,\begin{matrix} a_1,.., a_p \\ b_1,.., b_q \end{matrix} \;  \right)$ is the Meijer $G$-function defined by the line integral \cite[p.~415]{olver-2010a}
\begin{align}\label{defnmg}
G_{p,q}^{\,m,n} \!\left(  z  \Big| \,\begin{matrix} a_1,.., a_p \\ b_1,.., b_q \end{matrix} \;  \right) := \frac{1}{2 \pi i} \int_L \frac{\prod_{j=1}^m \Gamma(b_j - s) \prod_{j=1}^n \Gamma(1 - a_j +s) z^s  } {\prod_{j=m+1}^q \Gamma(1 - b_j + s) \prod_{j=n+1}^p \Gamma(a_j - s)}\mathrm{d}s
\end{align}
for $z\neq 0$, and $m, n, p, q\in\mathbb{Z}$ with $0\leq m \leq q \,,\, 0\leq n \leq p$ and
$ a_i - b_j \not\in \mathbb{N} $ for $1\leq i \leq p\,\,, 1\leq j \leq q$, where the path of integration $L$ separates the poles of the factors $\G(b_{j}-s)$ from those of the factors $\G(1-a_{j}+s)$. Note that in \eqref{Oberhettinger 5.8}, we employed the following theorem of Slater \cite[p.~415, Equation 16.17.2]{olver-2010a} which gives connection between Meijer $G$-function and the generalized hypergeometric function $ {}_p F_{q-1} $:
Assume $p \leq q $ and $b_j - b_k \not\in \mathbb{Z}$ for $ j \neq k, 1 \leq j, k \leq n $. Then 
\begin{align*}
G_{p,q}^{\,m,n} \!\left(  z  \Big| \,\begin{matrix} a_1,.., a_p \\ b_1,.., b_q \end{matrix} \;  \right) = \sum_{k=1}^{m} A_{p,q,k}^{m,n}(z) {}_p F_{q-1} \left( (-1)^{p-m-n} z \Big| \begin{matrix}
1+b_k - a_1,\cdots, 1+ b_k - a_p \\
1+ b_k - b_1, \cdots * \cdots, 1 + b_k - b_q 
\end{matrix} \right),
\end{align*}
where $*$ indicates that the entry $1 + b_k - b_k$ is omitted and 
$$
A_{p,q,k}^{m,n}(z) := z^{b_k}  \prod_{ l=1,  l\neq k}^{m} \Gamma(b_l - b_k ) \prod_{l=1}^n   \Gamma( 1 + b_k -a_l ) \left( \prod_{l=m}^{q-1} \Gamma(1 + b_k - b_{l+1}) \prod_{l=n}^{p-1} \Gamma(a_{l+1} - b_k)  \right)^{-1}.
$$
Note that the validity of \eqref{Oberhettinger 5.8} for $s=2$ is to be seen by taking the limit of expression on the left side in that last step as $s\to 2$ since both $\G(\frac{s}{2}-1)$ and $\textup{cosec}\left(\frac{\pi s}{2}\right)$ have simple pole at $s=2$. Thus
\begin{align*}
\int_{0}^\infty  \frac{t \cos( y t  )}{w^2 + t^2} \mathrm{d}t & = \frac{\sqrt{\pi}}{2}  G_{1, 3}^{2, 1} \left( \frac{ w^2 y^2}{4} \Big| \begin{matrix}
 &0 \\
& 0, 0, \frac{1}{2}
\end{matrix}  \right),
\end{align*}
and so
\begin{align}\label{doubleG}
\int_{0}^{\infty} \int_{0}^{\infty} \frac{t \cos(y t)}{w^2 + t^2 } \mathrm{d}t \mathrm{d}y & = \int_{0}^{\infty}  \frac{\sqrt{\pi}}{2}  G_{1, 3}^{2, 1} \left( \frac{ w^2 y^2}{4} \Big| \begin{matrix}
 &0 \\
& 0, 0, \frac{1}{2}
\end{matrix}  \right) \mathrm{d}y \nonumber \\
& = \frac{\sqrt{\pi}}{4} \int_{0}^{\infty}    G_{1, 3}^{2, 1} \left( \frac{ w^2 Y}{4} \Big| \begin{matrix}
 &0 \\
& 0, 0, \frac{1}{2}
\end{matrix}  \right) \frac{1}{\sqrt{Y}} \mathrm{d}Y .
\end{align}
Since \eqref{defnmg} implies
\begin{align*}
\int_{0}^{\infty} G_{p,q}^{\,m,n} \!\left( \Big. \eta x  \Big| \,\begin{matrix} a_1,.., a_p \\ b_1,.., b_q \end{matrix} \;  \right) x^{s-1} \mathrm{d}x = \frac{\eta^{-s} \prod_{j=1}^m \Gamma(b_j + s) \prod_{j=1}^n \Gamma(1 - a_j - s)  } {\prod_{j=m+1}^q \Gamma(1 - b_j - s) \prod_{j=n+1}^p \Gamma(a_j + s)},
\end{align*}
letting $m =2, n=1, p=1, q=3$ and $a_1=b_1=b_2=0, b_3= \frac{1}{2}$, $\eta = \frac{w^2}{4} $ and $x = Y$  in the above formula implies
\begin{align*}
\int_{0}^{\infty}    G_{1, 3}^{2, 1} \left( \frac{ w^2 Y}{4} \Big| \begin{matrix}
 &0 \\
& 0, 0, \frac{1}{2}
\end{matrix}  \right) Y^{s-1} \mathrm{d}Y =\frac{4^s \Gamma(s)\Gamma(s)\Gamma(1-s) }{w^{2s}\Gamma\Big(1-\frac{1}{2} -s \Big)}.
\end{align*}
Now let $s=1/2$ in the above equation and substitute the resultant in \eqref{doubleG} to arrive at \eqref{zerodoubleintegraleqn} since the Gamma function in the denominator on the right side has a pole at $s=1/2$ whereas the ones in the numerator are well-defined. This completes the proof.
\qed

\begin{proof}[Theorem \textup{\ref{raabesum}}][]
Let 
\begin{align}\label{gex}
g(x) = \int_{0}^\infty \frac{ v \cos( 2 \pi x v)}{ u^2 + v^2} \mathrm{d}v.  
\end{align}
That the above integral exists is clear from the fact that $f(v)=\frac{v}{2(u^2+v^2)}$ satisfies the hypotheses of Theorem \ref{Guind-pois}. To say that the two limits in Theorem \ref{Guind-pois} are equal implies, in particular, that they exist. The fact that $\int_{0}^{\infty}g(x)\, dx$ exists can be seen, in particular, from Lemma \ref{zerodoubleintegral}. Together, we conclude that $\sum_{m=1}^\infty g(m)$ is convergent. Employing Theorem \ref{Guind-pois} with $f(v)=\frac{v}{2(u^2+v^2)}$ and $g(x)$ as in \eqref{gex} and invoking Lemma \ref{zerodoubleintegral}, we see that
{\allowdisplaybreaks\begin{align}\label{usegui}
\sum_{m=1}^\infty g(m)
%&=\lim_{M\rightarrow\infty}\left(\sum_{n=1}^M f(n)-\int_0^M f(y)\, dy\right)\nonumber \\
&=\frac{1}{2}\lim_{M \rightarrow\infty}\left(\sum_{n=1}^M\frac{n}{u^2 + n^2}-\int_0^M \frac{v}{u^2+v^2}\, dv\right) \nonumber\\
&=\frac{1}{2}\lim_{M \rightarrow\infty}\left\{\left(\sum_{n=1}^M\frac{n}{u^2 + n^2}-\log M\right)+\left(\log M-\int_0^M \frac{v}{u^2+v^2}\, dv\right)\right\} \nonumber\\
&=\frac{1}{2}\left[-\frac{1}{2}(\psi(iu)+\psi(-iu))\right]+\frac{1}{2}\lim_{M\rightarrow\infty}\left(\log M-\int_0^M \frac{v}{u^2+v^2}\, dv\right) \nonumber\\
&=\frac{1}{2}\left[-\frac{1}{2}(\psi(iu)+\psi(-iu))\right]+\frac{1}{2}\lim_{M\rightarrow\infty}\left(\log M-\frac{1}{2}\left(\log(u^2+M^2)-\log u^2\right)\right) \nonumber\\
&=\frac{1}{2}\left[-\frac{1}{2}(\psi(iu)+\psi(-iu))\right]+\frac{1}{2}\lim_{M\rightarrow\infty}\left(\log \frac{M}{\sqrt{u^2+M^2}}+\log u\right)\nonumber\\
&=\frac{1}{2}\left(\log u-\frac{1}{2}(\psi(iu)+\psi(-iu))\right),
\end{align}}%
where in the third step, we used \cite[Equation (3.8)]{dixitlap}. Theorem \ref{raabesum} now follows from \eqref{usegui} and by employing the change of variable $v=t/(2\pi m)$ and replacing $x$ by $m$ and $u$ by $u/(2\pi)$.
\end{proof}
Finally, Theorem \ref{raabesum} and Lemma \ref{lemsix} together give following beautiful closed-form evaluation of a double sum. We record it as a separate theorem for its possible applicability in other studies.
\begin{theorem}
For $u>0$,
\begin{equation*}
\sum_{m=1}^{\infty}\sum_{k=0}^{\infty}\frac{(mu)^{2k}}{(2k)!}\left(\psi(2k+1)-\log(mu)\right)=\frac{1}{2}\left\{\log\left(\frac{u}{2\pi}\right)-\frac{1}{2}\left(\psi\left(\frac{iu}{2\pi}\right)+\psi\left(\frac{-iu}{2\pi}\right)\right)\right\}.
\end{equation*}
\end{theorem}
\begin{remark}
A mere look at the double series on the left side above indicates that one cannot interchange the order of the double sum. This makes its closed-form evaluation all the more interesting.
\end{remark}

\section{Proof of the formula for Hurwitz zeta function at rational arguments}\label{hgNby2}
We begin with a lemma which gives inverse Mellin transform of $\G(s)/\tan\left(\frac{\pi s}{2}\right)$. It is an important ingredient in the proof of Theorem \ref{dgkmgen}.
\begin{lemma}\label{MellTransOfCOT}
 For $ 0< \mathrm{Re}(s)=c_1 < 2 $ and $ \mathrm{Re}(u) >0$, we have
\begin{align*}
\frac{1}{2\pi i}\int_{(c_1)}\frac{\Gamma(s)}{\tan(\frac{\pi s}{2})}u^{-s}\ \mathrm{d}s&=\frac{2}{\pi}\int_0^\infty \frac{t\cos t}{u^2+t^2}\ \mathrm{d}t.
\end{align*}
\end{lemma}

\begin{proof}
For $0<\textrm{Re}(s)<1$, we have
\begin{align*}
\int_0^\infty t^{s-1}\cos t\ \mathrm{d}t=\Gamma(s) \cos\Big(\frac{\pi s}{2} \Big). 
\end{align*}
and for $0<\textrm{Re}(s)<2$, we know
\begin{align*}
\ \int_0^\infty t^{s-1} \frac{1}{1 + t^2 } \mathrm{d}t = \frac{\pi }{2} \textup{cosec}\Big(\frac{\pi s}{2} \Big).
\end{align*}
One can find the first of the two Mellin transforms given above in \cite[p.~1101, Formula (3)]{grn}. The second one can be easily obtained by replacing $s$ by $s/2$ and employing a change of variable $x=t^2$ in \cite[p. 1101, Formula (6)]{grn}. Now using Parseval's formula \cite[p.~83, Equation (3.1.13)]{kp},  for $0<\mathrm{Re}(s)<1$, one can obtain
\begin{align*}
\frac{1}{2\pi i}\int_{(c_1)}\frac{\Gamma(s)}{\tan(\frac{\pi s}{2})}u^{-s}\ \mathrm{d}s&=\frac{1}{2\pi i}\int_{(c_1)}\frac{\Gamma(s)\cos(\frac{\pi s}{2})}{\sin(\frac{\pi s}{2})}u^{-s}\ \mathrm{d}s\\ \nonumber
&=\frac{1}{\pi}\int_0^\infty \frac{2\cos t}{1+\frac{u^2}{t^2}}\frac{\mathrm{d}t}{t}\\
&=\frac{2}{\pi}\int_0^\infty \frac{t\cos t}{u^2+t^2}\ \mathrm{d}t.
\end{align*}
Now one can easily extend the region of validity of the above result to $0<$ Re$(s)<2$ by noting that when shift the line of integration Re$(s)=c_1$ to, say, Re$(s)=c_2, 1\leq c_2<2,$ one does not encounter any poles of the integrand and also that the integrals over the horizontal segments tend to zero as the height $T$ tends to $\infty$.
\end{proof}

\begin{lemma}\label{equality}
Let $N$ be an odd positive integer and $h>N/2$ be a positive integer. If $\frac{N-2h+1}{N}=-2j$ for some $j\in\mathbb{N}$, then $j=\left\lfloor\frac{h}{N}-\frac{1}{2}\right\rfloor$.
\end{lemma}
\begin{proof}
By hypothesis, $j+\frac{1}{2N}=\frac{h}{N}-\frac{1}{2}$. Since $j$ is an integer and $\left\lfloor\frac{1}{2N}\right\rfloor=0$ for $N\geq 1$, we have $j=\left\lfloor j+\frac{1}{2N}\right\rfloor=\left\lfloor\frac{h}{N}-\frac{1}{2}\right\rfloor$.
\end{proof}
The following lemma is well-known. We give here a proof for the sake of completeness.
\begin{lemma}\label{cheby}
For any $z\in\mathbb{C}$ and $N\in\mathbb{N}$,
\begin{align}\label{ss}
\frac{\sin( N z)}{\sin (z)} = \sum_{j= -(N-1)}^{N-1}{\vphantom{\sum}}'' \exp( i j z),
\end{align}
where $\sum_{j}{\vphantom{\sum}}'' $ means the summation is over $j=-(N-1), -(N-3),\cdots, N-3, N-1$. In particular, for $N$ odd,
\begin{align}\label{cc}
\frac{ \cos(N z)}{\cos( z )}=(-1)^{\frac{N-1}{2}}\sum_{j= -(N-1)}^{N-1}{\vphantom{\sum}}'' i^j\exp( -i j z),
\end{align}
and for N even,
\begin{align}\label{sc}
\frac{ \sin(N z)}{\cos( z )}=(-1)^{\frac{N}{2}}\sum_{j= -(N-1)}^{N-1}{\vphantom{\sum}}'' i^j\exp( i j z).
\end{align}
\end{lemma}
\begin{proof}
First let $N$ be odd. Start from the right side of \eqref{ss} and note that
\begin{align*}
\sum_{j= -(N-1)}^{N-1}{\vphantom{\sum}}'' \exp( i j z)&=\left(e^{-i(N-1)z}+e^{i(N-1)z}\right)+\cdots+\left(e^{-i2z}+e^{i2z}\right)+1\nonumber\\
&=\frac{2\sin(z)\left(\cos((N-1)z)+\cos((N-3)z)+\cdots+\cos(2z)\right)+\sin(z)}{\sin(z)}\nonumber\\
&=\frac{\left(\sin(Nz)-\sin((N-2)z)\right)+\cdots+\left(\sin(3z)-\sin(z)\right)+\sin(z)}{\sin(z)}\nonumber\\
&=\frac{\sin( N z)}{\sin (z)}.
\end{align*}
The case when $N$ is an even positive integer can be similarly proved. Also, \eqref{cc} follows at once by replacing $z$ by $\frac{\pi}{2}-z$ in \eqref{ss} and taking $N$ to be odd whereas \eqref{sc} is obtained by replacing $z$ by $\frac{\pi}{2}+z$ in \eqref{ss} and letting $N$ to be even.
\end{proof}
We have collected now all tools necessary for proving Theorem \ref{dgkmgen}. 

\begin{proof}[Theorem \textup{\ref{dgkmgen}}][]
The hypothesis $h\geq N/2$, $N\in\mathbb{N}$, will be used several times, without mention, in the proof. It is easy to see that the series $\sum_{n=0}^\infty n^{N-2h}\frac{\exp(-a n^Nx)}{1-\exp(-n^Nx)}$ is absolutely and uniformly convergent for any $x>0, N\in\mathbb{N}$. Thus, interchanging the order of summation and integration in the first step below, we see that for Re$(s)>\max\left(\frac{N-2h+1}{N},1\right)  =1$,
{\allowdisplaybreaks\begin{align*}
\int_{0}^{\infty}x^{s-1}\sum_{n=1}^\infty n^{N-2h}\frac{\exp(-a n^Nx)}{1-\exp(-n^Nx)}\, \mathrm{d}x&=\sum_{n=1}^\infty n^{N-2h}\int_{0}^{\infty}x^{s-1}\frac{\exp(-a n^Nx)}{1-\exp(-n^Nx)}\, \mathrm{d}x\nonumber\\
&=\sum_{n=1}^{\infty}n^{N-2h-Ns}\int_{0}^{\infty}\frac{y^{s-1}e^{-ay}}{1-e^{-y}}\, \mathrm{d}y\nonumber\\
&=\G(s)\zeta(s,a)\zeta(Ns-N+2h),
\end{align*}}
where, in the second step, we employed the change of variable $y=n^{N}x$ and in the third, we used the fact \cite[p.~251, Theorem 12.2]{apostol-1998a} that for Re$(s)>1$,
\begin{align*}
\int_{0}^{\infty}\frac{y^{s-1}e^{-ay}}{1-e^{-y}}\, \mathrm{d}y=\G(s)\zeta(s, a).
\end{align*}
Thus, for $\lambda=$ Re$(s)>1$, 
\begin{align}\label{mainequality}
\sum_{n=1}^\infty n^{N-2h}\frac{\exp(-a n^Nx)}{1-\exp(-n^Nx)}=\frac{1}{2 \pi i}\int_{(\lambda)}\Gamma(s)\zeta(s,a)\zeta\left(Ns-(N-2h)\right)x^{-s} \mathrm{d}s.
\end{align}
We now obtain an alternate evaluation of the above integral by shifting the line of integration and then by using Cauchy's residue theorem. Consider the contour $\mathcal{C}$  determined by the line segments $[\lambda-iT, \lambda+iT], [\lambda+iT,-r+iT], [-r +iT, -r -iT] \ \text{and}\ [-r-iT, \lambda-iT]$, where, $r$ is a sufficiently large positive real number which is not an integer and $\frac{2h}{N}-1<r<\frac{2h+1}{N}-1$. The reason for choosing the lower and upper bounds for $r$ will be explained soon. Let 
\begin{align}\label{F(s)}
F(s):=\Gamma(s)\zeta(s,a)\zeta(Ns-(N-2h))x^{-s}
\end{align}
and let $R_a$ denote the residue of $F(s)$ at the pole $s=a$. We first find poles of $F(s)$ and residues at those poles.

(1) $F(s)$ has a pole of order one at $s=0$ since $\Gamma(s)$ has a simple pole at $s=0$. The residue $R_0$ at this pole is given by
\begin{align}\label{R_0}
R_0=\lim_{s\rightarrow 0} s F(s)=\zeta(0,a)\zeta(-N+2h)=-\left(a-\tfrac{1}{2}\right)\zeta(2h-N),
\end{align}
since \cite[p.~264, Theorem 12.13]{apostol-1998a}
$$\zeta(-n,a)=-\frac{B_{n+1}(a)}{n+1},\ n\geq 0\, \text{and}\ B_1(a)=a-\frac{1}{2}.$$

(2) Since $\zeta(s,a)$ has a simple pole at $s=1$, $F(s)$ has a simple pole at $s=1$ with residue
\begin{align}\label{R_1}
R_1=\lim_{s\rightarrow 1}(s-1)F(s)=\frac{\zeta(2h)}{x}.
\end{align}

(3) Since $\frac{N-2h+1}{N}\neq-2\left\lfloor\frac{h}{N}-\frac{1}{2}\right\rfloor$, Lemma \ref{equality} implies $\frac{N-2h+1}{N}\neq-2j$ for any $j\geq 1$. Thus, $F(s)$ has a \emph{simple} pole at $s=\frac{N-2h+1}{N}$, owing to the pole of $\zeta(s)$ at $s=1$, with the residue
\begin{align}\label{R_N-2h+1/N1}
R_{\frac{N-2h+1}{N}}=\Gamma\left(\frac{N-2h +1}{N}\right)\zeta\left(\frac{N-2h +1}{N},a\right)\frac{x^{-\frac{(N-2h+1)}{N}}}{N}.
\end{align}

(4) Consider the simple poles of $\Gamma(s)$ at $s=-2j,\ j\in\mathbb{N}$, and the trivial zeros of $\zeta(Ns-N+2h)$ at $-2k, \ k\in\mathbb{N}$. It is important to see if some of these poles of $\G(s)$ get canceled by the trivial zeros of $\zeta(Ns-N+2h)$. To that end, suppose for some positive integers $j'$ and $k'$ we have $N(-2j')-N+2h=-2k'$. Then $k'=Nj'+\frac{N}{2}-h$. This implies that if $N$ is an odd positive integer, no pole of $\G(s)$ at a negative even integer will get canceled by a trivial zero of $\zeta(Ns-N+2h)$ since $k'$ is not an integer. However, if $N$ is an even positive integer, then $k'$ can equal $Nj'+\frac{N}{2}-h$, while being a positive integer, implying $j'>\frac{h}{N}-\frac{1}{2}$, that is, among the poles of $\G(s)$ at negative even integers, only the poles $-2j$, $1\leq j\leq\left\lfloor\frac{h}{N}-\frac{1}{2}\right\rfloor$, contribute towards the evaluation of the line integral. To sum up, when $N$ is odd integer, $F(s)$ has simple poles at all negative even integers $-2j,j\geq 1$, and when $N$ is an even integer, $F(s)$ has simple poles at $-2j$, where $1\leq j\leq\left\lfloor\frac{h}{N}-\frac{1}{2}\right\rfloor$. The residue at such a pole is
\begin{align}\label{R_-2j}
R_{-2j}&=\lim_{s\rightarrow -2j}(s+2j)\ \Gamma(s)\ \zeta(s,a)\ \zeta(Ns-(N-2h))\ x^{-s}\nonumber\\
&=\frac{\zeta(-2j,a)}{(2j)!}\zeta(-2jN-N+2h)\ x^{2j}\nonumber\\
&=-\frac{B_{2j+1}(a)}{(2j+1)!}\zeta(-2jN-N+2h)\ x^{2j}.
\end{align}
At this juncture, it deems necessary to explain why we choose the real part of the shifted line of integration to be $-r$ with $r>\frac{2h}{N}-1$. The reason is, this implies $-r<-2\left\lfloor\frac{h}{N}-\frac{1}{2}\right\rfloor$, and thus all poles of $\G(s)$ at negative even integers $-2j$, where $1\leq j\leq\left\lfloor\frac{h}{N}-\frac{1}{2}\right\rfloor$, lie inside the contour, thus contributing towards the evaluation of the line integral.

(5) Arguing as in (4), it can be seen that $F(s)$ has simple poles at $s=-(2j-1)$, $1\leq j\leq\left\lfloor\frac{h}{N}\right\rfloor$, and the residue at such a pole is 
\begin{align}\label{R_-(2j-1)}
R_{-(2j-1)}&=\lim_{s\rightarrow -(2j-1)} (s+(2j-1))F(s)\nonumber\\
&=\frac{(-1)^{2j-1}}{(2j-1)!}\zeta(-(2j-1),a)\ \zeta(2h-2Nj)x^{2j-1}\nonumber\\
&=(-1)^{h+1}2^{2h-1}\pi^{2h}\left(\frac{-1}{4\pi^2}\right)^{jN}\  \frac{B_{2j}(a)\ B_{2h-2jN}}{(2j)!\ (2h-2jN)!}\ x^{2j-1}.
\end{align}
Now applying Cauchy's residue theorem, we observe that
 \begin{align*}
 \frac{1}{2\pi i}&\left[\int_{\lambda-iT}^{\lambda+iT}+\int_{\lambda+iT}^{-r+iT}+\int_{-r +iT}^{-r
 -iT}+\int_{-r-iT}^{ \lambda-iT}\right]\Gamma(s)\zeta(s,a)\zeta\left(Ns-(N-2h)\right)x^{-s}\mathrm{d}s\nonumber\\
 &=R_0+R_1+R_{\frac{N-2h+1}{N}}+\sum_{j=1}^{\left\lfloor\frac{h}{N}\right\rfloor}R_{-(2j-1)}+\sum_{j=1}^{\lfloor\frac{h}{N}-\frac{1}{2}\rfloor}R_{-2j}.
 \end{align*}
 Now let $T\to\infty$. Using Stirling's formula \eqref{strivert} for $\Gamma(s)$ and elementary bounds on the Riemann zeta function and the Hurwitz zeta function, it can be seen that the integrals along the horizontal segments $[\lambda+iT,-r+iT],[-r-iT, \lambda-iT]$ approach zero as $T\rightarrow\infty$. Hence from \eqref{R_0}-\eqref{R_-(2j-1)}, we see that
\begin{align}\label{cauchyRes2}
\sum_{n=1}^\infty n^{N-2h}\frac{\exp(-a n^Nx)}{1-\exp(-n^Nx)} = P(x, a)+J(x, a),
\end{align}  
where $P(x, a)$ is the sum of all residues of $F(s)$, defined in \eqref{pxa}, and
\begin{align}\label{Integral}
J(x, a):=\frac{1}{2\pi i}\int_{(-r)}\Gamma(s)\zeta(s,a)\zeta(Ns-(N-2h))x^{-s} \mathrm{d}s.
\end{align}
It remains to show that $J(x, a)$ agrees with $S(x, a)$ defined in \eqref{sodda} and \eqref{sevena} respectively when $N$ is odd and even. To evaluate $J(x, a)$, we first make a change of variable $s\leftrightarrow 1-s$ in \eqref{Integral} so that
\begin{align}\label{J(x)}
J(x, a)=\frac{1}{2\pi i}\int_{(1+r)}\Gamma(1-s)\zeta(1-s,a)\zeta(2h-Ns))x^{s-1}\, \mathrm{d}s.
\end{align}
Now replace $s$ by $1-s$ in \eqref{hformula}, then multiply both sides of the resulting identity by $\G(1-s)$ to obtain, for Re$(s)>1$,
\begin{align}\label{gamhur}
\Gamma(1-s)\zeta(1-s,a) &= \frac{2\Gamma(1-s)\Gamma(s)}{(2\pi)^s}\left\{\cos\left(\frac{\pi s}{2}\right)\sum_{n=1}^\infty\frac{\cos(2\pi na)}{n^s} 
 +\sin\left(\frac{\pi s}{2}\right)\sum_{n=1}^\infty\frac{\sin(2\pi na)}{n^s}\right\}\nonumber\\
&=(2\pi)^{1-s}\left\{\frac{1}{2\sin\left(\frac{\pi s}{2}\right)}\sum_{n=1}^\infty\frac{\cos(2\pi na)}{n^s} +\frac{1}{2\cos\left(\frac{\pi s}{2}\right)}\sum_{n=1}^\infty\frac{\sin(2\pi na)}{n^s}\right\},
\end{align}
where in the last step, we used the reflection formula for the gamma function and then the double angle formula for sine for simplification.

We note here that Kanemitsu, Tanigawa and Yoshimoto \cite[p.~51]{ktyacta} use a formula equivalent to \eqref{hformula}, namely, for Re$(s)<0$,
\begin{equation*}
\zeta(s, a)=\frac{\G(1-s)}{(2\pi)^{1-s}}\left(\exp{\left(\frac{-\pi i(1-s)}{2}\right)}\sum_{n=1}^{\infty}\frac{e^{2\pi ia n}}{n^{1-s}}+\exp{\left(\frac{\pi i(1-s)}{2}\right)}\sum_{n=1}^{\infty}\frac{e^{-2\pi ia n}}{n^{1-s}}\right).
\end{equation*}
However, one can see that while the above formula is useful when $N$ is an even positive integer, it is not when $N$ is an odd positive integer. In fact, employing it leads to very complicated integrals which do not seem to lead us to any concrete result. On the other hand, \eqref{hformula} works for any positive integer $N$, irrespective of its parity, as will be seen in the remainder of the proof.

Now substitute \eqref{gamhur} in \eqref{J(x)} and invoke the functional equation \eqref{zetafe1} for $\zeta(2h-Ns)$ to obtain after simplification
\begin{align}\label{jj1j2}
J(x, a)=J_1(x, a)+J_2(x, a),
\end{align}
where
\begin{align}
J_1(x, a):&=\frac{(-1)^{h+1}2^{2h+1}\pi^{2h}}{x}\frac{1}{2\pi i}\int_{(1+r)}\left(\frac{(2\pi)^{N+1}}{x}\right)^{-s}\Gamma(1-2h+Ns)\zeta(1-2h+Ns)\nonumber\\
&\hspace{5cm}\times \left\{\frac{\sin\left(\frac{N\pi s}{2}\right)}{2\sin\left(\frac{\pi s}{2}\right)}\sum_{n=1}^\infty\frac{\cos(2\pi na)}{n^s}\right\}  \mathrm{d}s,\label{j1}\\
J_2(x, a):&=\frac{(-1)^{h+1}2^{2h+1}\pi^{2h}}{x}\frac{1}{2\pi i}\int_{(1+r)}\left(\frac{(2\pi)^{N+1}}{x}\right)^{-s}\Gamma(1-2h+Ns)\zeta(1-2h+Ns)\nonumber\\
&\hspace{5cm}\times \left\{\frac{\sin\left(\frac{N\pi s}{2}\right)}{2\cos\left(\frac{\pi s}{2}\right)}\sum_{n=1}^\infty\frac{\sin(2\pi na)}{n^s}\right\}  \mathrm{d}s\label{j2}.
\end{align}
We first evaluate $J_2(x, a)$. Its evaluation depends on the parity of $N$. We first assume that $N$ is odd. Employ the change of variable 
\begin{equation}\label{s1s}
s_1=Ns-2h+1
\end{equation}
in \eqref{j2} so that $c_1:=$Re$(s_1)>1$ (since $r>\frac{2h}{N}-1$), write $\zeta(s_1)=\sum_{m=1}^{\infty}m^{-s_{1}}$, and then interchange the order of double sum and the integral, permitted because of absolute convergence, to arrive at
\begin{align}\label{j2xinter}
J_2(x, a)=\frac{(-1)^{h+1}2^{2h+1}\pi^{2h}}{Nx}\left(\frac{(2\pi)^{N+1}}{x}\right)^{\frac{1-2h}{N}}\sum_{m,n=1}^{\infty}n^{\frac{{1-2h}}{N}}\sin(2\pi na)E(X_{m,n}),
\end{align}
where
\begin{align}\label{E(X)}
E(X_{m,n}):=\frac{(-1)^{h-1}}{2\pi i}\int_{(c_1)}\frac{\Gamma(s_1)\cos(\frac{\pi s_1}{2})}{2\cos\left(\frac{\pi}{2}\left(\frac{s_1+2h-1}{N}\right)\right)}X_{m,n}^{-s_1}\ \mathrm{d}s_1,
\end{align}
with
\begin{align}\label{X_{m,n}}
X_{m,n}:=2\pi m\left(\frac{2\pi n}{x}\right)^{1/N}.
\end{align}
Using \eqref{cc} in the second step below, we find that
\begin{align}\label{finalE(X)}
E(X_{m,n})%&=\frac{(-1)^{h-1}}{2\pi i}\int_{(c_1)}\frac{\Gamma(s_1)}{2\tan(\frac{\pi s_1}{2})}\frac{\sin(\frac{\pi s_1}{2})}{\cos\left(\frac{\pi}{2}\left(\frac{s_1+2h-1}{N}\right)\right)}X_{m,n}^{-s_1}\ ds_1\\ \nonumber
&=\frac{-1}{2\pi i}\int_{(c_1)}\frac{\Gamma(s_1)}{2\tan(\frac{\pi s_1}{2})}\frac{\cos\left(\frac{\pi}{2}(s_1+2h-1)\right)}{\cos\left(\frac{\pi}{2}\left(\frac{s_1+2h-1}{N}\right)\right)}X_{m,n}^{-s_1}\ \mathrm{d}s_1 \nonumber\\
&=\frac{(-1)^{\frac{N+1}{2}}}{2}\frac{1}{2\pi i}\int_{(c_1)}\frac{\Gamma(s_1)}{\tan(\frac{\pi s_1}{2})}\sum_{j=-(N-1)}^{N-1}{\vphantom{\sum}}''i^je^{-\frac{ij\pi}{2}\left(\frac{s_1+2h-1}{N}\right)}X_{m,n}^{-s_1}\ \mathrm{d}s_1\nonumber\\ 
&=\frac{(-1)^{\frac{N+1}{2}}}{2}\sum_{j= -(N-1)}^{N-1}{\vphantom{\sum}}''i^je^{\frac{-ij\pi(2h-1)}{2N}}\frac{1}{2\pi i}\int_{(c_1)}\frac{\Gamma(s_1)}{\tan(\frac{\pi s_1}{2})}{X_{m,n, j}^*}^{-s_1}\ \mathrm{d}s_1,
\end{align}
where 
\begin{align}\label{xmnstar}
X_{m,n, j}^{*}:=X_{m,n}e^{ \frac{ij\pi}{2N} }.
\end{align}
Now \eqref{s1s} and the inequality $\frac{2h}{N}-1<r<\frac{2h+1}{N}-1$ along with the fact that Re$(s)=1+r$ imply that $1<$Re$(s_1)<2$. The reason why we initially chose $r<\frac{2h+1}{N}-1$ is because, we need Re$(s_1)<2$ in order to use Lemma \ref{MellTransOfCOT}. Hence invoking Lemma \ref{MellTransOfCOT} to simplify the above representation for $E(X_{m,n})$ and then substituting the resultant in \eqref{j2xinter} gives, upon simplification,
\begin{align*}
J_2(x, a)&= \frac{2}{\pi Nx}(-1)^{h+\frac{N+3}{2}}(2\pi)^{2h}\left(\frac{(2\pi)^{N+1}}{x}\right)^{\frac{1-2h}{N}}\sum_{j= -(N-1)}^{N-1}{\vphantom{\sum}}''i^je^{\frac{-ij\pi(2h-1)}{2N}}\nonumber\\
&\quad\times\sum_{n=1}^{\infty}n^{\frac{1-2h}{N}}\sin(2\pi na)\sum_{m=1}^{\infty}\int_{0}^{\infty}\frac{t\cos(t)}{{X_{m,n, j}^*}^{2}+t^2}\, \mathrm{d}t.
\end{align*}
Now note that Re$(X_{m,n, j}^*)=2\pi m\left(\frac{2\pi n}{x}\right)^{\frac{1}{N}}\cos\left(\frac{\pi j}{2N}\right)>0$ as
\begin{equation*}
-\frac{\pi}{2}<-\frac{\pi(N-1)}{2N}\leq\frac{\pi j}{2N}\leq\frac{\pi (N-1)}{2N}<\frac{\pi}{2}. 
\end{equation*}
Hence apply Theorem \ref{raabesum} and then replace $j$ by $2j$ in the second step below to deduce that
\begin{align}\label{oddj2x}
J_2(x, a)&= (-1)^{h+\frac{N+3}{2}}\frac{(2\pi)^{2h}}{\pi Nx}\left(\frac{(2\pi)^{N+1}}{x}\right)^{\frac{1-2h}{N}}\sum_{j= -(N-1)}^{N-1}{\vphantom{\sum}}''i^je^{\frac{-ij\pi(2h-1)}{2N}}\sum_{n=1}^{\infty}n^{\frac{1-2h}{N}}\sin(2\pi na)\nonumber\\
&\quad\times\left\{\log\left(\left(\tfrac{2\pi n}{x}\right)^{\frac{1}{N}}e^{\frac{i\pi j}{2N}}\right)-\frac{1}{2}\left(\psi\left(i\left(\tfrac{2\pi n}{x}\right)^{\frac{1}{N}}e^{\frac{i\pi j}{2N}}\right)+\psi\left(-i\left(\tfrac{2\pi n}{x}\right)^{\frac{1}{N}}e^{\frac{i\pi j}{2N}}\right)\right)\right\}\nonumber\\
&=\frac{(-1)^{h+\frac{N+3}{2}}}{\pi N}\left(\frac{2\pi}{x}\right)^{\frac{N-2h+1}{N}}\sum_{j=-\frac{(N-1)}{2}}^{\frac{(N-1)}{2}}(-1)^j\textup{exp}\left(\frac{i\pi(1-2h)j}{N}\right)\nonumber\\
&\quad\times\sum_{n=1}^{\infty}\frac{\sin(2\pi na)}{n^{\frac{2h-1}{N}}}\left\{\log\left(\tfrac{1}{\pi}A_{N,j}\left(\tfrac{n}{x}\right)\right)-\tfrac{1}{2}\left(\psi\left(\tfrac{i}{\pi}A_{N,j}\left(\tfrac{n}{x}\right)\right)+\psi\left(-\tfrac{i}{\pi}A_{N,j}\left(\tfrac{n}{x}\right)\right)\right)\right\},
\end{align}
where 
\begin{equation*}
A_{N,j}(y)=\pi\left(2\pi y\right)^{\frac{1}{N}}e^{\frac{i\pi j}{N}}.
\end{equation*}
This completes the evaluation of $J_2(x)$ when $N$ is odd.

Let us now consider the case when $N$ is even. Note that \eqref{j2xinter} still holds with $E(X_{m,n})$ and $X_{m,n}$ the same as defined in \eqref{E(X)} and \eqref{X_{m,n}}. But now we use \eqref{sc} and \eqref{xmnstar} in the second step below to simplify $E(X_{m,n})$ as
{\allowdisplaybreaks\begin{align}\label{exmnNeven}
E(X_{m,n})& = \frac{1}{2 \pi i} \int_{(c_1)} \frac{\Gamma(s_1) \sin\left(\frac{\pi}{2}(s_1+2h-1)\right)}{2\cos\left(\frac{\pi}{2}\left(\frac{s_1+2h-1}{N}\right)\right)} X_{m,n}^{-s_1} \mathrm{d}s_1 \nonumber\\
& = (-1)^{\frac{N}{2} }  \sum_{j= -(N-1)}^{N -1 }{\vphantom{\sum}}''  i^j  \exp\Big(\frac{i \pi j (2h-1)}{2 N} \Big)\frac{1}{2 \pi i} \int_{(c_1)}  \Gamma(s_1) {X_{m,n,-j}^{*}}^{-s_1}  \mathrm{d}s_1 \nonumber\\
&=(-1)^{\frac{N}{2} }  \sum_{j= -(N-1)}^{N -1 }{\vphantom{\sum}}''  i^j  \exp\Big(\frac{i \pi j (2h-1)}{2 N} \Big)e^{-X_{m,n,-j}^{*}},
\end{align}}
where in the last step, we used \eqref{gammamelli} since Re$(X_{m,n,-j}^{*})>0$. Replacing $j$ by $-2j-1$ in \eqref{exmnNeven} and then substituting the resultant in \eqref{j2xinter}, we deduce that
\begin{align}\label{evenj2x}
J_2(x, a)=\frac{i (-1)^{h+\frac{N}{2}}}{N} \left( \frac{2\pi}{ x}\right)^{ \frac{N-2h+1}{N}}  \sum_{j= -\frac{N}{2}}^{\frac{N}{2} -1 }  (-1)^j   e^{\frac{i \pi (2 j +1) (1-2h)}{2 N}}
\sum_{n=1}^{\infty}      \frac{n^{\frac{{1-2h}}{N}}\sin(2\pi na)}{ \exp\big( 2A_{N,j+\frac{1}{2}}\left(\frac{n}{x}\right) \big) -1 }.
\end{align}
% \exp\left(\frac{i \pi (2 j +1) (1-2h)}{2 N} \right) 
From \eqref{oddj2x} and \eqref{evenj2x}, we obtain an expression for $J_{2}(x)$ for all positive integers $N$.

Now $J_1(x, a)$ from \eqref{j1} can be evaluated in a similar way to obtain
\begin{align}\label{j11}
J_1(x, a)=\frac{(-1)^{h+1}}{N} \left( \frac{2\pi}{ x}\right)^{ \frac{N-2h+1}{N}} \sum_{j=-\frac{(N-1)}{2}}^{\frac{(N-1)}{2}}e^{\frac{i\pi(1-2h)j}{N}}\sum_{n=1}^{\infty}\frac{\cos(2\pi na)}{n^{\frac{2h-1}{N}}\left(\textup{exp}\left(2A_{N,j}\left(\frac{n}{x}\right)\right)-1\right)}
\end{align}
for $N$ odd, whereas, for $N$ even,
\begin{align}\label{j12}
J_1(x, a)=\frac{(-1)^{h+1}}{N}\left(\frac{2\pi}{x}\right)^{\frac{N-2h+1}{N}}\sum_{j=-\frac{N}{2}}^{\frac{N}{2}-1}e^{\frac{i\pi(1-2h)\left(j+\frac{1}{2}\right)}{N}}\sum_{n=1}^{\infty}\frac{\cos(2\pi na)}{n^{\frac{2h-1}{N}}\left(\textup{exp}\left(2A_{N,j+\frac{1}{2}}\left(\frac{n}{x}\right)\right)-1\right)}.
\end{align}
In fact, the above expressions for $J_1(x)$ differ from the expression in the first equality in \cite[Equation (2.18)]{ktyhr} only in that the numerator of the summand of the infinite series in them involve $\cos(2\pi na)$, which is absent in the latter.

Finally, adding the corresponding sides of \eqref{j11} and \eqref{oddj2x} when $N$ is odd, and respectively of \eqref{j12} and \eqref{evenj2x} when $N$ is even gives expressions for $J(x)$ (see \eqref{jj1j2}). These are nothing but the expressions for $S(x, a)$ claimed in the statement of Theorem \ref{dgkmgen}. Along with \eqref{cauchyRes2}, this completes the proof of Theorem \ref{dgkmgen}.
\end{proof}

As remarked in the introduction, a special case of the above result, that is Theorem \ref{dgkmgen}, when $N$ is even, was previously obtained by Kanemitsu, Tanigawa and Yoshimoto \cite[Theorem 2.1]{ktyacta}. Before deriving their result from ours, we begin with Lemma 3.1 from \cite{dixitmaji1}.
\begin{lemma}\label{ktywigl}
For $a, u, v\in\mathbb{R}$, we have
\begin{align*}
2\textup{Re}\left(\frac{e^{iuv}}{\exp{\left(ae^{-iu}\right)}-1}\right)=\frac{\cos(a\sin(u)+uv)-e^{-a\cos(u)}\cos(uv)}{\cosh(a\cos(u))-\cos(a\sin(u))}.
\end{align*}
\end{lemma}
\begin{theorem}[Kanemitsu-Tanigawa-Yoshimoto \cite{ktyacta}]\label{ytk} 
Let $h' \ge 0$, $\ell \ge 0$ and $M \ge 1$ be fixed integers with $h'<M$,
and let $0<a\leq 1$ be a positive parameter. Let $x>0$. Let $A(y)=\pi(2\pi y)^{\frac{1}{2M}}$ and let
\begin{equation*}
\begin{split}
 a_j:=\cos\left(\tfrac{\pi}{2M}\bigl(\tfrac12-j\bigr)\right), b_j:=\sin\left(\tfrac{\pi}{2M}\bigl(\tfrac12-j\bigr)\right),\\
 B_j(n,h',\ell):=-2\pi a n-\frac{\pi(2h'+1)}{2M}\left(\frac12-j\right)
             -\frac{\pi(\ell-1)}{2},
\end{split}
\end{equation*}
and
\begin{align}\label{eq22}
f_j(n,h',\ell,x):=\frac{\cos(2A(\frac{n}{x})b_j+B_j(n,h',\ell))-e^{-2A(\frac{n}{x})a_j}
        \cos(B_j(n,h',\ell))}{\cosh\left(2A(\frac{n}{x})a_j\right)
        -\cos\left(2A(\frac{n}{x})b_j\right)}
\end{align}
Let 
\begin{align}\label{eq24}
P(x)&:=\zeta(2M(\ell+1)-2h')x^{-1}+\sum_{j=0}^{\ell}\frac{(-1)^j}{j!}\zeta(-j,a)
       \zeta\bigl(2M(\ell-j)-2h'\bigr)x^j\nonumber\\
			&\quad+\frac{1}{2M}\Gamma\left(-\ell+\frac{2h'+1}{2M}\right)\zeta\left(-\ell+
       \frac{2h'+1}{2M}, a\right)x^{\ell-\frac{2h'+1}{2M}}.
\end{align}
%\begin{align} 
%P(x)=&\, \zeta(2M(l+1)-2h')x^{-1} \\
 %    &+\sum_{j=0}^{l}\frac{(-1)^j}{j!}\zeta(-j,a)
  %     \zeta\bigl(2M(l-j)-2h'\bigr)x^j \nonumber \\
   %  &+\frac{1}{2M}\Gamma\left(-l+\frac{2h'+1}{2M}\right)\zeta\left(-l+
    %   \frac{2h'+1}{2M}, a\right)x^{l-\frac{2h'+1}{2M}},\nonumber
%\end{align}
Let $\Ppsum_{j}$ mean that the summation is performed over $j$, $j=-(M-1),-(M-3),\ldots,M-3,M-1$.
Then,\\
\begin{align} \label{eq23e}
&\sum_{n=1}^{\infty} \frac{1}{n^{2M\ell-2h'}}\, 
 \frac{\exp(-a n^{2M}x)}{1-\exp(-n^{2M}x)}=P(x)+U(x, a),
\end{align}
where
\begin{align}\label{uxameven}
U(x, a):=\frac{(-1)^{h'}}{2M}\left(\frac{2\pi}{x}\right)^{-\ell+\frac{2h'+1}{2M}}\hspace{1.5mm}
\!\!\!\! \sum_{j=-\frac{M}{2}}^{\frac{M}{2}-1}\sum_{n=1}^{\infty}
f_{2j+1}(n,h',\ell,x)n^{-1-\ell+\frac{2h'+1}{2M}}
\end{align}
for $M$ even, and 
\begin{align}\label{uxamodd}
U(x, a):=\frac{(-1)^{h'}}{2M}\left(\frac{2\pi}{x}\right)^{-\ell+\frac{2h'+1}{2M}}
\!\!\!\! \sum_{j=-\frac{M-1}{2}}^{\frac{M-1}{2}}\sum_{n=1}^{\infty}
f_{2j}(n,h',\ell,x)n^{-1-\ell+\frac{2h'+1}{2M}}.
\end{align}
for $M$ odd.
%\textup{(i)} If $M$ is even,
%\begin{align} \label{eq23e}
%&\sum_{n=1}^{\infty} \frac{1}{n^{2M\ell-2h'}}\, 
% \frac{\exp(-a n^{2M}x)}{1-\exp(-n^{2M}x)}\\[1ex]
%&=\frac{(-1)^{h'}}{2M}\left(\frac{2\pi}{x}\right)^{-\ell+\frac{2h'+1}{2M}}
%\!\!\!\! \sum_{j=-\frac{M}{2}}^{\frac{M}{2}-1}\ \ \sum_{n=1}^{\infty}
%f_{2j+1}(n,h',\ell,x)n^{-1-\ell+\frac{2h'+1}{2M}} %\nonumber \\
%+P(x).  \nonumber
%\end{align}
%\textup{(ii)} If $M$ is odd,
%\begin{align} \label{eq23o}
%&\sum_{n=1}^{\infty} \frac{1}{n^{2M\ell-2h'}}\, 
% \frac{\exp(-a n^{2M}x)}{1-\exp(-n^{2M}x)}\\[1ex]
%&=\frac{(-1)^{h'}}{2M}\left(\frac{2\pi}{x}\right)^{-\ell+\frac{2h'+1}{2M}}
%\!\!\!\! \sum_{j=-\frac{M-1}{2}}^{\frac{M-1}{2}}\ \ \sum_{n=1}^{\infty}
%f_{2j}(n,h',\ell,x)n^{-1-\ell+\frac{2h'+1}{2M}} %\nonumber \\
%+P(x).  \nonumber
%\end{align}
\end{theorem}
\begin{proof}
Substitute $N=2M$ and $h=M-h'+M \ell$ on both sides of \eqref{dgkmgeneqn}. Then the resulting left-hand side is the same as the Lambert series in \eqref{eq23e}.  With the above substitutions,
\begin{align}\label{pxapx}
P(x,a)&=\frac{1}{2M}\Gamma\left(-\ell+\frac{2h'+1}{2M}\right)\zeta\left(-\ell+\frac{2h'+1}{2M},a\right)x^{\ell-\frac{2h'+1}{2M}}\nonumber\\
&  +\sum_{j=0}^{\left\lfloor\frac{\ell}{2}-\frac{h'}{2M}\right\rfloor}\frac{\zeta(-2j,a)}{(2j)!}\zeta(2M(\ell-2j)-2h')x^{2j}+\frac{\zeta(2M(\ell+1)-2h')}{x}\nonumber\\ 
& +\sum_{j=1}^{\left\lfloor\frac{\ell}{2}-\frac{h'}{2M}+\frac{1}{2}\right\rfloor}\frac{(-1)^{2j-1}}{(2j-1)!}\zeta(-(2j-1),a)\zeta(-4Mj+2M-2h'+2M\ell)x^{2j-1}.
\end{align}
Note that $0<h'<M\ \Rightarrow \ 0<\frac{1}{2}-\frac{h'}{2m}<\frac{1}{2}\ \Rightarrow\ \left\lfloor \frac{1}{2}-\frac{h'}{2m}\right\rfloor$\ =\ 0, and $0<h'<M\ \Rightarrow \ -\frac{1}{2}<-\frac{h'}{2M}<0\ \Rightarrow \left\lfloor-\frac{h'}{2M}\right\rfloor=-1$. Thus,
\begin{equation}\label{floor1}
\left\lfloor \frac{\ell}{2}+\frac{1}{2}-\frac{h'}{2m}\right\rfloor= \begin{cases} 
      \frac{\ell}{2} ,&  \ell \ \text{is even} \\
     \frac{\ell-1}{2} ,& \ell \ \text{is odd}, 
   \end{cases}
\end{equation}
and
\begin{equation}\label{floor2}
\left\lfloor\frac{\ell}{2}-\frac{h'}{2M}\right\rfloor= \begin{cases} 
      \frac{\ell}{2}-1 ,&  \ell \ \text{is even} \\
     \frac{\ell-1}{2} ,& \ell \ \text{is odd}. 
   \end{cases}
\end{equation}
Using \eqref{floor1} and \eqref{floor2}, we see that, irrespective of the parity of $\ell$, the two finite sums over $j$ in \eqref{pxapx} combine together to give
\begin{equation*}
\sum_{j=0}^{\ell}\frac{(-1)^j}{(j)!}\zeta(-j,a)\zeta(2M(\ell-j)-2h')x^j,
\end{equation*}
which, when combined with the other expression in \eqref{pxapx}, shows that our $P(x, a)$ equals $P(x)$, which is defined in \eqref{eq24}. 

Next, we have to show that our $S(x, a)$ from \eqref{sevena} matches with the expressions for $U(x, a)$ in \eqref{uxameven} and \eqref{uxamodd} corresponding to $M$ even and $M$ odd respectively. We only prove this in the case when $M$ is even. That for $M$ odd can be similarly proved. 

Now substituting $N=2M$ and $h=M-h'+M\ell$, with $M$ even, say $M=2k$, in \eqref{sevena} and simplifying, we see that
 \begin{align*}
S(x,a)=&\frac{(-1)^{h'+1}}{4k}\left(\frac{2\pi}{x}\right)^{-\ell+\frac{2h'+1}{4k}}\sum_{n=1}^{\infty}n^{-1-\ell+\frac{1+2h'}{4k}}\sum_{j=-2k}^{2k-1}\exp\left(\frac{i\pi(1-4k+2h'-4k\ell)\left(j+\frac{1}{2}\right)}{4k}\right)\nonumber\\
&\quad\times \frac{\cos(2\pi na)+i(-1)^{j+1}\sin(2\pi na)}{\exp\left(2\pi\left(\frac{2\pi n}{x}\right)^{\frac{1}{4k}}e^{\frac{i\pi}{4k}\left(j+\frac{1}{2}\right)}\right)-1}.
\end{align*}
Now split the sum over $j$ according to the parity of $j$ and simplify so as to obtain
\begin{align*}
S(x,a)=&\frac{(-1)^{h'+1}}{4k}\left(\frac{2\pi}{x}\right)^{-\ell+\frac{2h'+1}{4k}}\sum_{n=1}^{\infty}n^{-1-\ell+\frac{1+2h'}{4k}}\nonumber\\
&\quad\times\Bigg\{\sum_{j=-k}^{k-1}\frac{\exp\left(-i\left(2\pi na-\frac{\pi(2h'+1)}{4k}\left(\frac{4j+1}{2}\right)+\frac{(\ell+1)}{2}(4j+1)\pi\right)\right)}{\exp\left(2\pi\left(\frac{2\pi n}{x}\right)^{\frac{1}{4k}}e^{\frac{i\pi}{4k}\left(\frac{4j+1}{2}\right)}\right)-1}\nonumber\\
&\quad\qquad+\sum_{j=-k}^{k-1}\frac{\exp\left(i\left(2\pi na+\frac{\pi(2h'+1)}{4k}\left(\frac{4j+3}{2}\right)-\frac{(\ell+1)}{2}(4j+3)\pi\right)\right)}{\exp\left(2\pi\left(\frac{2\pi n}{x}\right)^{\frac{1}{4k}}e^{\frac{i\pi}{4k}\left(\frac{4j+3}{2}\right)}\right)-1}\Bigg\}.
\end{align*}
Replace $j$ by $-j-1$ in the second sum and then observe that the resulting corresponding summands of the two sums are complex conjugates of each other so that
\begin{align*}
S(x,a)=\frac{(-1)^{h'+1}}{4k}\left(\frac{2\pi}{x}\right)^{-\ell+\frac{2h'+1}{4k}}\sum_{n=1}^{\infty}n^{-1-\ell+\frac{1+2h'}{4k}}
\sum_{j=-k}^{k-1}2\textup{Re}\left(\frac{e^{iuv}}{\exp{\left(a e^{-iu}\right)-1}}\right),
\end{align*}
where $a=2A\left(\frac{n}{x}\right), u=-\frac{\pi}{4k}\left(\frac{4j+1}{2}\right)$, and $uv=-2\pi an+\frac{\pi(2h'+1)}{4k}\left(\frac{4j+1}{2}\right)-\frac{\pi(\ell+1)(4j+1)}{2}$. Using Lemma \ref{ktywigl}, the notations in the hypotheses of Theorem \ref{ytk}, \eqref{eq22} and the fact that $k=M/2$, we deduce that
\begin{align*}
S(x,a)=\frac{(-1)^{h'}}{4k}\left(\frac{2\pi}{x}\right)^{-\ell+\frac{2h'+1}{4k}}\sum_{j=-\frac{M}{2}}^{\frac{M}{2}-1}\sum_{n=1}^{\infty}f_{2j+1}(n, h', \ell, x)n^{-1-\ell+\frac{1+2h'}{4k}},
\end{align*}
which is nothing but \eqref{uxameven}. Thus we derive \eqref{eq23e} from \eqref{dgkmgeneqn}. As remarked before, \eqref{uxamodd} can be proved by a similar argument.
\end{proof}
\section{A two-parameter generalization of Ramanujan's formula for $\zeta(2m+1)$}\label{ord2}
This section is devoted to proving Theorems \ref{dgkmord2} and \ref{ggram}, which, as will be seen, are equivalent to each other. We then give interesting special cases of Theorem \ref{ggram}. Before proving Theorem \ref{dgkmord2}, we begin with a lemma.
\begin{lemma}\label{iff}
Let $N$ be an odd positive integer. If $h>\frac{N}{2}$, then $\frac{N-2h+1}{N}=-2\left\lfloor\frac{h}{N}-\frac{1}{2}\right\rfloor$ if and only if $h=\frac{N+1}{2}+Nm$, where $m\in\mathbb{N}\cup\{0\}$.
\end{lemma}
\begin{proof}
Let $\left\lfloor\frac{h}{N}-\frac{1}{2}\right\rfloor=m$. Since $h>N/2$, we have $m\in\mathbb{N}\cup\{0\}$. Let $\frac{h}{N}-\frac{1}{2}=m+r$, where $0\leq r<1$. Then $h=\frac{N}{2}+N(m+r)$. From the hypothesis, $N-2h+1=-2Nm$. Since the last two equations imply $r=\frac{1}{2N}$, we get $h=\frac{N+1}{2}+Nm$. The other direction is trivial. 
\end{proof}
\begin{proof}[Theorem \textup{\ref{dgkmord2}}][]
The setup for the proof of this theorem is exactly similar to that of Theorem \ref{dgkmgen}. Hence we only give details where they differ from those of the latter.
 
Note that $a\in(0,1]$ is fixed, and our integrand $F(s)$, defined in \eqref{F(s)}, is 
%\begin{equation*}
$F(s):=\Gamma(s)\zeta(s,a)$\newline $\zeta(Ns-(N-2h))x^{-s}$.
%\end{equation*}
The poles of $\G(s)$ include negative even integers whereas $\zeta(Ns-(N-2h))$ has a simple pole at $s=\frac{N-2h+1}{N}$. Since $N$ is odd, it may happen that $\frac{N-2h+1}{N}=-2j$ for some positive integers $N$ and $j$. As explained in the introduction, Lemma \ref{equality} then implies that $j=\left\lfloor\frac{h}{N}-\frac{1}{2}\right\rfloor$. This may imply a double order pole of $F(s)$ at $s=\frac{N-2h+1}{N}=-2\left\lfloor\frac{h}{N}-\frac{1}{2}\right\rfloor$ if $\zeta\left(\frac{N-2h+1}{N}, a\right)\neq 0$, or a simple pole (or a removable singularity) if $\zeta\left(\frac{N-2h+1}{N}, a\right)=0$. 
However, even if $\zeta\left(\frac{N-2h+1}{N}, a\right)=0$, one may first calculate the residue assuming a double pole and then apply this fact, and the answer obtained would be same as that deduced by first applying $\zeta\left(\frac{N-2h+1}{N}, a\right)=0$ and then accordingly calculating the residue.

Thus the residue at $\frac{N-2h+1}{N}$ is given by
%Irrespective of the order of the pole, we can always calculate the residue assuming that we do not have any information whether or not $\zeta\left(\frac{N-2h+1}{N}, a\right)=0$, and if it is the case that $\zeta\left(\frac{N-2h+1}{N}, a\right)=0$, then the residue calculation obtained before will yield the same answer that one would get after appl
%In the introduction it has already been explained as to how, for $0<a<1$, the condition $\frac{N-2h+1}{N}=-2\left\lfloor\frac{h}{N}-\frac{1}{2}\right\rfloor\neq0$ leads to a double order pole of $F(s)$, defined in \eqref{F(s)}, at $s=\frac{N-2h+1}{N}$, unlike the simple pole when we considered $\frac{N-2h+1}{N}\neq-2\left\lfloor\frac{h}{N}-\frac{1}{2}\right\rfloor$ in Theorem \ref{dgkmgen}. (Note that $a=1$ does not lead to a double pole at  $s=\frac{N-2h+1}{N}=-2\left\lfloor\frac{h}{N}-\frac{1}{2}\right\rfloor$ as, in this case, $\zeta(s, a)=\zeta(s)$ also has a zero at $-2\left\lfloor\frac{h}{N}-\frac{1}{2}\right\rfloor$.) Hence the residue at this pole is given by
\begin{align}\label{R_N-2h+1/N*}
R_{\frac{N-2h+1}{N}}&=\lim_{s\to\frac{N-2h+1}{N}} \left(\frac{\mathrm{d}}{\mathrm{d}s} \left(s- \frac{N-2h +1}{N}\right)^2 \Gamma(s)\ \zeta(s,a)\ \zeta(Ns - (N-2h))x^{-s}\right)\nonumber\\
%&= \frac{x^{2{\left\lfloor{\frac{h}{N} - \frac{1}{2}}\right\rfloor}}}{N(2{\lfloor{\frac{h}{N} - \frac{1}{2}}\rfloor})!} \Bigg(\zeta\left(-2{\left\lfloor{\tfrac{h}{N} - \tfrac{1}{2}}\right\rfloor},a\right)\bigg((N-1)\gamma + \sum_{l=1}^{2{\lfloor{\frac{h}{N} - \frac{1}{2}}\rfloor}}\frac{1}{l} - \log(x)\bigg)\nonumber \\
%& \hspace{3cm}+\left.\frac{\partial }{\partial s} \zeta(s,a)\right|_{s=-2{\lfloor{\frac{h}{N} - \frac{1}{2}}\rfloor}} \Bigg)\nonumber\\
%%&= \frac{x^{-2{\lfloor{\frac{h}{N} - \frac{1}{2}}\rfloor}}}{N((2{\lfloor{\frac{h}{N} - \frac{1}{2}}\rfloor})!)} \Bigg(\zeta\left(-2{\left\lfloor{\frac{h}{N} - \frac{1}{2}}\right\rfloor},a\right)~ %\left(\psi\left({2\left\lfloor{\frac{h}{N} - \frac{1}{2}}\right\rfloor} +1\right) + N\gamma - \log x\right)\nonumber \\ 
%%&\hspace{3cm} + \zeta'\left(-2{\left\lfloor{\frac{h}{N} - \frac{1}{2}}\right\rfloor},a\right) \Bigg)\nonumber\\
&= \frac{x^{2{\lfloor{\frac{h}{N} - \frac{1}{2}}\rfloor}}}{N(2{\lfloor{\frac{h}{N} - \frac{1}{2}}\rfloor})!} \Bigg\{ -\frac{B_{2\lfloor{\frac{h}{N} - \frac{1}{2}}\rfloor +1} (a)}{2\lfloor{\frac{h}{N} - \frac{1}{2}}\rfloor +1}~\left(\psi\left({2\left\lfloor{\frac{h}{N} - \frac{1}{2}}\right\rfloor} +1\right)+ N\gamma - \log x\right)\nonumber\\
&\hspace{3.5cm} + \zeta'\left(-2{\left\lfloor{\frac{h}{N} - \frac{1}{2}}\right\rfloor},a\right)\Bigg\},
\end{align}
%since by an application of \eqref{psifunc}, we have
%\begin{align*}
%\sum_{l=1}^{2{\lfloor{\frac{h}{N} - \frac{1}{2}}\rfloor}}\frac{1}{l}=\sum_{l=1}^{2{\lfloor{\frac{h}{N} - \frac{1}{2}}\rfloor}}\left(\psi(l+1) - \psi(l)\right)=\psi\left( 2\left\lfloor{\frac{h}{N} - \frac{1}{2}}\right\rfloor +1 \right)+\gamma.
%\end{align*}
Thus, from \eqref{R_0}, \eqref{R_1}, \eqref{R_-2j}, \eqref{R_-(2j-1)}, \eqref{R_N-2h+1/N*} and \eqref{sodda}, we see that 
\begin{align}\label{dgkmord2p}
\sum_{n=1}^{\infty}n^{N-2h}\frac{\textup{exp}(-an^{N}x)}{1-\textup{exp}(-n^{N}x)}=P^{*}(x,a)+S(x,a),
\end{align}
where
\begin{align}\label{psxabef}
P^{*}(x,a)&:=-\left(a-\frac{1}{2}\right)\zeta(-N+2h)+\frac{\zeta(2h)}{x}-\sum_{j=1}^{\left\lfloor\frac{h}{N}-\frac{1}{2}\right\rfloor-1}\frac{B_{2j+1}(a)}{(2j+1)!}\zeta\left(2h-(2j+1)N\right)x^{2j}\nonumber\\
&\quad+\frac{x^{2\left\lfloor\frac{h}{N}-\frac{1}{2}\right\rfloor}}{N\left(2\left\lfloor\tfrac{h}{N}-\tfrac{1}{2}\right\rfloor\right)!}\bigg\{-\frac{B_{2\left\lfloor\frac{h}{N}-\frac{1}{2}\right\rfloor+1}(a)}{2\left\lfloor\frac{h}{N}-\frac{1}{2}\right\rfloor+1}\left(\psi\left(2\left\lfloor\tfrac{h}{N}-\tfrac{1}{2}\right\rfloor+1\right)+N\g-\log x\right)\nonumber\\
&\quad+\zeta'\left(-2\left\lfloor\tfrac{h}{N}-\tfrac{1}{2}\right\rfloor,a\right)\bigg\}+(-1)^{h+1}2^{2h-1}\pi^{2h}\sum_{j=1}^{\left\lfloor\frac{h}{N}\right\rfloor}\left(\frac{-1}{4\pi^2}\right)^{jN}\frac{B_{2j}(a)B_{2h-2jN}}{(2j)!(2h-2jN)!}x^{2j-1},
\end{align}
and the calculation for $S(x, a)$ remains the same exactly as in proof of Theorem \ref{dgkmgen}. 

As we now show, \eqref{dgkmord2p} can be simplified to a great extent using the following result of Koyama and Kurokawa \cite[p.~7]{koyakuro} for an even positive integer $k$ and $0<a\leq 1$:
{\allowdisplaybreaks\begin{align}\label{kaykay}
\zeta'(-k,a)&=\frac{2(-1)^{\frac{k}{2}}k!}{(2\pi)^{k+1}}\bigg\{\sum_{n=1}^{\infty}\frac{(\log n)\sin(2\pi na)}{n^{k+1}}+\left(\log(2\pi)-\psi(k+1)\right)\sum_{n=1}^{\infty}\frac{\sin(2\pi na)}{n^{k+1}}\nonumber\\
&\qquad\qquad\qquad+\frac{\pi}{2}\sum_{n=1}^{\infty}\frac{\cos(2\pi na)}{n^{k+1}}\bigg\}.
\end{align}}
Even though Koyama and Kurokawa write $\log(2\pi)+\g-\left(1+\tfrac{1}{2}+\cdots+\tfrac{1}{k}\right)$ in place of $\left(\log(2\pi)-\psi(k+1)\right)$, it is easy to see with the help of \eqref{psifunc} that they are equal. Also, even though they work with $0<a<1$, it is easy to see that the formula holds for $a=1$ as long as $k$ is even, $k>0$, and is then a well-known result, see for example, \cite[Equation (1)]{ktyham}:
\begin{equation*}
\zeta'(-k)=\frac{1}{2}(-1)^{\frac{k}{2}}(2\pi)^{-k}(k!)\zeta(k+1).
\end{equation*}
It is important to note that \eqref{kaykay} also holds for $k=0$ but only for $0<a<1$, and is then an equivalent form of the well-known Kummer formula for $\log\G(a)$ \cite[p.~4]{kummer}.

By Lemma \ref{iff}, we know that for $h>N/2$, we have $\frac{N-2h+1}{N}=-2\left\lfloor\frac{h}{N}-\frac{1}{2}\right\rfloor$ if and only if $h=\frac{N+1}{2}+Nm$, where $m\in\mathbb{N}$. Thus, we let $h=\frac{N+1}{2}+Nm, m\in\mathbb{N}$ in \eqref{dgkmord2p}. We employ \eqref{kaykay} with $k=2m, m>0,$ in the expression for the residue in \eqref{psxabef} arising due to double pole to simplify it as
\begin{align}\label{muk1}
&\frac{x^{2m}}{N(2m)!}\left\{-\frac{B_{2m+1}(a)}{2m+1}\left(\psi(2m+1)+N\g-\log x\right)+\zeta'(-2m,a)\right\}\nonumber\\
&=\frac{x^{2m}B_{2m+1}(a)}{N(2m+1)!}\left(-N\g+\log\left(\frac{x}{2\pi}\right)\right)+\frac{2(-1)^{m}x^{2m}}{N(2\pi)^{2m+1}}\left(\sum_{n=1}^{\infty}\frac{(\log n)\sin(2\pi na)}{n^{2m+1}}+\frac{\pi}{2}\sum_{n=1}^{\infty}\frac{\cos(2\pi na)}{n^{2m+1}}\right),
\end{align}
where, in the course of simplification, the series $\sum_{n=1}^{\infty}\frac{\sin(2\pi na)}{n^{2m+1}}$ is expressed in terms of Bernoulli polynomials using their Fourier expansion \cite[p.~805]{as}:
\begin{equation}\label{fou}
B_{2m+1}(a)=\frac{2(-1)^{m+1}(2m+1)!}{(2\pi)^{2m+1}}\sum_{n=1}^{\infty}\frac{\sin(2\pi na)}{n^{2m+1}}.
\end{equation}
Moreover, part of the expression for $S(x, a)$ in \eqref{sodda} can be simplified, namely,
\begin{align}\label{muk2}
&\frac{(-1)^{h+1}}{N}\left(\frac{2\pi}{x}\right)^{\frac{N-2h+1}{N}}\sum_{j=-\frac{(N-1)}{2}}^{\frac{(N-1)}{2}}\textup{exp}\left(\frac{i\pi(1-2h)j}{N}\right)\frac{(-1)^{j+\frac{N+1}{2}}}{\pi}\sum_{n=1}^{\infty}\frac{\sin(2\pi na)}{n^{\frac{2h-1}{N}}}\log\left(\tfrac{1}{\pi}A_{N,j}\left(\tfrac{n}{x}\right)\right)\nonumber\\
&=\frac{(-1)^{m+\frac{N+3}{2}}}{N}\left(\frac{x}{2\pi}\right)^{2m}\sum_{j=-\frac{(N-1)}{2}}^{\frac{(N-1)}{2}}\frac{(-1)^{2j+\frac{N+1}{2}}}{\pi}\sum_{n=1}^{\infty}\frac{\sin(2\pi na)}{n^{2m+1}}\log\left(\left(\frac{2\pi n}{x}\right)^{\frac{1}{N}}e^{\frac{i\pi j}{N}}\right)\nonumber\\
&=\frac{(-1)^{m+1}}{\pi N}\left(\frac{x}{2\pi}\right)^{2m}\sum_{j=-\frac{(N-1)}{2}}^{\frac{(N-1)}{2}}\left\{\left(\frac{1}{N}\log\left(\frac{2\pi}{x}\right)+\frac{i\pi j}{N}\right)\sum_{n=1}^{\infty}\frac{\sin(2\pi na)}{n^{2m+1}}+\frac{1}{N}\sum_{n=1}^{\infty}\frac{(\log n)\sin(2\pi na)}{n^{2m+1}}\right\}\nonumber\\
&=\frac{x^{2m}}{N(2m+1)!}\log\left(\frac{2\pi}{x}\right)B_{2m+1}(a)-\frac{(-1)^{m}}{\pi N}\left(\frac{x}{2\pi}\right)^{2m}\sum_{n=1}^{\infty}\frac{(\log n)\sin(2\pi na)}{n^{2m+1}},
\end{align}
where in the last step, we again used \eqref{fou}. Now combine \eqref{muk1} and \eqref{muk2} to deduce that
\begin{align}\label{muk3}
&\frac{x^{2m}}{N(2m)!}\left\{-\frac{B_{2m+1}(a)}{2m+1}\left(\psi(2m+1)+N\g-\log x\right)+\zeta'(-2m,a)\right\}\nonumber\\
&+\frac{(-1)^{h+1}}{N}\left(\frac{2\pi}{x}\right)^{\frac{N-2h+1}{N}}\sum_{j=-\frac{(N-1)}{2}}^{\frac{(N-1)}{2}}\textup{exp}\left(\frac{i\pi(1-2h)j}{N}\right)\frac{(-1)^{j+\frac{N+1}{2}}}{\pi}\sum_{n=1}^{\infty}\frac{\sin(2\pi na)}{n^{\frac{2h-1}{N}}}\log\left(\tfrac{1}{\pi}A_{N,j}\left(\tfrac{n}{x}\right)\right)\nonumber\\
&=-\g B_{2m+1}(a)\frac{x^{2m}}{(2m+1)!}+\frac{(-1)^{m}x^{2m}\pi}{N(2\pi)^{2m+1}}\sum_{n=1}^{\infty}\frac{\cos(2\pi na)}{n^{2m+1}}.
\end{align}
Substituting \eqref{muk3} in \eqref{dgkmord2p} and noting that $m=\left\lfloor\frac{h}{N}-\frac{1}{2}\right\rfloor=\frac{2h-1-N}{2N}$ leads us to \eqref{psxa}.
\end{proof}

\begin{proof}[Theorem \textup{\ref{ggram}}][]
Let $h=\frac{N+1}{2}+Nm, m>0$, $x=2^{N}\a$ and $\a\b^{N}=\pi^{N+1}$ in Theorem \ref{dgkmord2}. To write the sum over $j$ going from $0$ to $\left\lfloor\tfrac{h}{N}\right\rfloor$ in terms of $\a$ and $\b$, we use the fact that
\begin{align}\label{e2c}
\pi(2\pi)^{(2m+1)N-2jN}x^{2j-1}=2^{2Nm}\a^{2j+\frac{2N}{N+1}(m-j)}\b^{N+\frac{2N^2}{N+1}(m-j)}.
\end{align}
Now rearrange the terms of the resulting identity upon the aforementioned substitutions, multiply both sides of the rearranged identity by $\a^{-2Nm/(N+1)}$ , and then simplify to arrive at \eqref{zetageneqna}.
\end{proof}
%%As shown below \eqref{zetageneqn}, which itself is a generalization of \eqref{zetaodd}, that is, Ramanujan's formula for $\zeta(2m+1)$, is a special case of Theorem \ref{ggram}.
%%\begin{proof}[\textup{\eqref{zetageneqn}}][]
%%Let $a=1$ in Theorem \ref{ggram} and simplify.
%Not only is the series $\displaystyle\sum_{n=1}^{\infty}\frac{n^{-2m-1}\cos(2\pi na)}{\textup{exp}\left((2n)^{\frac{1}{N}}\b e^{\frac{i\pi j}{N}}\right)-1}$ absolutely and uniformly convergent on $0<a<1$, as explained in the introduction \textbf{[Explain!]}, the series 
%\begin{equation*}
%\sum_{n=1}^{\infty}\frac{\sin(2\pi na)}{n^{2m+1}}\left\{\log\left(\tfrac{\beta}{2\pi}(2n)^{\frac{1}{N}} e^{\frac{i\pi j}{N}}\right)-\frac{1}{2}\left(\psi\left(\tfrac{i\beta}{2\pi}(2n)^{\frac{1}{N}} e^{\frac{i\pi j}{N}}\right)+\psi\left(\tfrac{-i\beta}{2\pi}(2n)^{\frac{1}{N}} e^{\frac{i\pi j}{N}}\right)\right)\right\}
%\end{equation*}
%too is absolutely and uniformly convergent in $0<a<1$. Hence we can interchange the order of limit and summation in both. Also 
%Note that \cite[Equation (1)]{ktyham}
%\begin{equation}\label{derzetaodd}
%\zeta'(-2m,1)=\zeta'(-2m)=\frac{1}{2}(-1)^m(2\pi)^{-2m}(2m)!\zeta(2m+1),
%\end{equation}
%which can be easily proved by using the Hurwitz formula \eqref{hformula}. These together give \eqref{zetageneqn}.
%\end{proof}
Letting $N=1$ in Theorem \ref{ggram} gives the following result which can be thought of as a different one-parameter generalization, as compared to \eqref{zetageneqn}, of \eqref{zetaodd}.
\begin{theorem}\label{ggramN1}
Let $0<a\leq 1$. Let $\a,\b>0$ such that $\a\b=\pi^{2}$. Then for $m\in\mathbb{Z}, m>0$,
{\allowdisplaybreaks\begin{align*}
&\a^{-m}\left(\left(a-\frac{1}{2}\right)\zeta(2m+1)+\sum_{j=1}^{m-1}\frac{B_{2j+1}(a)}{(2j+1)!}\zeta(2m+1-2j)(2\a)^{2j}+\sum_{n=1}^{\infty}\frac{n^{-2m-1}\textup{exp}\left(-2an\a\right)}{1-\textup{exp}\left(-2n\a\right)}\right)\nonumber\\
&=(-\b)^{-m}\bigg[\frac{(-1)^{m+1}(2\pi)^{2m}B_{2m+1}(a)\g}{(2m+1)!}+\frac{1}{2}\sum_{n=1}^{\infty}\frac{\cos(2\pi na)}{n^{2m+1}}+\sum_{n=1}^{\infty}\frac{n^{-2m-1}\cos(2\pi na)}{\textup{exp}\left(2n\b\right)-1}\nonumber\\
&\quad+\frac{1}{2\pi}\sum_{n=1}^{\infty}\frac{\sin(2\pi na)}{n^{2m+1}}\left(\psi\left(\tfrac{in\beta}{\pi}\right)+\psi\left(\tfrac{-in\beta}{\pi}\right)\right)\bigg]+(-1)^{m}2^{2m}\sum_{j=0}^{m+1}\frac{(-1)^jB_{2j}(a)B_{2m-2j+2}}{(2j)!(2m-2j+2)!}\a^{j}\b^{m+1-j}.
\end{align*}}
\end{theorem}
We now give corollaries of Theorem \ref{ggram} when $a$ takes special values in the interval $(0,1)$. 
\subsection{Special case $a=1/2$ of Theorem \ref{ggram}}\label{splhalf} 
\begin{proof}[Theorem \textup{\ref{ggramhalf}}][]
Let $a=1/2$ in Theorem \ref{ggram}. To simplify, we use \cite[p.~4]{temme},
\begin{equation}\label{bjhalf}
B_j\left(\tfrac{1}{2}\right)=(2^{1-j}-1)B_{j}.
\end{equation}
Along with the fact that $B_{2j+1}=0$, this implies that 
\begin{equation}\label{b2jp1}
B_{2j+1}\left(\frac{1}{2}\right)=0.
\end{equation}
We also employ the identity $\sum_{n=1}^{\infty}(-1)^{n}n^{-2m-1}=(2^{-2m}-1)\zeta(2m+1)$.
%Also differentiating both sides of \eqref{hzetahalf} with respect to $s$, letting $s=-2m$ and then using \eqref{derzetaodd}, we obtain
%\begin{align}
%\zeta'\left(-2m,\tfrac{1}{2}\right)=\frac{1}{2}\left(2^{-2m}-1\right)\zeta(2m+1).
%\end{align}
These together imply \eqref{zetagenhalfa}.
\end{proof}
\begin{proof}[Corollary \textup{\ref{transzetaodd}}][]
Subtract the complete expression in square brackets in \eqref{zetagenhalfa} from its both sides, multiply both sides of the resulting identity by $\a^{\frac{2Nm}{N+1}}$, and then let $\a=\b=\pi$. The finite sum on the right side of the resulting identity then becomes a polynomial in $\pi$ with non-zero rational coefficients. Since $\pi$ is transcendental, this proves the result.
\end{proof}
Corollaries \ref{rtype} and \ref{criterionzeta} have been already proved in the introduction, hence we refrain from repeating the arguments here.
%\begin{proof}[Corollary \textup{\ref{rtype}}][]
%This follows by allowing $N$ to run over \emph{all} positive odd integers because Corollary \ref{transzetaodd} produces a transcendental number for every such $N$.
%\end{proof}
\subsection{Proof of Corollary \ref{zeta311}}\label{splfourth}
Let $E_{k}$ denote the $k^{\textup{th}}$ Euler number, defined by means of the generating function 
\begin{equation*}
\frac{1}{\cosh z}=\sum_{k=0}^{\infty}\frac{E_k}{k!}z^{k}\hspace{4mm} (|z|<\tfrac{1}{2}\pi).
\end{equation*}
Let $a=1/4$ in Theorem \ref{ggram} to obtain
%Let $N$ be an odd positive integer and $\a,\b>0$ such that $\a\b^{N}=\pi^{N+1}$. Let $E_{k}$ denote the $k^{\textup{th}}$ Euler number. Then for $N\geq 1$ and a positive integer $m$,
{\allowdisplaybreaks\begin{align*}
&\a^{-\frac{2Nm}{N+1}}\bigg(-\frac{1}{4}\zeta(2Nm+1)-\frac{1}{4}\sum_{j=1}^{m-1}\frac{E_{2j}}{(2j)!}\zeta(2Nm+1-2jN)(2^{N-2}\a)^{2j}\nonumber\\&\qquad\quad+\sum_{n=1}^{\infty}\frac{n^{-2Nm-1}\textup{exp}\left(-\tfrac{1}{4}(2n)^{N}\a\right)}{1-\textup{exp}\left(-(2n)^{N}\a\right)}\bigg)\nonumber\\
&=\frac{\left(-\b^{\frac{2N}{N+1}}\right)^{-m}2^{2m(N-1)}}{N}\bigg[\frac{(-1)^m\pi^{2m}N\g E_{2m}}{2^{2m+2}(2m)!}+\frac{(2^{-2m}-1)}{2^{2m+2}}\zeta(2m+1)\nonumber\\
&\quad+(-1)^{\frac{N+3}{2}}\sum_{j=\frac{-(N-1)}{2}}^{\frac{N-1}{2}}(-1)^{j}\bigg\{\frac{1}{2^{2m+1}}\sum_{n=1}^{\infty}\frac{(-1)^nn^{-2m-1}}{\textup{exp}\left((4n)^{\frac{1}{N}}\b e^{\frac{i\pi j}{N}}\right)-1}\nonumber\\
&\quad+\frac{(-1)^{j+\frac{N+3}{2}}}{2\pi}\sum_{n=1}^{\infty}\frac{\sin\left(\frac{n\pi}{2}\right)}{n^{2m+1}}\left(\psi\left(\tfrac{i\beta}{2\pi}(2n)^{\frac{1}{N}} e^{\frac{i\pi j}{N}}\right)+\psi\left(\tfrac{-i\beta}{2\pi}(2n)^{\frac{1}{N}} e^{\frac{i\pi j}{N}}\right)\right)\bigg\}\bigg]\nonumber\\
&\quad+(-1)^{m+\frac{N+3}{2}}2^{2Nm}\sum_{j=0}^{\left\lfloor\frac{N+1}{2N}+m\right\rfloor}\frac{(-1)^{j+1}2^{-2j}\left(1-2^{1-2j}\right)B_{2j}B_{N+1+2N(m-j)}}{(2j)!(N+1+2N(m-j))!}\a^{\frac{2j}{N+1}}\b^{N+\frac{2N^2(m-j)}{N+1}},
\end{align*}}
%\begin{align}\label{zetagenfourtha}
%&\a^{-\frac{2Nm}{N+1}}\left(-\frac{1}{4}\zeta(2Nm+1)-\frac{1}{4}\sum_{j=1}^{m-1}\frac{E_{2j}}{(2j)!}\zeta(2Nm+1-2jN)(2^{N-2}\a)^{2j}+\sum_{n=1}^{\infty}\frac{n^{-2Nm-1}\textup{exp}\left(-\tfrac{1}{4}(2n)^{N}\a\right)}{1-\textup{exp}\left(-(2n)^{N}\a\right)}\right)\nonumber\\
%&=\frac{\left(-\b^{\frac{2N}{N+1}}\right)^{-m}2^{2m(N-1)}}{N}\bigg[\frac{(-1)^m(2\pi)^{2m}}{(2m)!}\left\{4^{-(2m+1)}E_{2m}\left(\psi(2m+1)+N\g-\log(2^{N}\a)\right)+\zeta'\left(-2m,\tfrac{1}{4}\right)\right\}\nonumber\\
%&\quad+(-1)^{\frac{N+3}{2}}\sum_{j=\frac{-(N-1)}{2}}^{\frac{N-1}{2}}(-1)^{j}\bigg(\frac{1}{2^{2m+1}}\sum_{n=1}^{\infty}\frac{(-1)^nn^{-2m-1}}{\textup{exp}\left((4n)^{\frac{1}{N}}\b e^{\frac{i\pi j}{N}}\right)-1}+\frac{1}{\pi}(-1)^{j+\frac{N+1}{2}}\nonumber\\
%&\quad\times\sum_{n=1}^{\infty}\frac{\sin\left(\frac{n\pi}{2}\right)}{n^{2m+1}}\left\{\log\left(\tfrac{\beta}{2\pi}(2n)^{\frac{1}{N}} e^{\frac{i\pi j}{N}}\right)-\frac{1}{2}\left(\psi\left(\tfrac{i\beta}{2\pi}(2n)^{\frac{1}{N}} e^{\frac{i\pi j}{N}}\right)+\psi\left(\tfrac{-i\beta}{2\pi}(2n)^{\frac{1}{N}} e^{\frac{i\pi j}{N}}\right)\right)\right\}\bigg)\bigg]\nonumber\\
%&\quad+(-1)^{m+\frac{N+3}{2}}2^{2Nm}\sum_{j=0}^{\left\lfloor\frac{N+1}{2N}+m\right\rfloor}\frac{(-1)^{j+1}2^{-2j}\left(1-2^{1-2j}\right)B_{2j}B_{N+1+2N(m-j)}}{(2j)!(N+1+2N(m-j))!}\a^{\frac{2j}{N+1}}\b^{N+\frac{2N^2(m-j)}{N+1}},
%\end{align}
since \cite[p.~26]{magobersoni}
\begin{equation*}
B_{n}\left(\tfrac{1}{4}\right)=-nE_{n-1}4^{-n}-2^{-n}(1-2^{1-n})B_n
\end{equation*}
and $E_{2n+1}=0$. Now let $\a=\b=\pi$, $m=5$ and $N=1$ in the above identity and simplify.

\section{A two-parameter generalization of the transformation formula of $\log\eta(z)$}\label{limiting}
Here we prove Theorem \ref{dgkmord2m0} which is a two-parameter generalization of the transformation formula of the logarithm of the Dedekind eta-function stated in \eqref{logdede}.
\begin{proof}[Theorem \textup{\ref{dgkmord2m0}}][]
Before we prove Theorem \ref{dgkmord2m0}, it is important to know how it differs from Theorem \ref{dgkmord2}. In Theorem \ref{dgkmord2}, the condition $\frac{N-2h+1}{N}=-2\left\lfloor\frac{h}{N}-\frac{1}{2}\right\rfloor\neq 0$ suggested that we separately consider the contribution $-\left(a-\frac{1}{2}\right)\zeta(-N+2h)$ arising due to the simple pole of $\G(s)$ at $s=0$. 

However, in Theorem \ref{dgkmord2m0}, we have the condition $\frac{N-2h+1}{N}=-2\left\lfloor\frac{h}{N}-\frac{1}{2}\right\rfloor=0$, that is, $h=\frac{N+1}{2}$, which means that the integrand $F(s)$, defined in \eqref{F(s)}, has a double order pole at $s=0$ except when $a=1/2$ as will be explained below. So we can as well use the same formula that we used in the proof of Theorem \ref{dgkmord2} to calculate the residue at the double order pole at $s=\frac{N-2h+1}{N}=-2\left\lfloor\frac{h}{N}-\frac{1}{2}\right\rfloor\neq 0$, that is \eqref{R_N-2h+1/N*}, to calculate the residue at the double order pole at $s=\frac{N-2h+1}{N}=-2\left\lfloor\frac{h}{N}-\frac{1}{2}\right\rfloor=0$ in Theorem \ref{dgkmord2m0}. But then
%Thus we calculate the residue at this double order pole through \eqref{R_N-2h+1/N*},
the term $-\left(a-\frac{1}{2}\right)\zeta(-N+2h)$ appearing in Theorem \ref{dgkmord2} does not appear in this context. Note also that \cite[p.~264, Equation (17)]{apostol-1998a} $\zeta(0,a)=\frac{1}{2}-a\neq 0$, except when $a=\frac{1}{2}$, which indeed means that we have a double order pole when $a\neq\frac{1}{2}$. Also when $a=\frac{1}{2}$, even though we get a simple pole at $s=0$, one can always apply \eqref{R_N-2h+1/N*} in this case too and get the correct residue contribution.

\emph{Taking the above thing into account}, we let $h=\frac{N+1}{2}$ in \eqref{dgkmord2p} and simplify the resultant using the facts \cite[p.~3]{temme} $B_1(a)=\left(a-\frac{1}{2}\right)$, \cite[p.~54]{temme} $\psi(1)=-\g$ and \cite[Equations (9), $(22^{a})$]{lerchform} $\zeta'(0,a)=\log\G(a)-\frac{1}{2}\log(2\pi)$. This results in \eqref{limalim}.

The variant of \eqref{limalim}, that is, \eqref{limab} can be proved by letting $x=2^{N}\a, \a\b^{N}=\pi^{N+1}$ in \eqref{limalim}, making use of \eqref{e2c} with $m=0$ and then by simplifying the resultant.
\end{proof}
\begin{proof}[Corollary \textup{\ref{dgkmord2m0kum}}][]
As mentioned in the proof of Theorem \ref{dgkmord2m0}, the term $-\left(a-\frac{1}{2}\right)\zeta(-N+2h)$ does not appear when $\frac{N-2h+1}{N}=-2\left\lfloor\frac{h}{N}-\frac{1}{2}\right\rfloor=0$. With this understanding, we let $h=\frac{N+1}{2}$ in Theorem \ref{dgkmord2} and while simplifying, we use following formula valid for $0<a<1$ \cite[p.~45, Formula \textbf{1.441.2}]{grn}:
\begin{equation*}
\sum_{n=1}^{\infty}\frac{\cos(2\pi na)}{n}=-\frac{1}{2}\log(2(1-\cos(2\pi a))).
\end{equation*}
This results in \eqref{limalimsim}.
\end{proof}

\begin{proof}[Corollary \textup{\ref{limitingahalfN1}}][]
Let $x=2^{N}\a$ and $\a\b^{N}=\pi^{N+1}$ in \eqref{limalimsim} so as to obtain
\begin{align}\label{limabsim}
&\sum_{n=1}^{\infty}\frac{\textup{exp}(-a(2n)^{N}\a)}{n(1-\textup{exp}(-(2n)^{N}\a))}-\frac{1}{N}(-1)^{\frac{N+3}{2}}\sum_{j=\frac{-(N-1)}{2}}^{\frac{N-1}{2}}(-1)^{j}\bigg\{\sum_{n=1}^{\infty}\frac{\cos(2\pi na)}{n\left(\textup{exp}\left((2n)^{\frac{1}{N}}\b e^{\frac{i\pi j}{N}}\right)-1\right)}\nonumber\\
&\quad+\frac{1}{2\pi}(-1)^{j+\frac{N+3}{2}}\sum_{n=1}^{\infty}\frac{\sin(2\pi na)}{n}\left(\psi\left(\tfrac{i\beta}{2\pi}(2n)^{\frac{1}{N}} e^{\frac{i\pi j}{N}}\right)+\psi\left(\tfrac{-i\beta}{2\pi}(2n)^{\frac{1}{N}} e^{\frac{i\pi j}{N}}\right)\right)\bigg\}\nonumber\\
&=\g\left(\frac{1}{2}-a\right)-\frac{\log\left(2\sin(\pi a)\right)}{2N}
%&=\frac{1}{N}\left(\left(\tfrac{1}{2}-a\right)\left((N-1)\gamma-\log\left(2^{N}\a\right)\right)+\log\G(a)-\frac{1}{2}\log 2\pi\right)\nonumber\\
+(-1)^{\frac{N+3}{2}}\sum_{j=0}^{\left\lfloor\frac{N+1}{2N}\right\rfloor}\frac{(-1)^jB_{2j}(a)B_{N+1-2Nj}}{(2j)!(N+1-2Nj)!}\a^{\frac{2j}{N+1}}\b^{N-\frac{2N^2j}{N+1}}.
\end{align}
Now let $a=1/2, N=1$ in \eqref{limabsim} and simplify.
\end{proof}
\begin{proof}[Corollary \textup{\ref{nesterenkotype}}][]
Let $a=1/4, N=1$ in \eqref{limabsim} and simplify. This leads to \eqref{gammalogirr}. Also, \eqref{gammalogirrpi} follows from \eqref{gammalogirr} by letting $\a=\b=\pi$.
\end{proof}
\begin{proof}[Corollary \textup{\ref{irreuler1}}][]
If $\a, \b$ and $\log 2$ are linearly independent over $\mathbb{Q}$, then the right-hand side of Corollary \ref{nesterenkotype} is irrational. That forces at least one of 
\begin{align*}
\g, \hspace{1mm}\sum_{n=1}^{\infty}\frac{e^{3 n\a/2}}{n(e^{2n\a}-1)},\hspace{1mm}\sum_{n=1}^{\infty}\frac{(-1)^n}{n\left(e^{4n\b}-1\right)},\text{and}\hspace{1mm}\frac{1}{2\pi}\sum_{n=1}^{\infty}\frac{(-1)^n}{2n-1}\left(\psi\left(\tfrac{i\b}{\pi}(2n-1)\right)+\psi\left(-\tfrac{i\b}{\pi}(2n-1)\right)\right)
\end{align*}
to be irrational.
\end{proof}
\begin{proof}[Corollary \textup{\ref{irreuler2}}][]
This follows from \eqref{gammalogirrpi} since $\pi$ and $\log 2$ are linearly independent over $\mathbb{Q}$.
\end{proof}
\begin{corollary}
\begin{align}\label{ramlagber}
\sum_{n=1}^{\infty}\frac{e^{n\pi}}{n\left(e^{2n\pi}-1\right)}-\sum_{n=1}^{\infty}\frac{(-1)^n}{n\left(e^{2n\pi}-1\right)}=-\frac{1}{2}\log 2+\frac{\pi}{8}.
\end{align}
\end{corollary}
\begin{proof}
Let $N=1, a=1/2$, $\a=\b=\pi$ in \eqref{limabsim} and simplify.
\end{proof}
Equation \ref{ramlagber}, given in an equivalent form in \cite[p.~169]{berndtrocky}, is a special case of a result in Ramanujan's Notebooks \cite[Vol. I, p.~257, no. 12; Vol. II, p.~169, no. 8(ii)]{ramnote} which was rediscovered by Lagrange \cite{lagrange}. See \cite[pp.~168-169]{berndtrocky} for more details.
\section{Proof of Theorem \ref{dgkmgenh0Nby2} and its special cases}\label{0hNby2}

\begin{proof}[Theorem \textup{\ref{dgkmgenh0Nby2}}][]
Since the proof is very similar to that of Theorem \ref{dgkmgen}, we will be very brief. 

As before, \eqref{mainequality} holds, but now for Re$(s)=\l>\max\left(\frac{N-2h+1}{N},1\right)$. We choose the contour $[\lambda-iT, \lambda+iT], [\lambda+iT,-r+iT], [-r +iT, -r -iT] \ \text{and}\ [-r-iT, \lambda-iT]$, where, $r$ is  positive real number such that $0<r<\frac{1}{N}$, the reason for which will be clear soon. The poles of the integrand $F(s)$, defined in \eqref{F(s)}, that are enclosed in the contour are the simple poles at $s=0, 1$ and $\frac{N-2h+1}{N}$, the residues of whom are same as those calculated in \eqref{R_0}, \eqref{R_1} and \eqref{R_N-2h+1/N1} respectively. Thus, using Cauchy's residue theorem, letting $T\to\infty$ and noting that the integrals along the horizontal segments approach zero, and invoking \eqref{mainequality}, we see that
\begin{align}\label{cauchyThm2}
\sum_{n=1}^\infty n^{N-2h}\frac{\exp(-a n^Nx)}{1-\exp(-n^Nx)} =R_0+R_1+R_{\frac{N-2h+1}{N}}+J(x, a),
\end{align}  
where $J(x, a)$ is defined in \eqref{Integral}. 
%\begin{align}\label{Integral1}
%J(x, a):=\frac{1}{2\pi i}\int_{(-r)}\Gamma(s)\zeta(s,a)\zeta(Ns-(N-2h))x^{-s} \mathrm{d}s.
%\end{align}
We first prove part (i), that is, when $N$ is an odd positive integer. From \eqref{Integral} to \eqref{xmnstar}, the calculations for evaluating $J(x, a)$ remain exactly the same. Now \eqref{s1s} and the inequality $0<r<1/N$ along with the fact that Re$(s)=1+r$ imply $N-2h+1<c_1:=$Re$(s_1)<N-2h+2$. In order to apply Lemma \ref{MellTransOfCOT}, we need to again shift the line of integration from Re$(s_1)=c_1$ to Re$(s_1)=c_2$, where $0<c_2<2$. In doing so, we encounter poles of the integrand $\frac{\Gamma(s_1)}{\tan(\frac{\pi s_1}{2})}{X_{m,n, j}^*}^{-s_1}$ at $s=2, 4, \cdots, N-2h+1$. Again, the integrals along the horizontal segments approach zero as the height of the contour tends to $\infty$. Now
\begin{equation*}
\lim_{s_1\to 2k}\frac{(s_1-2k)\G(s_1)}{\tan\left(\frac{\pi s_1}{2}\right)}{X_{m,n, j}^*}^{-s_1}=\frac{2}{\pi}\G(2k){X_{m,n, j}^*}^{-2k}
\end{equation*}
along with Lemma \ref{MellTransOfCOT} and \eqref{finalE(X)} imply that
\begin{align*}
E(X_{m,n})=\frac{(-1)^{\frac{N+1}{2}}}{\pi}\sum_{j= -(N-1)}^{N-1}{\vphantom{\sum}}''i^je^{\frac{-ij\pi(2h-1)}{2N}}\left(\int_0^\infty \frac{t\cos t}{{X_{m,n,j}^*}^2+t^2}\ \mathrm{d}t+\sum_{k=1}^{\frac{N-2h+1}{2}}\Gamma(2k){X_{m,n,j}^*}^{-2k}\right)
\end{align*}
so that along with \eqref{j2xinter}, we have
\begin{align}\label{j2xanew}
J_2(x, a)&= \frac{2}{\pi Nx}(-1)^{h+\frac{N+3}{2}}(2\pi)^{2h}\left(\frac{(2\pi)^{N+1}}{x}\right)^{\frac{1-2h}{N}}\sum_{j= -(N-1)}^{N-1}{\vphantom{\sum}}''i^je^{\frac{-ij\pi(2h-1)}{2N}}\nonumber\\
&\quad\times\sum_{n=1}^{\infty}n^{\frac{1-2h}{N}}\sin(2\pi na)\sum_{m=1}^{\infty}\left(\int_{0}^{\infty}\frac{t\cos(t)}{{X_{m,n, j}^*}^{2}+t^2}\, dt+\sum_{k=1}^{\frac{N-2h+1}{2}}\Gamma(2k){X_{m,n,j}^*}^{-2k}\right).
\end{align}
Employ Theorem \ref{raabesum} using \eqref{xmnstar} and \eqref{X_{m,n}} to see that
\begin{align}\label{befsummingn}
&\sum_{m=1}^{\infty}\left(\int_{0}^{\infty}\frac{t\cos(t)}{{X_{m,n, j}^*}^{2}+t^2}\, dt+\sum_{k=1}^{\frac{N-2h+1}{2}}\Gamma(2k){X_{m,n,j}^*}^{-2k}\right)\nonumber\\
&=\frac{1}{2}\left(\log\left(\tfrac{1}{\pi}A_{N,j}\left(\tfrac{n}{x}\right)\right)-\tfrac{1}{2}\left(\psi\left(\tfrac{i}{\pi}A_{N,j}\left(\tfrac{n}{x}\right)\right)+\psi\left(-\tfrac{i}{\pi}A_{N,j}\left(\tfrac{n}{x}\right)\right)\right)+T(N, h, x, j)\right),
\end{align}
where
\begin{equation*}
T(N, h, x, j):=2\displaystyle\sum_{k=1}^{\frac{N-2h+1}{2}}\frac{\G(2k)\zeta(2k)}{\left(2\pi\left(\frac{2\pi n}{x}\right)^{1/N}e^{\frac{i\pi j}{N}}\right)^{2k}}.
\end{equation*}
Now observe that \eqref{psiasymp} implies
\begin{equation*}
\log\left(\tfrac{1}{\pi}A_{N,j}\left(\tfrac{n}{x}\right)\right)-\tfrac{1}{2}\left(\psi\left(\tfrac{i}{\pi}A_{N,j}\left(\tfrac{n}{x}\right)\right)+\psi\left(-\tfrac{i}{\pi}A_{N,j}\left(\tfrac{n}{x}\right)\right)\right)=O_{N, x}\left(n^{-2/N}\right),
\end{equation*}
and since $1\leq k\leq\frac{N-2h+1}{2}$, $T(N, h, x, j)=O_{N, x}\left(n^{-2/N}\right)$.
Thus if we multiply both sides of \eqref{befsummingn} by $n^{\frac{1-2h}{N}}\sin(2\pi na)$ and then sum over $n$, we can write the sum as
\begin{align}\label{subs}
&\sum_{n=1}^{\infty}n^{\frac{1-2h}{N}}\sin(2\pi na)\sum_{m=1}^{\infty}\left(\int_{0}^{\infty}\frac{t\cos(t)}{{X_{m,n, j}^*}^{2}+t^2}\, dt+\sum_{k=1}^{\frac{N-2h+1}{2}}\Gamma(2k){X_{m,n,j}^*}^{-2k}\right)\nonumber\\
&=\frac{1}{2}\sum_{n=1}^{\infty}\frac{\sin(2\pi na)}{n^{\frac{2h-1}{N}}}\left\{\log\left(\tfrac{1}{\pi}A_{N,j}\left(\tfrac{n}{x}\right)\right)-\tfrac{1}{2}\left(\psi\left(\tfrac{i}{\pi}A_{N,j}\left(\tfrac{n}{x}\right)\right)+\psi\left(-\tfrac{i}{\pi}A_{N,j}\left(\tfrac{n}{x}\right)\right)\right)\right\}\nonumber\\
&\quad+\frac{1}{2}\sum_{n=1}^{\infty}\frac{\sin(2\pi na)}{n^{\frac{2h-1}{N}}}T(N, h, x, j),
\end{align}
since the series $\sum_{n=1}^{\infty}\frac{\sin(2\pi na)}{n^{\frac{2h+1}{N}}}$ converges for $0<a\leq 1$ as long as $(2h+1)/N>0$, that is, $h\geq 0$, which is what we have in our hypotheses. (This also explains why we fail to obtain a transformation for our series when $h<0$.)

Now substituting \eqref{subs} in \eqref{j2xanew}, noting that the expression for $J_1(x, a)$ remains exactly as in \eqref{j11}, we deduce along with \eqref{cauchyThm2}, \eqref{R_0}, \eqref{R_1}, \eqref{R_N-2h+1/N1} and \eqref{jj1j2} that for $0<a\leq 1$,
\begin{align}\label{oddh0nby2long}
&\sum_{n=1}^{\infty}n^{N-2h}\frac{\textup{exp}(-an^{N}x)}{1-\textup{exp}(-n^{N}x)}\nonumber\\
&=-\left(a-\tfrac{1}{2}\right)\zeta(-N+2h)+\frac{\zeta(2h)}{x}+\frac{1}{N}\G\left(\tfrac{N-2h+1}{N}\right)\zeta\left(\tfrac{N-2h+1}{N},a\right)x^{-\frac{(N-2h+1)}{N}}+S(x,a)\nonumber\\
&\quad+\frac{(-1)^{h+\frac{N+3}{2}}}{\pi N}\left(\frac{2\pi}{x}\right)^{\frac{N-2h+1}{N}}\sum_{j=-\frac{(N-1)}{2}}^{\frac{(N-1)}{2}}(-1)^{j}\textup{exp}\left(\frac{i\pi(1-2h)j}{N}\right)\sum_{n=1}^{\infty}\frac{\sin(2\pi na)}{n^{\frac{2h-1}{N}}}T(N, h, x, j).
\end{align}
%Note that one can separate the series with $\sin(2\pi n a)$ in its summand into a sum of two series because of \eqref{psiasymp} and the fact that $\sum_{n=1}^{\infty}\sin(2\pi na)/n^{\epsilon}$ converges for $0<a\leq 1$ (since the series with $\sin(2\pi n a)$ in its summand simply vanishes for $a=1$).
Thus \eqref{oddh0nby2long} leads to \eqref{oddh0Nby2} for $a=1$ with $g(N, h, 1)=-\frac{1}{2}\zeta(-N+2h)$. 

When $0<a<1$, in view of \eqref{gnha}, \eqref{oddh0Nby2} and \eqref{oddh0nby2long}, it suffices to show that
\begin{align}\label{specsimp}
&\frac{(-1)^{h+\frac{N+3}{2}}}{\pi N}\left(\frac{2\pi}{x}\right)^{\frac{N-2h+1}{N}}\sum_{j=-\frac{(N-1)}{2}}^{\frac{(N-1)}{2}}(-1)^{j}e^{\frac{i\pi(1-2h)j}{N}}\sum_{n=1}^{\infty}\frac{\sin(2\pi na)}{n^{\frac{2h-1}{N}}}T(N, h, x, j)=\left(a-\tfrac{1}{2}\right)\zeta(-N+2h).
\end{align}
Use Euler's formula \eqref{ef} along with the fact that $\zeta(1-2k)=-B_{2k}/(2k)$, or equivalently, use the functional equation for $\zeta(2k)$ to simplify $T(N, h, x, j)$ as
\begin{equation*}
T(N, h, x, j)=\sum_{k=1}^{\frac{N-2h+1}{2}}\frac{(-1)^{k}\zeta(1-2k)}{\left(\left(\frac{2\pi n}{x}\right)^{1/N}e^{\frac{i\pi j}{N}}\right)^{2k}}.
\end{equation*}
Now substitute the above representation of $T(N, h, x, j)$ in \eqref{specsimp}, separate the term corresponding to $k=\frac{N-2h+1}{2}$ on the left side and invoke \cite[p.~45, Formula \textbf{1.441.1}]{grn} $\sum_{n=1}^{\infty}\frac{\sin(2\pi na)}{n}=-\pi\left(a-\frac{1}{2}\right)$ to see that this term to be equal to $\left(a-\tfrac{1}{2}\right)\zeta(-N+2h)$. Thus we need only show that
\begin{align}\label{overjnk}
&\frac{(-1)^{h+\frac{N+3}{2}}}{\pi N}\left(\frac{2\pi}{x}\right)^{\frac{N-2h+1}{N}}\sum_{k=1}^{\frac{N-2h-1}{2}}\frac{(-1)^{k}\zeta(1-2k)}{\left(\frac{2\pi}{x}\right)^{2k/N}}\sum_{n=1}^{\infty}\frac{\sin(2\pi na)}{n^{\frac{2h+2k-1}{N}}}\sum_{j=-\frac{(N-1)}{2}}^{\frac{(N-1)}{2}}(-1)^{j}e^{\frac{i\pi(1-2h-2k)j}{N}}=0.
\end{align}
In the sum in \eqref{cc}, replace $j$ by $2j$ and then let $z=\frac{\pi}{2N}(2h+2k-1)$ so that
\begin{align}\label{overj}
\sum_{j=-\frac{(N-1)}{2}}^{\frac{(N-1)}{2}}(-1)^{j}e^{\frac{i\pi(1-2h-2k)j}{N}}=\frac{\cos\left(\frac{\pi(2h+2k-1)}{2}\right)}{\cos\left(\frac{\pi(2h+2k-1)}{2N}\right)}.
\end{align}
Now $\cos\left(\frac{\pi(2h+2k-1)}{2}\right)=0$, however, it should also be shown that $\cos\left(\frac{\pi(2h+2k-1)}{2N}\right)\neq 0$. To that end, note that $1\leq k<\frac{N-2h+1}{2}$ implies $\frac{2h+1}{N}\leq\frac{2h+2k-1}{N}<1$. Also, $h\geq 0$ implies $\frac{2h+1}{N}\geq\frac{1}{N}$. Combining, we see that $\frac{1}{N}\leq\frac{2h+2k-1}{N}<1$, so that $\cos\left(\frac{\pi(2h+2k-1)}{2N}\right)\neq 0$ for $N>1$. Thus, the sum over $j$ in \eqref{overj} equals $0$ for $N>1$ which implies \eqref{overjnk}. For $N=1$, note that $h<N/2$ along with $h\geq 0$ implies $h=0$ so that the sum over $k$ in \eqref{overjnk} is empty, and hence \eqref{overjnk} holds again. Hence \eqref{specsimp} holds and therefore \eqref{oddh0Nby2} holds with $g(N, h, a)=0$ for $0<a<1$.

We omit the proof of \eqref{evenh0Nby2} since it is exactly along the lines of the proof of Theorem \ref{dgkmgen}. By a similar argument one can see that \eqref{evenh0Nby2} holds also when $h<0$, unlike the case when $N$ is odd.
\end{proof}
%\begin{corollary}\label{hzero}
%Let $N$ be an odd positive integer. Let $0<a\leq 1$ and $x>0$. Let $g(N, h, a)$ be defined as in \eqref{gnha}. Then,
%\begin{align}
%&\sum_{n=1}^{\infty}n^{N}\frac{\textup{exp}(-an^{N}x)}{1-\textup{exp}(-n^{N}x)}=\frac{1}{N^2}\G\left(\frac{1}{N}\right)\zeta\left(1+\frac{1}{N},a\right)x^{-\left(1+\frac{1}{N}\right)}-\frac{1}{2x}+g(N, 0, a)\nonumber\\
%&\quad-\frac{1}{N}\left(\frac{2\pi}{x}\right)^{1+\frac{1}{N}}\sum_{j=-\frac{(N-1)}{2}}^{\frac{(N-1)}{2}}e^{\frac{i\pi j}{N}}\Bigg[\sum_{n=1}^{\infty}\frac{n^{\frac{1}{N}}\cos(2\pi na)}{\left(\textup{exp}\left(2A_{N,j}\left(\frac{n}{x}\right)\right)-1\right)}+\frac{(-1)^{j+\frac{N+1}{2}}}{\pi}\nonumber\\
%&\quad\times\sum_{n=1}^{\infty}n^{\frac{1}{N}}\sin(2\pi na)\left\{\log\left(\tfrac{1}{\pi}A_{N,j}\left(\tfrac{n}{x}\right)\right)-\tfrac{1}{2}\left(\psi\left(\tfrac{i}{\pi}A_{N,j}\left(\tfrac{n}{x}\right)\right)+\psi\left(-\tfrac{i}{\pi}A_{N,j}\left(\tfrac{n}{x}\right)\right)\right)\right\}\Bigg]
%\end{align}
%\end{corollary}
%\begin{proof}
%Let $h=0$ in Theorem \ref{dgkmgenh0Nby2} and simplify.
%\end{proof}
%%\begin{corollary}\label{hzerosplsim}
%\begin{align}\label{ramge5}
%&\a\sum_{n=1}^{\infty}\frac{ne^{2n\a(1-a)}}{e^{2n\a}-1}+\b\sum_{n=1}^{\infty}\frac{n\cos(2\pi n a)}{e^{2n\b}-1}\nonumber\\
%&=\frac{\psi'(a)}{4\a}-\frac{1}{4}+\frac{\b}{\pi}\sum_{n=1}^{\infty}n\sin(2\pi na)\left\{\log\left(\frac{n\b}{\pi}\right)-\frac{1}{2}\left(\psi\left(\frac{in\b}{\pi}\right)+\psi\left(\frac{-in\b}{\pi}\right)\right)\right\}.
%\end{align}
%\end{corollary}
\begin{proof}[Corollary \textup{\ref{hzerospl}}][]
Let $N=1$ so that $h=0$ in part (i) of Theorem \ref{dgkmgenh0Nby2}. Use the fact \cite[p.~608, Formula \textbf{25.11.12}]{olver-2010a} $\zeta(\ell,a)=\frac{(-1)^{\ell}}{(\ell-1)!}\psi^{(\ell-1)}(a)$, let $x=2\a, \a\b=\pi^2$ and simplify.
\end{proof}
\begin{proof}[Corollary \textup{\ref{befkon}}][]
Let $a=1/2$, $\a=\b=\pi$ in Corollary \ref{hzerospl} and use the fact \cite[p.~144, Formula \textbf{5.15.3}]{olver-2010a} $\psi'(1/2)=\pi^2/2$.
\end{proof}
\begin{proof}[Corollary \textup{\ref{catalan}}][]
Let $a=1/4, \a=\b=\pi$ in Corollary \ref{hzerospl} and use the fact \cite[p.~144, Formula \textbf{5.15.1}]{olver-2010a} $\psi'(1/4)=8G+\pi^2$, where $G$ is the Catalan's constant given by $G=\sum_{n=0}^{\infty}(-1)^{n}(2n+1)^{-2}$.
\end{proof}
%When $0<a<1$, one can further simplify \eqref{ramgennnn}, or equivalently \eqref{ramgennn}, as follows:

\section{A vast generalization of Wigert's formula for $\zeta\left(\frac{1}{N}\right)$}\label{nevenasec}
Except for Theorems \ref{dgkmgen} and \ref{dgkmgenh0Nby2}, we have mostly concentrated on results for an odd positive integer $N$. In this section, we are concerned with the results for $N$ even. We begin with the proof of a two-parameter generalization of Wigert's formula \cite[pp.~8-9, Equation (5)]{wig}, \cite[Equation (1.2)]{dixitmaji1}. 
\begin{proof}[Theorem \textup{\ref{nevena}}][] 
Here $N$ is an even positive integer. We first prove the result for a non-negative integer $m$ using Theorem \ref{dgkmgen}. For $m<0$, it can be proved using Theorem \ref{dgkmgenh0Nby2}. 

Suppose $m$ is a non-negative integer. Then let $h=\frac{N}{2}+Nm$, $x=2^{N}\a$ in Theorem \ref{dgkmgen} and let $\b>0$ be defined by $\a\b^{N}=\pi^{N+1}$. After rearranging some terms, we obtain
\begin{align}\label{neveneqnap}
&\left(a-\frac{1}{2}\right)\zeta(2Nm)+\sum_{j=1}^{m}\frac{B_{2j+1}(a)}{(2j+1)!}\zeta(2N(m-j))(2^N\a)^{2j}+\sum_{n=1}^{\infty}\frac{n^{-2Nm}\textup{exp}\left(-a(2n)^{N}\a\right)}{1-\exp{\left(-(2n)^{N}\a\right)}}\nonumber\\
&=\frac{1}{N} \G\left( \frac{1-2Nm}{N}  \right) \zeta\left( \frac{1-2Nm}{N} , a \right) (2^N \a )^{\frac{2Nm -1}{N}  }+(-1)^{\frac{N}{2}+1 } \frac{1}{N} \Big( \frac{2 \pi }{2^N \a} \Big)^{\frac{1-2Nm}{N} }\nonumber\\
&\quad\times \sum_{ j = - \frac{N}{2} }^{ \frac{N}{2} - 1} e^{\frac{i\pi(2j+1)}{2N}} \exp\left(-\tfrac{i \pi}{2}(2m+1)(2j+1)  \right)\sum_{n=1}^{\infty}\frac{\cos(2\pi na)+i(-1)^{j+\frac{N}{2}+1}\sin(2\pi na)}{n^{2m +1 -\frac{1}{N}}\left(\textup{exp}\left((2n)^{\frac{1}{N}} \b  e^{\frac{i\pi(2j+1)}{2N}} \right)-1\right)}\nonumber\\
&\quad+\frac{(-1)^{\frac{N}{2}+1}}{2}(2\pi)^{N+2Nm}\sum_{j=0}^{m}\left(\frac{-1}{4\pi^2}\right)^{jN}\frac{B_{2j}(a)B_{N+2N(m-j)}}{(2j)!(N+2N(m-j))!}(2^{N}\a)^{2j-1},
\end{align}
where we have used the fact $ 2 A_{N, j + \frac{1}{2}}\left(\frac{n}{x}  \right) = (2n)^{\frac{1}{N}} \b  e^{\frac{i\pi(2j+1)}{2N}}  $.

We now simplify some of the expressions on the right-hand side. Since $\a\b^{N}=\pi^{N+1}$,
\begin{align}
\Big( \frac{\pi}{\a} \Big)^{\frac{1-2Nm}{N}   } & =  \a^{\frac{2Nm-1}{N+1}   }  \b ^{\frac{1-2Nm}{N+1} }, \label{alpha beta}\\
\pi^{N+2Nm-2Nj} \a^{2j-1}&=\a^{ \frac{2j+2Nm-1}{N+1} } \b^{ N+ \frac{2N^2 (m-j) - N }{N+1}}. \label{alpha beta polynomial}
\end{align}
We now split the sum $\sum_{ j = - \frac{N}{2} }^{ \frac{N}{2} - 1}$ as
\begin{equation*}
\sum_{ j = - \frac{N}{2} }^{ \frac{N}{2} - 1}=\sum_{ j = 0 }^{ \frac{N}{2} - 1}+\sum_{ j = - \frac{N}{2} }^{ -1}
\end{equation*}
and replace $j$ by $-1-j$ in the second sum. Then combining the corresponding terms in the resulting two finite sums on the above right-hand side, using the fact that $\exp{\left(-\tfrac{1}{2}\left(i\pi(2j+1)(2m+1)\right)\right)}$\newline
$=i(-1)^{j+m+1}$ and then simplifying, we get
{\allowdisplaybreaks\begin{align}\label{him2}
&\sum_{ j = - \frac{N}{2} }^{ \frac{N}{2} - 1} e^{\frac{i\pi(2j+1)}{2N}} \exp\left(-\tfrac{i \pi}{2}(2m+1)(2j+1)  \right)\sum_{n=1}^{\infty}\frac{\cos(2\pi na)+i(-1)^{j+\frac{N}{2}+1}\sin(2\pi na)}{n^{2m +1 -\frac{1}{N}}\left(\textup{exp}\left((2n)^{\frac{1}{N}} \b  e^{\frac{i\pi(2j+1)}{2N}} \right)-1\right)}\nonumber\\
&=-2(-1)^{m+1}\sum_{j=0}^{\frac{N}{2}-1}(-1)^j\Bigg[\sum_{n=1}^{\infty}\frac{\cos(2\pi na)}{n^{2m+1-\frac{1}{N}}}\textup{Im}\Bigg(\frac{e^{\frac{i\pi(2j+1)}{2N}}}{\exp{\left((2n)^{\frac{1}{N}}\b e^{\frac{i\pi(2j+1)}{2N}}\right)}-1}\Bigg)\nonumber\\
&\qquad\qquad\qquad\qquad\qquad\qquad+(-1)^{j+\frac{N}{2}+1}\sum_{n=1}^{\infty}\frac{\sin(2\pi na)}{n^{2m+1-\frac{1}{N}}}\textup{Re}\Bigg(\frac{e^{\frac{i\pi(2j+1)}{2N}}}{\exp{\left((2n)^{\frac{1}{N}}\b e^{\frac{i\pi(2j+1)}{2N}}\right)}-1}\Bigg)\Bigg].
\end{align}}
Now substituting \eqref{him2} in \eqref{neveneqnap}, using \eqref{alpha beta} and \eqref{alpha beta polynomial}, and then multiplying both sides of the resulting identity by $\a^{ -\left(\frac{2Nm-1}{N+1}\right) }$, we arrive at \eqref{neveneqna}. This completes the proof for $m>0$. For $m<0$, the result follows from part (ii) of Theorem \ref{dgkmgenh0Nby2}. The argument is exactly the same as above and is hence omitted.
\end{proof}
\begin{remark}
For $h\geq N/2$, $N$ even, we considered $h=\frac{N}{2}+Nm, m\in\mathbb{N}\cup\{0\}$ in the proof of Theorem \ref{nevena} given above. However, one can even consider a more general $h$ of the form $h=\frac{N}{2}+Nm+r, 0\leq r<N$ and derive identities analogous to Theorem \ref{nevena}.
\end{remark}
\begin{proof}[Corollary \textup{\ref{transzetaeven}}][]
Let $\a=\b=\pi$ and put $a = \frac{1}{2}$ in Theorem \ref{nevena}. Using \eqref{bjhalf}, \eqref{b2jp1}, and multiplying both sides of the resulting equation by $\pi^{ \frac{2Nm -1 }{N+1} }$, one obtains the following:
\begin{align}\label{N even transcendence}
& \sum_{n=1}^{\infty}\frac{n^{-2Nm}\textup{exp}\left(-\frac{1}{2}(2n)^{N}\pi \right)}{1-\exp{\left(-(2n)^{N}\pi \right)}} = \frac{ 2^{2Nm-1}}{N} \Bigg( (-1)^m \frac{ \zeta\Big(2m +1 - \frac{1}{N}  \Big) \Big( 2^{\frac{1-2Nm}{N} } -1 \Big)}{ 2^{2m +1 -\frac{1}{N }} \cos\Big( \frac{\pi}{2N}\Big) } \nonumber \\ 
&\quad-2(-1)^{\frac{N}{2}+m}2^{\frac{1-2Nm}{N}} \sum_{j=0}^{\frac{N}{2}-1}(-1)^j \sum_{n=1}^{\infty}\frac{(-1)^n }{n^{2m+1-\frac{1}{N}}} \textup{Im}\Bigg(\frac{e^{\frac{i\pi(2j+1)}{2N}}}{\exp{\left((2n)^{\frac{1}{N}}\pi e^{\frac{i\pi(2j+1)}{2N}}\right)}-1}\Bigg) \Bigg) \nonumber \\
&\quad +(-1)^{\frac{N}{2}+1}2^{2Nm-1} \pi^{N(2m+1)-1 }\sum_{j=0}^{m}\frac{(2^{1-2j}-1)B_{2j} B_{(2m+1-2j)N}}{(2j)!((2m+1-2j)N)!}\pi^{2j(1-N)},
\end{align}
where in the course of simplification we used
\begin{align*}
(-1)^m \frac{ \zeta\left(2m +1 - \frac{1}{N}  \right) \left( 2^{\frac{1-2Nm}{N} } -1 \right)}{ 2^{2m +1 -\frac{1}{N }} \cos\left( \frac{\pi}{2N}\right) } = \pi^{\frac{2Nm -1}{N}} \G\left( \frac{1-2Nm}{N}  \right) \zeta\left( \frac{1-2Nm}{N} , \frac{1}{2} \right),
\end{align*}
which follows from \eqref{zetafe1} with $s = 2m +1 - \frac{1}{N}$ and \eqref{hzetahalf}.

One can now easily check that $\cos\left( \frac{\pi}{2N} \right)$ is always an algebraic number for every $N \in \mathbb{N}$. Also the last term on the right-hand side of \eqref{N even transcendence} is always a non-zero polynomial of $\pi$ with rational coefficients. Therefore it is a transcendental number, which implies our corollary.  
\end{proof}

Even though Theorem \ref{klusch} can be proved using Theorem \ref{nevena}, we prefer to give the proof using Theorem \ref{dgkmgen} because the conditions on $\a$ and $\b$ in the former two theorems are different. We begin with an analogue of Lemma \ref{ktywigl} to be used along with the latter in the proof of Theorem \ref{klusch}. 

%An analogue of the above lemma will be used in the proof of Theorem \ref{klusch} and is as follows.
\begin{lemma}\label{ktywigl2}
For $a, u, v \in \mathbb{R}$, we have
\begin{align*}
2 \,\textup{Im}\left(\frac{e^{iuv}}{\exp{\left(ae^{-iu}\right)}-1}\right)= \frac{\sin(a\sin(u)+uv)-e^{-a\cos(u)}\sin(uv)}{\cosh(a\cos(u))-\cos(a\sin(u))}.
\end{align*}
\end{lemma}
We omit the proof since it can be proved along similar lines as the proof of Lemma \ref{ktywigl} given in \cite{dixitmaji1}.
\begin{proof}[Theorem \textup{\ref{klusch}}][]
Let $h=\frac{N}{2}$ in Theorem \ref{dgkmgen}. This gives
\begin{align*}
& \sum_{n=1}^{\infty} \frac{\exp(-a n^{N} x)}{ 1 - \exp(-n^N x )}  = \frac{1}{2}\Big( a - \frac{1}{2} \Big) + \frac{\zeta(N)}{x} + \frac{1}{N} \Gamma\Big( \frac{1}{N}  \Big) \zeta\Big( \frac{1}{N}, a  \Big) x^{-\frac{1}{N} }  \nonumber \\
& + \frac{(-1)^{\frac{N}{2} + 1 }}{N} \left( \frac{2 \pi }{x} \right)^{\frac{1}{N} } \sum_{j=-\frac{N}{2}}^{\frac{N}{2}-1}e^{\frac{i\pi(1-2h)\left(j+\frac{1}{2}\right)}{N}}
\sum_{n=1}^{\infty}\frac{\cos(2\pi na)+i(-1)^{j+\frac{N}{2}+1}\sin(2\pi na)}{n^{1-\frac{1}{N}}\left(\textup{exp}\left(2A_{N,j+\frac{1}{2}}\left(\frac{n}{x}\right)\right)-1\right)}.
\end{align*}
Wigert's formula \cite[pp.~8-9, Equation (5)]{wig} (see also \cite[Equation (1.2)]{dixitmaji1}) is a special case of the above formula when $a=1$.

Now let $N=2$ and simplify so as to obtain
\begin{align}\label{invisible}
\sum_{n=1}^{\infty} \frac{\exp(-a n^{2} x)}{ 1 - \exp(-n^2 x )}
%&= \frac{1}{2}\Big( a - \frac{1}{2} \Big) + \frac{\pi^2}{6 x} + \frac{1}{2} \sqrt{\frac{\pi}{x}} \zeta\Big( \frac{1}{2}, a \Big) + \frac{1}{2} \sqrt{\frac{2 \pi }{x}} \nonumber \\
%& \times \Bigg[ \exp\Big(\frac{i \pi}{4}  \Big) \left(  \sum_{n=1}^{\infty} \frac{\cos(2 \pi n a )}{\sqrt{n} \Big(\exp\Big( ( 2\pi)^{\frac{3}{2} } \sqrt{\frac{n}{x}} e^{-\frac{i \pi}{4} }  \Big) - 1 \Big)} -i  \sum_{n=1}^{\infty} \frac{\sin(2 \pi n a )}{\sqrt{n} \Big(\exp\Big( ( 2\pi)^{\frac{3}{2} } \sqrt{\frac{n}{x}} e^{\frac{-i \pi}{4}}  \Big) - 1 \Big)} \right) \nonumber \\
%& + \exp\Big(\frac{- i \pi}{4}  \Big)  \left(  \sum_{n=1}^{\infty} \frac{\cos(2 \pi n a )}{\sqrt{n} \Big(\exp\Big( ( 2\pi)^{\frac{3}{2} } \sqrt{\frac{n}{x}} e^{\frac{i \pi}{4} }  \Big) - 1 \Big)}  + i  \sum_{n=1}^{\infty} \frac{\sin(2 \pi n a )}{\sqrt{n} \Big(\exp\Big( ( 2\pi)^{\frac{3}{2} } \sqrt{\frac{n}{x}} e^{\frac{i \pi}{4}}  \Big) - 1 \Big)} \right) \Bigg] \nonumber \\
&= \frac{1}{2}\left( a - \frac{1}{2} \right) + \frac{\pi^2}{6 x} + \frac{1}{2} \sqrt{\frac{\pi}{x}} \zeta\left( \frac{1}{2}, a \right)\nonumber\\
&\quad+ \sqrt{\frac{2 \pi }{x}}\Bigg[ \sum_{n=1}^{\infty} \frac{\cos(2 \pi n a )}{\sqrt{n}} \mathrm{Re} \Bigg( \frac{\exp(i\pi/4)}{\exp\Big( ( 2\pi)^{\frac{3}{2} } \sqrt{\frac{n}{x}} e^{-\frac{i \pi}{4} }  \Big) - 1 } \Bigg)\nonumber\\
&\qquad\qquad\quad+ \sum_{n=1}^{\infty} \frac{\sin(2 \pi n a )}{\sqrt{n}} \mathrm{Im} \Bigg( \frac{\exp(i\pi/4)}{ \exp\Big( ( 2\pi)^{\frac{3}{2} } \sqrt{\frac{n}{x}} e^{-\frac{i \pi}{4} }  \Big) - 1 } \Bigg) \Bigg]. 
\end{align}
Letting $a=1$ and using Lemma \ref{ktywigl}, we obtain a formula of Ramanujan \cite{ramnote}, \cite[p.~314]{bcbramsecnote}, \cite[p.~332]{lnb}:
\begin{align*}
\sum_{n=1}^{\infty} \frac{1}{\exp( n^2 x) -1 }  &= \frac{1}{4} + \frac{\pi^2}{6 x} + \frac{1}{2} \sqrt{\frac{\pi}{x}} \zeta\left( \frac{1}{2} \right) \nonumber\\
&\quad+  \sqrt{\frac{ \pi }{2x}}\sum_{n =1}^{\infty} \frac{1}{\sqrt{n}} \Bigg(\frac{ \cos\left(2 \pi^{\frac{3}{2}}\sqrt{\frac{n}{x}}+\frac{\pi}{4}  \right)-e^{-2\pi^{\frac{3}{2}}\sqrt{\frac{n}{x}}}\cos\left(\frac{\pi}{4}\right)   }{ \cosh\left(2 \pi^{\frac{3}{2}} \sqrt{\frac{n}{x}}  \right)-  \cos\left(2 \pi^{\frac{3}{2}}\sqrt{\frac{n}{x}} \right)  }  \Bigg).
\end{align*}
For $0<a<1$, one is able to further simplify \eqref{invisible}. To that end, keeping in mind that \eqref{hformula} holds also for Re$(s)<1$ in this case, we let $s=1/2$ in it to obtain
\begin{equation}\label{hformulahalf}
\zeta\left(\frac{1}{2},a\right)=\sum_{n=1}^{\infty}\frac{\cos(2\pi na)}{\sqrt{n}}+\sum_{n=1}^{\infty}\frac{\sin(2\pi na)}{\sqrt{n}}.
\end{equation}
Invoking \eqref{hformulahalf} in \eqref{invisible} and then simplifying using Lemmas \ref{ktywigl} and \ref{ktywigl2} leads to
\begin{align*}
 \sum_{n=1}^{\infty} \frac{\exp(-a n^{2} x)}{ 1 - \exp(-n^2 x )}  &= \frac{1}{2}\Big( a - \frac{1}{2} \Big) + \frac{\pi^2}{6 x}\nonumber \\
 &\quad+\frac{1}{2}\sqrt{\frac{\pi }{x}}\Bigg[ \sum_{n=1}^{\infty} \frac{\cos( 2\pi n a) }{\sqrt{n}} \Bigg(\frac{ \sinh\Big(2 \pi^{\frac{3}{2}}\sqrt{\frac{n}{x}}  \Big)-\sin\Big(2 \pi^{\frac{3}{2}}\sqrt{\frac{n}{x}}  \Big)   }{ \cosh\Big(2 \pi^{\frac{3}{2}} \sqrt{\frac{n}{x}}  \Big)-  \cos\Big(2 \pi^{\frac{3}{2}}\sqrt{\frac{n}{x}} \Big)  } \Bigg) \\
& \qquad\qquad\quad + \sum_{n=1}^{\infty} \frac{\sin( 2\pi n a) }{\sqrt{n}} \Bigg(\frac{ \sinh\Big(2 \pi^{\frac{3}{2}}\sqrt{\frac{n}{x}}  \Big) + \sin\Big(2 \pi^{\frac{3}{2}}\sqrt{\frac{n}{x}}  \Big)   }{ \cosh\Big(2 \pi^{\frac{3}{2}} \sqrt{\frac{n}{x}}  \Big)-  \cos\Big(2 \pi^{\frac{3}{2}}\sqrt{\frac{n}{x}} \Big)  } \Bigg)\Bigg].
 \end{align*}
Finally, let $x=\a$ and let $\b=4\pi^3/\a$ to arrive at \eqref{kluscheqn}.

Now \eqref{kluscheqnhalf} simply follows by letting $a=1/2$ in \eqref{kluscheqn}.
 \end{proof}

\section{Concluding remarks}\label{con}
First of all, we would like to mention that all of our results involving $x$, or $\a$ and $\b$, can be extended by analytic continuation to complex values of $x, \a$, and $\b$ such that Re$(x)>0$, Re$(\a)>0$ and Re$(\b)>0$. 

It is clear from \cite{dixitmaji1} as well as from the above work that in order to understand the arithmetical nature of Euler's constant, the values of the Riemann zeta function at odd positive integers as well as at rational arguments, further study of the generalized Lambert series considered here is absolutely essential. We refer the reader to two recent papers \cite{masser} and \cite{boxalljones} for some quantitative results on rational values of the Riemann zeta function.

When $N\in2\mathbb{N}$, we were able to transform the series $\sum_{n=1}^{\infty}n^{N-2h}\frac{\textup{exp}(-an^{N}x)}{1-\textup{exp}(-n^{N}x)}$, $0<a\leq 1$, for \emph{any} integer $h$. However, for $N$ odd and positive, we could do so only for $h\geq 0$. Thus it remains to be seen if there exists a transformation of this series when $N$ is odd and $h<0$. If done, this might give us a complete generalization of \eqref{zetageneqn}, that is of \cite[Theorem 1.2]{dixitmaji1}. 

As mentioned in Remark \ref{bc} after Theorem \ref{dgkmgen}, Kanemitsu, Tanigawa and Yoshimoto \cite[Theorem 2.1]{ktyacta} obtained a formula for Hurwitz zeta function at rational arguments, that is, $\zeta\left(\frac{b}{c},a\right)$, where $b$ is a negative odd integer and $c$ is a positive even integer. The only other case which remains to be seen is when $b$ and $c$ are both odd since the the case when they are both even can be reduced to one of the three cases. 

Let $\chi$ denote the primitive Dirichlet character modulo $q$. Using the identity \cite[p.~71, Equation (16)]{dav} $L(s, \chi)=q^{-s}\sum_{n=1}^{q}\chi(n)\zeta(s,n/q)$ and Theorem \ref{dgkmgen}, we can obtain a representation for $L\left(\frac{N-2h+1}{N}, \chi\right)$. Similarly, working with Theorems \ref{dgkmgen}, \ref{dgkmord2} and the identity obtained by differentiating the above identity with respect to $s$, one can obtain a representation for $L'\left(-2j, \chi\right), j\in\mathbb{N}$. These representations may be useful in computing these quantities numerically. 
%For $N$ even and $h\geq N/2$, this has been done by Kanemitsu, Tanigawa and Yoshimoto \cite{ktyacta}.
%One can similarly derive a representation for the multiple Hurwitz zeta function $\zeta_r(s, a)$ using the identity \cite[]{}

In \cite{dixitmaji1}, it was shown that any two odd zeta values of the form $\zeta(4k+3)$ are related to each other by means of the relation that each such odd zeta value obeys with $\zeta(3)$ as governed by the case $a=1$ of Theorem \ref{ggram}, that is, \eqref{zetageneqn}. Also, while it was shown that such a relation is not possible for \emph{every} pair of the form $\left(\zeta(4k_1+1),  \zeta(4k_2+1)\right)$, through \eqref{zetageneqn}, it does exist for \emph{some} such pairs. However, \eqref{zetageneqn} has a limitation in that no two odd zeta values, one of which is of the form $\zeta(4k_1+1)$ and another $\zeta(4k_2+3)$, are related through it. This is partially overcome through our generalization of \eqref{zetageneqn}, that is, Theorem \ref{ggram}, in that now it is possible to have a relation between two odd zeta values, one of the form $\zeta(4k+3)$ and another of the form $\zeta(8k+5)$. This prompts to ask if there exists a transformation which would relate two odd zeta values, one of which is of the form $\zeta(4k+3)$ and another of the form $\zeta(8k+1)$.

\begin{center}
\begin{table}[h]
\caption{Left and right-hand sides of Theorem \ref{dgkmgen} (with series truncated up to the first $10^5$ terms) }
\label{t1}
\renewcommand{\arraystretch}{1}
{\small
\begin{tabular}{|l|l|l|l|l|l|}
\hline
$N$ & $h$ &$a$ & $x$ & Left-hand side & Right-hand side \\
\hline
$2$ & $2$ & $1$ & $1.2345$ & $ 0.412204713295378 $  &  $ 0.412204713295378 $\\
\hline
$3$ & $3$ & $ \frac{1}{2}$ &  $2.3565$ &  $0.340045844295895 $ & $0.340045844295895 $ \\
\hline
$4$ & $5$ & $\frac{1}{3}$ & $\pi$ & $0.366769348622027$ & $ 0.366769348624188 $ \\
\hline
$5$ & $9$ & $ \frac{2}{7}$ & $\sqrt{2}$ & $0.882042733561192$ & $0.882042733560249$ \\
\hline
$6$ & $4$ & $\frac{1}{\sqrt{2}}$ & $2^{\sqrt{3}}$ &  $0.099037277331145 $ & $0.099037277329 +  \,2.5998536522\times 10^{-19} i $ \\
\hline
$7$ & $5$ & $\frac{3}{5}$ & $1+\sqrt{5}$ & $0.149340139836146 $ & $ 0.149340139821542 $ \\
\hline
\end{tabular}}

\end{table}
\end{center}

\begin{center}
\begin{table}[h]
\caption{Left and right-hand sides of Theorem \ref{dgkmord2} (with series truncated up to the first $10^5$ terms) }
\label{t2}
\renewcommand{\arraystretch}{1}
{\small
\begin{tabular}{|l|l|l|l|l|l|}
\hline
$N$ & $h$ & $a$ & $x$ & Left-hand side & Right-hand side \\
\hline
$1$ & $2$ & $\frac{1}{10}$ & $3.317$ & $ 0.8297473488759233 $ & $ 0.8297473488759262 $ \\
\hline 
$3$ & $5$ & $\frac{1}{1+\sqrt{3}} $ & $\sqrt{5}$ & $0.4939094866586267 $ & $0.4939094866586265 - 2.1455813429 \times 10^{-22} i$  \\
\hline
$5$ & $8$ & $\frac{2}{9} $ & $\sqrt{2} + \sqrt{3}  $& $ 0.5193356374630188 $ & $0.5193356374630185 - 1.6093950728 \times 10^{-18} i $ \\
\hline
$7$ &  $11$& $\sqrt{2}-1$ & $\pi + 0.1234$  & $0.2688885270333226   $ & $ 0.2688885270333224$ \\
\hline
  $9$  &  $ 14 $   & $\frac{1}{4}$  & $ \pi^{\sqrt{2}} $        &  $ 0.2849538075110331 $  & $ 0.2849538075110331 - 4.7986747033 \times 10^{-18} i $      \\
 \hline 
\end{tabular}}

\end{table}
\end{center}

\begin{center}
\begin{table}[h]
\caption{ Odd zeta values related through Theorem \ref{ggram} }
\label{t4}
\renewcommand{\arraystretch}{1}
{\small
\begin{tabular}{|l|l|l|l|}
\hline
$N$ & $m$ & $\zeta(2Nm + 1 -2j N), \quad 0 \leq j \leq m-1  $ \\
\hline

 $1$   &    $5$      & $\zeta(3), \zeta(5), \zeta(7), \zeta(9), \zeta(11)  $                              \\
    \hline 
    $1$ &  $100$  & $\zeta(3), \zeta(5), \zeta(7),\zeta(9),\ldots ,\zeta(201)$ \\
    \hline
  $3$   &    $7$    &    $ \zeta(7), \zeta(13),\zeta(19), \zeta(25),\ldots, \zeta(43) $                            \\
    \hline 
 $5$    &   $15$     &       $ \zeta(11)                             , \zeta(21),\zeta(31),\zeta(41), \ldots, \zeta(151) $                          \\
    \hline 
  $7$   &    $50$     &    $  \zeta(15),\zeta(29),\zeta(43), \zeta(57),\ldots, \zeta(701)$                        \\
    \hline 
\end{tabular}}

\end{table}
\end{center}

\begin{center}
\begin{table}[h]
\caption{Left and right-hand sides of Theorem \ref{dgkmord2m0} (with series truncated up to the first $10^5$ terms) }
\label{t3}
\renewcommand{\arraystretch}{1}
{\small
\begin{tabular}{|l|l|l|l|l|l|}
\hline
$N$  & $a$ & $x$ & Left-hand side & Right-hand side \\
\hline

  $1$    & $\frac{5}{6} $   & $3.987$    &         $0.03741091204936647 $        &               $0.03741091204936687$ \\
    \hline 
$ 3 $      &  $1$  & $\pi + \sqrt{2}  $    &  $0.01061757521903389$     &   $ 0.01061757521903386 + 4.23476435512 \times 10^{-20} i $     \\
    \hline 
 $ 5 $     & $\frac{7}{11} $   &  $23.317$   &                $ 3.596656780667 \times 10^{-7}$ &       $ 3.596662150884 \times 10^{-7}$         \\
    \hline 
   $7$    & $\frac{1}{\sqrt{7}}  $   & $e$     &  $ 0.3832192774947001 $    &  $ 0.3832192773449392 -    2.865123974\times 10^{-19} i $ \\
    \hline 
   $9$      &  $\frac{3}{11} $  &  $1.2852 $  &                $ 0.9736312065003231 $ &      $  0.973631195916481-3.083952846\times 10^{-18} i  $      \\
    \hline 
    $11$    &  $\sqrt{3} -1  $  & $10.2854  $    &       0.000537059726103    &   0.0005369572144785  \\
    \hline
\end{tabular}}

\end{table}
\end{center}

%\vspace{5cm}
\section*{Acknowledgements}
The authors thank Bruce C. Berndt, Ram Murty, Yuri V. Nesterenko, Purusottam Rath, Michel Waldschmidt and Wadim Zudilin for interesting discussions. The first author's research is supported by the SERB-DST grant RES/SERB/MA/P0213/1617/0021 whereas the third author is a SERB National Post Doctoral Fellow (NPDF) supported by the fellowship PDF/2017/000370. Both sincerely thank SERB-DST for the support.

%\begin{center}
%\textbf{Acknowledgements}
%\end{center}

\end{document}